\documentclass[11pt,reqno]{amsart}
\usepackage[T1]{fontenc}
\usepackage{amsmath}
\usepackage[margin=25mm]{geometry}
\usepackage{amsthm}
\usepackage{amsfonts}
\usepackage{amssymb}
\usepackage[colorlinks=true]{hyperref}

\newtheorem{theorem}{Theorem}[section]
\newtheorem{lemma}[theorem]{Lemma}
\newtheorem{proposition}[theorem]{Proposition}
\newtheorem{corollary}[theorem]{Corollary}

\theoremstyle{remark}
\newtheorem{remark}[theorem]{\it \bf{Remark}\/}
\numberwithin{equation}{section}

\catcode`@=11
\def\section{\@startsection{section}{1}%
  \z@{1.5\linespacing\@plus\linespacing}{.5\linespacing}%
  {\normalfont\bfseries\large\centering}}
\catcode`@=12

\newcommand{\R}{\mathbb{R}}
\newcommand{\Z}{\mathbb{Z}}
\newcommand{\C}{\mathbb{C}}
\newcommand{\T}{\mathbb{T}}
\newcommand{\Disc}{\mathbb{D}}
\renewcommand{\leq}{\leqslant}
\renewcommand{\geq}{\geqslant}
\newcommand{\eps}{\varepsilon}
\newcommand{\pa}{\partial}
\newcommand{\la}{\langle}
\newcommand{\ra}{\rangle}
\newcommand{\W}{\mathcal W}
\newcommand{\FL}{\mathcal{F}L^{2,\infty}}
\newcommand{\Hlog}{H^{\frac32,\sigma}_{\log}}
\newcommand{\D}{\mathcal D}
\newcommand{\Kt}{\mathcal K_t}
\def\Re{{\rm Re}\,}
\def\Im{{\rm Im}\,}

\begin{document}

\title[Infinite cascades for the cubic half-wave equation]{Infinite cascades for the defocusing cubic half-wave equation on the torus}
\author{Xi Chen}
\address{Xi Chen, Department of Mathematics and Computer Science, University of Basel, Spiegelgasse 1, 4051 Basel, Switzerland}
\email{{\tt xi01.chen@unibas.ch}}

\subjclass[2020]{35B40, 35Q55, 37K55}

\begin{abstract}
We construct global solutions of the defocusing cubic half-wave equation
on the torus whose $H^s$ norms grow exponentially for every
$1/2<s<3/2$. Such solutions exist at every positive mass.  To the best of our knowledge, this
provides the first example of a genuine infinite cascade for a defocusing
fractional nonlinear Schr\"odinger equation on the torus. In particular, our result resolves an open question posed by Gérard and Grellier \cite{GeGr2016,GeGr2017} concerning the existence of unbounded Sobolev trajectories for the defocusing cubic half-wave equation on the circle.
\end{abstract}

\keywords{Half-wave equation, growth of Sobolev norms, cubic Szeg\H{o} equation, Blaschke products, backward integration, defect measures}

\maketitle
\setcounter{tocdepth}{1}
\tableofcontents

\section{Introduction}


\subsection{The defocusing cubic half-wave equation on the circle}


In this paper we consider the defocusing cubic half-wave equation on the circle $\T:=\R/2\pi\Z$:
\begin{equation}
\label{halfwave}
i\pa_tu=|D|u+|u|^2u,\qquad (t,x)\in\R\times\T,\qquad u(t,x)\in\C,
\end{equation}
where $D:=-i\pa_x$ and $|D|$ is the Fourier multiplier with symbol $|n|$, $n\in\Z$. Evolution equations with a nonlocal fractional dispersion appear as continuum limits of lattice systems with long-range interactions \cite{KiLeSt2013}, in one-dimensional models of wave turbulence \cite{MaMc1997, CaMa2001, ZGPD2001}, and, with the relativistic dispersion $(1-\Delta)^{\frac12}$ on $\R^3$, in the mean-field description of boson stars \cite{ElSc2007, FrLe2007}. Equation \eqref{halfwave} is the case $\alpha=1$ of the defocusing cubic fractional nonlinear Schr\"odinger equations (fractional NLS)
\begin{equation}
\label{fractional}
i\pa_tu=|D|^\alpha u+|u|^2u,\qquad \alpha>0,
\end{equation}
which are the one-dimensional models of dispersive wave turbulence of Majda, McLaughlin and Tabak \cite{MaMc1997}, with a defocusing nonlinearity which contains no fractional derivative.

All Lebesgue norms on $\T$ are taken with respect to the normalized measure $\frac{dx}{2\pi}$, and the Fourier coefficients are $\hat f(n):=\int_\T f(x)e^{-inx}\frac{dx}{2\pi}$. We write $\la n\ra:=(1+n^2)^{\frac12}$, $\|f\|_{H^s}^2:=\sum_n\la n\ra^{2s}|\hat f(n)|^2$, and $\|f\|_{\dot H^s}^2:=\sum_{n\neq0}|n|^{2s}|\hat f(n)|^2$ for the homogeneous norm, which ignores the constant mode. We write $a\lesssim b$ (or $b\gtrsim a$) if $a\leq Cb$ for a constant $C>0$ independent of the time variable, $a\sim b$ if $a\lesssim b\lesssim a$, and $a=O(b)$ if $|a|\lesssim b$; a subscript, as in $\lesssim_s$, emphasizes that $C$ depends on the index $s$. We also write $\la t\ra:=(1+t^2)^{\frac12}$ for $t\in\R$. The Hamiltonian equation \eqref{halfwave} has three conservation laws, the mass, the momentum and the energy:
\begin{equation}
\label{conservationlaws}
\begin{gathered}
M(u):=\sum_{n\in\Z}|\hat u(n)|^2,\qquad P(u):=\sum_{n\in\Z}n|\hat u(n)|^2,\\
E(u):=\frac12\sum_{n\in\Z}|n||\hat u(n)|^2+\frac14\int_\T|u|^4\frac{dx}{2\pi}.
\end{gathered}
\end{equation}
The Cauchy problem for \eqref{halfwave} is globally well-posed in $H^{\frac12}(\T)$, and the $H^s$ regularity, $s>\frac12$, is propagated by the flow \cite[Proposition 1]{GeGr2012}. Since both terms of the energy are nonnegative, the conservation laws control the energy norm uniformly in time:
\begin{equation}
\label{boundH12}
\|u(t)\|_{H^{\frac12}}^2\leq\sum_{n\in\Z}(1+|n|)|\hat u(t,n)|^2\leq M(u)+2E(u),\qquad t\in\R.
\end{equation}
No conservation law controls the norms above $H^{\frac12}$. The growth of $\|u(t)\|_{H^s}$, $s>\frac12$, measures the transfer of the conserved energy from low to high frequencies, and this paper is devoted to the construction of solutions for which this transfer is unbounded.


\subsection{Growth of high Sobolev norms}


The transfer of energy to high frequencies, or forward energy cascade, is one of the central phenomena in the long-time dynamics of Hamiltonian partial differential equations on compact domains. It is the deterministic counterpart of the theory of wave (weak) turbulence, a statistical theory of random nonlinear waves far from equilibrium; see the discussion in \cite{St2010, Ha2016, St2026}. For the defocusing cubic nonlinear Schr\"odinger equation (NLS) on $\T^2$,
\begin{equation}
\label{nlst2}
i\pa_tu+\Delta u=|u|^2u,\qquad (t,x)\in\R\times\T^2,
\end{equation}
Bourgain asked whether there exists a solution with $u(0)\in H^s(\T^2)$, $s>1$, such that
$$\limsup_{t\to+\infty}\|u(t)\|_{H^s}=+\infty,$$
and how fast such a growth can be \cite[p. 43]{Bo2000}. The question is formulated in Sobolev spaces: the datum is only required to belong to $H^s$, and one single datum has to generate an unbounded orbit.  The original question does not impose
$C^\infty$ regularity on either the datum or the solution. We shall call a global solution $u\in\mathcal C([0,+\infty),H^s)$ of a Hamiltonian equation an {\it infinite cascade} if $\limsup_{t\to+\infty}\|u(t)\|_{H^s}=+\infty$, where the index $s$ is strictly above the regularity controlled by the conservation laws. The search for such orbits is a recurring problem in Hamiltonian
PDE on compact domains. Staffilani's 2010 lecture presents forward
cascades as a conjectural form of weak turbulence for periodic
dispersive equations~\cite[pp.~32--33]{St2010}.
Hani~\cite[Section~2.1]{Ha2016} formulates the search for
unbounded Sobolev orbits explicitly and mentions that the question
extends to other dispersive and wave equations.

Polynomial upper bounds on the growth were proved by Bourgain \cite{Bo1996}, Staffilani \cite{St1997}, Sohinger \cite{Soh1, Soh2}, and Colliander, Kwon and Oh \cite{CKO}. Examples of unbounded growth were first obtained by Bourgain for modified equations designed for this purpose \cite{Bo1995, Bo1996}, and Kuksin proved growth estimates for small dispersion NLS \cite{Kuksin}. For \eqref{nlst2}, Colliander, Keel, Staffilani, Takaoka and Tao constructed solutions whose $H^s$ norm, $s>1$, grows from an arbitrarily small to an arbitrarily large value in finite time \cite{CKSTT2010}, and Guardia and Kaloshin estimated the time needed for this growth \cite{GK2015}; see also \cite{Guardia2014, HP, GHP, GHHMP2023}. In these results the datum depends on the amount of growth, and the behavior of the solution after the growth time is not controlled. Infinite cascades were obtained by Hani for a family of cubic Hamiltonian equations on $\T^2$ containing the resonant cubic NLS \cite{H2014}, and by Hani, Pausader, Tzvetkov and Visciglia for the cubic NLS on the product spaces $\R\times\T^d$, $2\leq d\leq4$, where the dispersion in the real direction allows one to follow the resonant dynamics for all times \cite{HPTV}. The infinite cascade question for \eqref{nlst2} remains open, and the main difficulty is the control of a turbulent solution on an infinite time interval \cite[Section 2]{St2026}.

For the fractional NLS \eqref{fractional}, the conservation laws control the $H^{\frac\alpha2}$ norm, and the same question can be asked above this level. G\'erard asked whether the Sobolev norm of a solution to the defocusing fractional NLS \eqref{fractional} can be unbounded \cite[pp. 20]{Ge2017}. For the defocusing cubic half-wave equation on the torus \eqref{halfwave}, The existence of unbounded Sobolev trajectories was explicitly raised as an open problem by Gérard and Grellier in \cite{GeGr2016} and in \cite[Remark 1]{GeGr2017}.  For smooth solutions of \eqref{halfwave}, Thirouin proved that $\|u(t)\|_{H^k}$, for every integer $k\geq1$, is bounded by a constant depending on $k$ and $\|u(0)\|_{H^k}$, times $e^{Ct^2}$, where $C$ depends on $k$ and $\|u(0)\|_{H^{\frac12}}$ \cite{Th2017}. Recently, Maspero and Murgante proved arbitrarily large growth on finite time intervals, from small data in high regularity, for a fractional quasilinear Schr\"odinger equation on $\T$ with sublinear dispersion \cite{MaMu2026}. On the line, we refer to \cite{IoPu2014} for the long-time behavior of fractional NLS, and to \cite{BGLV2019} for the half-wave equation, which is not dispersive and for which small data scattering fails. More recently, the author constructed infinite cascades for the defocusing cubic half-wave equation on the line \cite{Ch2026}.


\subsection{The Szeg\H{o} program}


Let $\Pi_+$ be the Szeg\H{o} projector onto the nonnegative Fourier modes, $\widehat{\Pi_+f}(n):={\bf 1}_{n\geq0}\hat f(n)$, let $\Pi_-:=I-\Pi_+$, and let $L^2_+:=\Pi_+L^2(\T)$ be the Hardy space. The half-wave equation \eqref{halfwave} splits into
$$i(\pa_t+\pa_x)\Pi_+u=\Pi_+(|u|^2u),\qquad i(\pa_t-\pa_x)\Pi_-u=\Pi_-(|u|^2u).$$
For data $u(0)\in L^2_+$ of small size $\eta$ in a sufficiently regular Sobolev space, the negative frequencies are only created by the nonlinearity, and they are transported in the direction opposite to the positive frequencies. G\'erard and Grellier \cite{GeGr2012} proved that on the circle, for times of order $\eta^{-2}\log\frac1\eta$, the solution is approximated by the solution of the cubic Szeg\H{o} equation
\begin{equation}
\label{szego}
i\pa_tw=\Pi_+(|w|^2w),\qquad w=\Pi_+w,
\end{equation}
after the change of frame $x\mapsto x-t$; Pocovnicu obtained first-order approximations on the line and second-order approximations on the circle \cite{Po2013}. The cubic Szeg\H{o} equation is the totally resonant model of \eqref{halfwave}. G\'erard and Grellier discovered that it is completely integrable, with a Lax pair given by Hankel operators \cite{GeGr2010, GeGrbis, GeGr2015, GeGr2017}; see also the survey \cite{GeGr2019}. On the circle, rational data generate quasi-periodic trajectories which are bounded in every $H^s$, $s>\frac12$ \cite{GeGrbis, GeGr2015}, while a dense $G_\delta$ set of smooth data generates trajectories whose $H^s$ norms are unbounded along a sequence of times and return to their initial values along another sequence \cite{GeGr2017}. On the line, Pocovnicu found explicit solutions with $\|w(t)\|_{H^s}\gtrsim|t|^{2s-1}$ \cite{Po, Po2011}, and G\'erard and Pushnitski proved that unbounded $H^1$ orbits are generic in $H^1_+(\R)$ \cite{GePu2025}. Combined with the approximation theorems, these results give arbitrarily large growth on long, but finite, time intervals for the half-wave equation \cite{GeGr2012, Po2013}.

Several modifications of the Szeg\H{o} equation exhibit sustained growth along explicit trajectories. Xu proved that the linearly perturbed equation
\begin{equation}
\label{alphaszego}
i\pa_tw=\Pi_+(|w|^2w)+\varkappa\hat w(0),\qquad \varkappa>0,
\end{equation}
has rational trajectories with exponential growth of all Sobolev norms above $H^{\frac12}$ \cite{Xu2014}, and that composition with finite Blaschke products produces trajectories with several concentration points \cite[Section 6]{Xu2015}. Thirouin obtained exponential growth, with a matching upper bound, for a quadratic Szeg\H{o} equation \cite{Th2019}; Biasi and Evnin constructed explicit cascades for a truncation of \eqref{szego} \cite{BiEv2022}, see also \cite{BiBiEv2019}; and G\'erard and Grellier, and G\'erard, Grellier and He, obtained unbounded trajectories for damped Szeg\H{o} equations \cite{GeGr2020, GeGrHe2022}. In anisotropic settings, the Szeg\H{o} dynamics can be transferred to a dispersive equation for all times thanks to an additional dispersive direction: Xu constructed solutions with unbounded $L^2_xH^s_y$ norms for $(i\pa_t+\pa_x^2-|D_y|)u=|u|^2u$ on $\R_x\times\T_y$ \cite{Xu2017}, and the author obtained logarithmic cascades for the same equation on $\R^2$ \cite{Ch2024}.

On the line, the focusing half-wave equation is $L^2$-critical: Krieger, Lenzmann and Rapha\"el constructed minimal mass blow-up solutions \cite{KrLeRa}, see also \cite{CSZ2024, KKP2025, KKP2026, Park2026}, and Choffrut and Pocovnicu proved ill-posedness below $H^{\frac12}$ \cite{ChPo2018}. G\'erard, Lenzmann, Pocovnicu and Rapha\"el constructed a two-soliton whose $H^1$ norm grows on a transient time interval and then saturates \cite{GLPR2018}, and the author constructed infinite cascades for the cubic half-wave equation on the line \cite{Ch2026}.

The difficulty in turning these approximation results into an infinite cascade for \eqref{halfwave} is the stability of the turbulent dynamics on an infinite time interval with the periodic restriction. We overcome this difficulty for one explicit family of trajectories, by solving \eqref{halfwave} backwards from $t=+\infty$, in the spirit of the construction of multi-solitary waves \cite{merlemulti, Martelmulti, MMnls, KMRhartree, GLPR2018}.


\subsection{Statement of the results}


Our simplest result can be stated without describing the profiles.

\begin{theorem}[Infinite cascades at every mass]
\label{thmallmass}
For every $M_0>0$, there exists a global solution $u\in\bigcap_{r<\frac32}\mathcal C(\R,H^r(\T))$ of \eqref{halfwave} with
$$M(u)=M_0,\qquad P(u)=2,\qquad E(u)=\Big(1+\frac{M_0}{2}\Big)^2,$$
such that, for every $\frac12<s<\frac32$,
\begin{equation}
\label{growthallmass}
\|u(t)\|_{H^s}=C_{M_0,s}\,e^{(2s-1)\sqrt{2M_0}\,t}\,(1+o(1))\ \ \mbox{as}\ \ t\to+\infty,
\end{equation}
where $C_{M_0,s}:=\sqrt{(M_0+2)\Gamma(2s+1)}\,\big(\frac{4M_0}{M_0+2}\big)^{\frac12-s}$.
\end{theorem}

To the best of our knowledge, Theorem \ref{thmallmass} provides the first infinite cascade for a defocusing fractional nonlinear Schr\"odinger equation on the torus, and it answers the open question raised in \cite{GeGr2016,GeGr2017,Ge2017} on the unboundedness of the Sobolev norm in the defocusing periodic case. The growth holds simultaneously for all $\frac12<s<\frac32$, for all large times and not only along a sequence of times, with an explicit exponential rate. The cubic half-wave equation \eqref{halfwave} belongs to the family of fractional  NLS model \eqref{fractional} by Majda, McLaughlin and Tabak \cite{MaMc1997}. Our result provides an exact, unforced and undamped Hamiltonian trajectory exhibiting an infinite-time cascade in this model. While distinct from the statistical steady-state cascades studied in wave turbulence theory, it rigorously demonstrates persistent transfer toward increasingly fine spatial scales.

\medskip

\noindent{\it The explicit profiles}. The solutions of Theorem \ref{thmallmass} belong to a family that we now describe. The family depends on two integers $K,N$, two positive numbers $M,Q$, and a finite Blaschke product $B$ of degree $J\geq1$. The plane wave $e^{-iKx}$, of frequency $-K$, is the carrier: the solution is a modulation of this plane wave. The integer $N$ is a covering number: the solution is $\frac{2\pi}{N}$-periodic up to this modulation. The number $M$ is the mass, and $\sqrt Q$ measures the amplitude of the part of the solution which concentrates, that is, of its deviation from a modulated plane wave: this deviation consists of $NJ$ peaks, which are bumps of height $2\sqrt Q$ and of width proportional to the scale $\lambda(t)$ defined in \eqref{defplambda} below, located near $NJ$ fixed points of the circle (Proposition \ref{propspikes}). We assume
\begin{equation}
\label{admissible}
1\leq K\leq N,\qquad \big|\sqrt{2K}-\sqrt Q\big|<\sqrt M<\sqrt{2K}+\sqrt Q.
\end{equation}
The second condition means that $\sqrt{2K}$, $\sqrt Q$, $\sqrt M$ are the side lengths of a nondegenerate triangle. We let $\vartheta\in(0,\pi)$ be the angle opposite to the side $\sqrt M$, and $\Lambda$ be four times the area of this triangle:
\begin{equation}
\label{deftheta}
\begin{gathered}
M=2K+Q-2\sqrt{2KQ}\cos\vartheta,\qquad \Lambda:=2\sqrt{2KQ}\sin\vartheta,\\
\Lambda^2=8KQ-(2K+Q-M)^2.
\end{gathered}
\end{equation}
The number $\Lambda>0$ will be the exponential rate of concentration. We define the two functions of time
\begin{equation}
\label{defplambda}
p(t):=\cos\vartheta-i\sin\vartheta\tanh\Big(\frac{\Lambda t}{2}\Big),\qquad \lambda(t):=1-|p(t)|^2=\frac{\sin^2\vartheta}{\cosh^2\big(\frac{\Lambda t}{2}\big)}.
\end{equation}
As $t\to+\infty$, the point $p(t)\in\Disc:=\{|z|<1\}$ converges to the point $e^{-i\vartheta}$ of the unit circle, and $\lambda(t)=4\sin^2\vartheta\,e^{-\Lambda t}(1+O(e^{-\Lambda t}))$ is the scale of the concentration: in the variable $y$ introduced below, the profile concentrates on $J$ arcs of length $O(\lambda(t))$, and its Fourier coefficients spread over the frequencies $n\lesssim\lambda(t)^{-1}$.

Let $B$ be a finite Blaschke product of degree $J\geq1$ vanishing at the origin:
\begin{equation}
\label{defblaschke}
B(z):=e^{i\gamma_0}\,z\prod_{m=2}^{J}\frac{z-\beta_m}{1-\overline{\beta_m}z},\qquad \gamma_0\in\R,\ \ \beta_m\in\Disc.
\end{equation}
The restriction of $B$ to the unit circle winds $J$ times around the origin with positive speed $|B'|$ (see Lemma \ref{lemmablaschke}), so the equation $B(\zeta)=e^{i\vartheta}$ has exactly $J$ distinct solutions $\zeta_1,\dots,\zeta_J$ on the unit circle, and $|B'(\zeta_j)|>0$. These points will be the concentration points. For $t\in\R$ and $z\in\C$ with $p(t)B(z)\neq1$, let
\begin{equation}
\label{defTheta}
\Theta(t,z):=\frac{B(z)-\overline{p(t)}}{1-p(t)B(z)},\qquad A(t,z):=\sqrt{2K}+\sqrt Q\,\Theta(t,z).
\end{equation}
For each $t$, $\Theta(t,\cdot)$ is a finite Blaschke product of degree $J$ whose $J$ zeros, the solutions of $B(z)=\overline{p(t)}$, converge to the points $\zeta_j$ as $t\to+\infty$. Hence $A(t,\cdot)$ is a constant plus a unimodular function on the circle, whose argument, as $t\to+\infty$, increases by $2\pi$ on an arc of length $O(\lambda(t))$ around each point $\zeta_j$, while its derivative is $O(\lambda(t))$ away from the points $\zeta_j$. The profile is finally
\begin{equation}
\label{defU}
U(t,x):=e^{-iKx-i(M+K)t}A\big(t,e^{iN(x-t)}\big),\qquad (t,x)\in\R\times\T.
\end{equation}
In words, $U$ is obtained from $A$ by a translation at unit speed, by the $N$-fold covering $y:=N(x-t)$ of the circle, by the carrier $e^{-iKx}$ and by a phase rotation. Let finally
\begin{equation}
\label{defell}
\ell(t):=1+|\log\lambda(t)|,\qquad\mbox{so that}\ \ \lambda(t)\ell(t)\sim\la t\ra e^{-\Lambda t}\ \ \mbox{for}\ \ t\geq0,
\end{equation}
and let $\W$ denote the Wiener algebra of absolutely convergent Fourier series, with norm $\|f\|_\W:=\sum_n|\hat f(n)|$.

\begin{theorem}[Concentrating trajectories]
\label{thmmain}
Let $K,N,M,Q$ satisfy \eqref{admissible} and let $B$ be a finite Blaschke product of degree $J$ with $B(0)=0$. There exists a global solution $u\in\bigcap_{r<\frac32}\mathcal C(\R,H^r(\T))$ of \eqref{halfwave} whose Fourier coefficients vanish outside $-K+N\Z$, with the following properties.
\begin{enumerate}
\item {\rm (}Approximation{\rm ).} For $t\geq0$ and $\frac12<s<\frac32$,
\begin{equation}
\label{approxmain}
\|u(t)-U(t)\|_{\W}\lesssim\lambda(t)\ell(t),\qquad \|u(t)-U(t)\|_{H^s}\lesssim_s\lambda(t)^{\frac32-s}\ell(t).
\end{equation}
\item {\rm (}Growth{\rm ).} For $\frac12<s<\frac32$, as $t\to+\infty$,
\begin{equation}
\label{growthmain}
\|u(t)\|_{H^s}=N^s\Big(Q\,\Gamma(2s+1)\sum_{j=1}^J|B'(\zeta_j)|^{2s-1}\Big)^{\frac12}\big(4\sin^2\vartheta\big)^{\frac12-s}e^{(s-\frac12)\Lambda t}\,(1+o(1)).
\end{equation}
\item {\rm (}Conservation laws{\rm ).}
\begin{equation}
\label{invariantsmain}
M(u)=M,\qquad P(u)=NJQ-KM,\qquad E(u)=\frac{NJQ+KM}{2}+\frac{M^2}{4}.
\end{equation}
\item {\rm (}Uniqueness{\rm ).} $u$ is the only solution of \eqref{halfwave} in $\mathcal C(\R,H^1)$ whose Fourier coefficients vanish outside $-K+N\Z$ and which satisfies \eqref{approxmain} for $s=1$.
\end{enumerate}
The implicit constants depend on $K,N,M,Q,B$ and $s$.
\end{theorem}

Theorem \ref{thmallmass} is the case $K=N=J=1$, $B(z)=z$, $M=M_0$, $Q=M_0+2$ of Theorem \ref{thmmain}. The relative error in \eqref{approxmain} is $\lambda\ell\sim\la t\ra e^{-\Lambda t}$, since $\|U(t)\|_{H^s}\sim\lambda(t)^{\frac12-s}$ while $\|U(t)\|_\W\sim1$; indeed $\sqrt M=\|U(t)\|_{L^2}\leq\|U(t)\|_\W\lesssim1$. The concentration points and their relative widths can be prescribed: for distinct points $\zeta_1,\dots,\zeta_J$ of the circle and weights $w_j>0$ with $\sum_jw_j=1$, there is a Blaschke product $B$ of degree $J$ with $B(0)=0$, $B(\zeta_j)=e^{i\vartheta}$ and $|B'(\zeta_j)|^{-1}=w_j$, and the $j$-th peak then concentrates at the scale $w_j\lambda(t)$ in the variable $y:=N(x-t)$, which corresponds to the half-width $\frac{w_j\lambda(t)}{2N}(1+o(1))$ in the variable $x$, that is, to the half-width at half-maximum of the squared modulus of the deviation of the solution from a modulated plane wave near each of the $N$ corresponding concentration points; see Corollary \ref{corgeometry} and \eqref{defcenters}. The number of concentration points can be arbitrarily large on one fixed level set of the conservation laws: for $K=N=1$, $M=2$ and $Q=\frac2J$, the solutions satisfy $(M(u),P(u),E(u))=(2,0,3)$ for every $J\geq1$, at the price of the rate $\Lambda_J:=\frac{2\sqrt{4J-1}}{J}$, which tends to zero as $J\to+\infty$; see Corollary \ref{corsameshell}.

\medskip

\noindent{\it Endpoint regularity}. The restriction $s<\frac32$ in Theorem \ref{thmmain} is sharp in a precise sense. Let $\FL$ be the space of functions with $\|f\|_{\FL}:=\sup_n\la n\ra^2|\hat f(n)|<\infty$, and, for $\sigma\in\R$, let $\Hlog$ be the logarithmic Sobolev space with norm
\begin{equation}
\label{deflogsobolev}
\|f\|_{\Hlog}^2:=\sum_{n\in\Z}\frac{\la n\ra^3}{\log^{2\sigma}(e+|n|)}|\hat f(n)|^2.
\end{equation}
Note that $\FL\subset\Hlog$ for $\sigma>\frac12$, and that $\FL$ is contained in the Besov space $\mathrm B^{\frac32}_{2,\infty}(\T)$, defined with the dyadic blocks of the Fourier series as in \cite[Chapter 2]{BCD2011}; this inclusion is elementary, see the proof of Theorem \ref{thmendpoint}.

\begin{theorem}[Endpoint regularity]
\label{thmendpoint}
Let $u$ be the solution of Theorem \ref{thmmain}.
\begin{enumerate}
\item {\rm (}Upper bound{\rm ).} For every $t\in\R$, $u(t)\in\FL$; in particular $u(t)\in\mathrm B^{\frac32}_{2,\infty}$ and $u(t)\in\Hlog$ for every $\sigma>\frac12$. Moreover, for $t\geq0$ and $\sigma>\frac12$,
\begin{equation}
\label{endpointupper}
\|u(t)-U(t)\|_{\FL}\lesssim\ell(t),\qquad \|u(t)-U(t)\|_{\Hlog}\lesssim_\sigma\ell(t)^{1-\sigma},
\end{equation}
and, as $t\to+\infty$,
$$\|u(t)\|_{\Hlog}=N^{\frac32}\Big(6Q\sum_{j=1}^J|B'(\zeta_j)|^2\Big)^{\frac12}\lambda(t)^{-1}\ell(t)^{-\sigma}(1+o(1)).$$
\item {\rm (}Sharpness{\rm ).} There is an explicit complex number $\Xi$, which only depends on $K$, $N$, $M$ and $Q$ and is given by \eqref{defxi}, such that if $\Xi\neq0$, then $u(t)\notin\Hlog$ for every $t\in\R$ and every $\sigma\leq\frac12$; in particular $u(t)\notin H^{\frac32}$. Moreover, $\Xi\neq0$ if $M=2K$. For fixed $K$ and $N$, the set of the pairs $(M,Q)$ with $\Xi\neq0$ is open and dense in the set of the pairs which satisfy \eqref{admissible}, and its complement in this set has Lebesgue measure zero.
\end{enumerate}
\end{theorem}

\noindent{\it Concentration of the kinetic energy}. We call $\||D|^{\frac12}u\|_{L^2}^2=\sum_n|n||\hat u(n)|^2$ the kinetic energy of $u$, and $\big||D|^{\frac12}u\big|^2\frac{dx}{2\pi}$ its density. In the frame moving with unit speed, the solution converges below $H^{\frac12}$ to a plane wave, while a fixed amount of kinetic energy concentrates at the $NJ$ points
\begin{equation}
\label{defxjk}
x_{j,k}:=\frac{\arg\zeta_j+2\pi k}{N},\qquad 1\leq j\leq J,\ \ 0\leq k\leq N-1.
\end{equation}
We denote by $\delta_x$ the Dirac measure at $x\in\T$, and we say that a family of finite measures on $\T$ converges weakly if it converges against every continuous function on $\T$.

\begin{theorem}[Concentration and defect measures]
\label{thmdefect}
Let $u$ be the solution of Theorem \ref{thmmain}, let $V(t,x):=e^{i(M+2K)t}u(t,x+t)$, and let $c_\infty:=\sqrt{2K}-\sqrt Q\,e^{i\vartheta}$, so that $|c_\infty|^2=M$. As $t\to+\infty$:
\begin{enumerate}
\item $V(t)\rightharpoonup c_\infty e^{-iKx}$ weakly in $H^{\frac12}$ and strongly in $L^2$, and $\|V(t)-c_\infty e^{-iKx}\|_{\dot H^{\frac12}}^2\to NJQ$;
\item weakly in the sense of measures on $\T$,
\begin{gather*}
\big||D|^{\frac12}V(t)\big|^2\frac{dx}{2\pi}\rightharpoonup KM\frac{dx}{2\pi}+Q\sum_{j=1}^J\sum_{k=0}^{N-1}\delta_{x_{j,k}},\\
\big||D|^{\frac12}\big(V(t)-c_\infty e^{-iKx}\big)\big|^2\frac{dx}{2\pi}\rightharpoonup Q\sum_{j=1}^J\sum_{k=0}^{N-1}\delta_{x_{j,k}};
\end{gather*}
we call the last limit the defect measure of $u$;
\item on the level set $M(u)=2$, $P(u)=0$, $E(u)=3$ of the conservation laws, the defect measures of the solutions of Theorem \ref{thmmain}, obtained as $J$, $B$ and the parameters vary, are dense among the positive Borel measures on $\T$ of total mass $2$, for the weak convergence of measures.
\end{enumerate}
\end{theorem}

Each concentration point carries the kinetic energy $Q$, independently of its width, and each peak has the same profile $\xi\mapsto\frac{2\sqrt Q\,e^{i\vartheta}}{1-i\xi}$, in the variable $\xi\in\R$ given by the signed distance to the center of the peak divided by its half-width; this profile depends neither on the point nor on $B$; see Remark \ref{remarkdensities} and Proposition \ref{propspikes}. The density statement is obtained by varying the solution, since each solution has finitely many concentration points.


\subsection{Comments on the results}


\noindent 

\noindent 1. {\it The mechanism of growth}. On the lattice $-K+N\Z$, the dispersion relation $m\mapsto|m|$ of \eqref{halfwave} coincides with the linear function $m$ at the nonnegative frequencies $m=Nn-K$, $n\geq1$, while at the carrier frequency $m=-K$, which lies on the other side of the kink of $|m|$ at $m=0$, it equals $m+2K$. In the frame moving with unit speed, the linear part $m$ is removed by the translation, and the modes $Nn-K$, $n\geq0$, only feel the cubic Szeg\H{o} nonlinearity and the linear term $2K\hat w(0)$ on the constant mode. In other words, the perturbed Szeg\H{o} equation \eqref{alphaszego} with $\varkappa=2K$ is not an approximation, but the exact equation for the modes $Nn-K$, $n\geq0$, up to their coupling with the other modes through the nonlinearity; see Lemma \ref{lemmacomoving}. The profiles $A$ are, up to the phase factor $e^{-i(M+2K)t}$, the rational trajectories of \eqref{alphaszego} found by Xu \cite{Xu2014, Xu2015}: the Blaschke product $\Theta(t)$ has zeros converging exponentially fast to the circle, and its Fourier coefficients spread over the frequencies $n\lesssim\lambda(t)^{-1}$, which gives $\|A(t)\|_{H^s}\sim\lambda(t)^{\frac12-s}$.

\smallskip

\noindent
2. {\it Negative frequencies and the threshold $s<\frac32$}. The cubic nonlinearity of the profile has a negative-frequency part, $\Pi_-(|A|^2A)=2K\sqrt Q\,\overline F$ with $F:=\Theta+\bar p$, whose Wiener norm does not tend to zero, so the negative frequencies cannot be discarded. On the negative frequencies the dispersion relation has the opposite slope, so that, in the moving frame, the modes $Nn-K$ with $n<0$ rotate with the large frequency $2(K+N|n|)$, up to a bounded correction. Hence their response to the forcing gains one derivative: the first negative-frequency correction has Fourier coefficients of size at most $\frac{\lambda}{|n|}e^{-\kappa_*\lambda|n|}$, for a constant $\kappa_*>0$, so that its $H^s$ norm is of order $\lambda^{\frac32-s}$ and its Wiener norm of order $\lambda\ell$. When the error is propagated backwards from a time $\tau$ to a time $t<\tau$, its $H^s$ norm, measured relative to the norm $\lambda^{\frac12-s}$ of the profile, loses the factor $(\lambda(t)/\lambda(\tau))^{s-\frac12}\leq e^{(s-\frac12)\Lambda(\tau-t)}$, while the forcing decays like $e^{-\Lambda\tau}$. These two rates are compatible exactly when $s<\frac32$.

\smallskip

\noindent
3. {\it Finite regularity}. Bourgain's question is formulated in $H^s$, and our solutions have finite regularity: their Fourier coefficients are $O(|n|^{-2})$, and, when $\Xi\neq0$, this decay cannot be improved by any positive power of $\log|n|$, since $u(t)\notin H^{\frac32,\frac12}_{\log}$. The $|n|^{-2}$ tail is created by the negative-frequency forcing, which is transmitted back to the nonnegative frequencies by the coupling of Proposition \ref{proplinearized}; its leading coefficient $\Xi$ is a Mellin transform, that is, an integral of the form $\int_0^{+\infty}\rho^{-i\omega}g(\rho)\,d\rho$ with $\omega:=\frac{M-2K}{\Lambda}$, of an explicit function $g$ of the frequency rescaled by the concentration scale.

\smallskip

\noindent
4. {\it Geometry of the concentration}. The Blaschke product $B$ prescribes the concentration points and their relative widths, while the kinetic energy carried by each point is the fixed number $Q$. On one fixed level set of the conservation laws, the number of concentration points is unbounded, while the exponential rate tends to zero, and the defect measures are dense; on the other hand, in Appendix \ref{appendixinvariants} we show that the explicit family has a positive energy floor and an optimal ratio between the squared exponential rate and the energy, and we describe the level sets with unbounded numbers of concentration points (Propositions \ref{propfloor}, \ref{proprate} and \ref{proppeaks}).


\subsection{Strategy of the proof}


We outline the main steps of the proof of Theorem \ref{thmmain}, and the additional step needed for Theorem \ref{thmendpoint}.

\medskip

\noindent{\bf Step 1:} {\it Co-moving variables and explicit profiles}. We write $u(t,x)=e^{-iKx-i(M+K)t}v(t,N(x-t))$. The new unknown $v(t,y)$ solves
$$i\pa_tv=(L-M-2K)v+|v|^2v,$$
where $L$ is the Fourier multiplier with symbol $2(K-Nn)_+$. On nonnegative frequencies $L$ only acts on the constant mode, $Lf=2K\hat f(0)$, while on negative frequencies it is the large operator $2(K+N|n|)$. The profile $A:=\sqrt{2K}+\sqrt Q\,\Theta$ solves the projection of this equation on nonnegative frequencies exactly. This is checked through a short list of algebraic identities on the circle, the main ones being
$$|A|^2=M+2\sqrt{2KQ}\,\Re F,\qquad i\pa_t\Theta=2\sqrt{2KQ}\,(\Re F)\,\Theta,\qquad F:=\Theta+\bar p;$$
see Proposition \ref{propalgebra}. The partial fraction decomposition of $F$ splits the profile into $J$ peaks, concentrated at the scales $\lambda_j(t)=\frac{\lambda(t)}{|B'(\zeta_j)|}(1+O(\lambda(t)))$, which interact weakly because they are separated; see Lemma \ref{lemmapeaks}.

\medskip

\noindent{\bf Step 2:} {\it The negative-frequency correction and a norm adapted to the concentration}. We look for $v=A+\eps$. The forcing term $2K\sqrt Q\,\overline F$ of the equation for $\eps$ is not small, but it lives on negative frequencies, where $L$ is large. Inverting $L$ there defines an explicit correction $\bar z\bar h_1$, where $z:=e^{iy}$ and $\bar h_1$ is the complex conjugate of a function $h_1\in L^2_+$, which is $O(\lambda\ell)$ in the norm
$$\|f\|_{X^s_t}:=\|f\|_\W+\lambda(t)^{s-\frac12}\|f\|_{H^s}.$$
This norm is designed so that the profile itself is bounded, $\|A(t)\|_{X^s_t}\lesssim1$, and so that $X^s_t$ is an algebra with constants independent of $t$. In this norm, the multiplication by a function concentrated at the scale $\lambda$, that is, whose Fourier coefficients are bounded by $C_0\lambda e^{-\kappa_0\lambda|n|}$ for some constants $C_0,\kappa_0>0$, followed by the inverse of $L$, is a small operator, of size $\lambda\ell$; see Lemma \ref{lemmasmoothing}.

\medskip

\noindent{\bf Step 3:} {\it Coordinates adapted to the concentration}. The linearized operator at $A$ is not small, and its coupling between positive and negative frequencies must be kept. We decompose
$$\eps=\iota_t(a)+z\Theta^2f+\bar z\bar k,\qquad z:=e^{iy},$$
where $\iota_t(a)$ belongs to the orthogonal complement $\Kt$ of $z\Theta^2L^2_+$ in $L^2_+$, a space of dimension $2J+1$ which we parametrize by the vector $a\in\C^{2J+1}$, and $f$ and $k$ belong to $L^2_+$; the explicit correction $h_1$ of Step 2 will be subtracted from $k$. The space $\Kt$ is the tangent space of the set of rational functions containing the profile, and the positive part of the linearized flow maps it into itself. On $\Kt$ the linearized flow reduces to a finite-dimensional system for $a$, whose limit as $t\to+\infty$ is triangular in suitable coordinates, with the eigenvalues $0$ and $\Lambda$ only; see Lemma \ref{lemmafiniteblock}. The transverse coordinate $f$, which describes the component $z\Theta^2f$ of $\eps$ in the orthogonal complement $z\Theta^2L^2_+$ of $\Kt$, is not forced by the finite-dimensional part, and its own evolution is a rotation; the modes of the coordinate $k$ rotate at the large frequencies $2(N(n+1)+K)$, $n\geq0$. These facts are exact algebraic consequences of the Blaschke structure; see Proposition \ref{proplinearized}.

\medskip

\noindent{\bf Step 4:} {\it A normal form for the infinite-dimensional part}. The pair $(f,k)$ is coupled through the multiplication by the square of the function $\tilde A:=\sqrt Q+\sqrt{2K}\,\Theta$. We split this coupling into its constant part and a concentrated part. Since the modes of $k$ rotate at large frequencies, the concentrated part can be removed by a bounded change of variables: after subtracting from $k$ the correction $h_1$ and a small function which depends on $a$, we write the pair as $(I+\mathcal T(t))\Phi$, where $\mathcal T(t)=O(\lambda\ell)$ solves a commutator equation exactly; see Lemma \ref{lemmatransformation} and Proposition \ref{propnormalform}. The remaining constant-coefficient part splits into $2\times2$ blocks with real eigenvalues, and its propagator grows at most linearly in time; see Lemma \ref{lemmahardygroup}.

\medskip

\noindent{\bf Step 5:} {\it Backward integration from $t=+\infty$}. After Steps 3 and 4, the equation for the coordinates $Z:=(a,\Phi)$ of $\eps$ is a triangular linear system with propagator $\mathcal U(t,\tau)$, perturbed by terms of size $\lambda\ell$. Since the finite-dimensional block has no eigenvalue with negative real part, and since the constant-coefficient group grows at most linearly, the propagator is polynomially bounded backwards in time, except for the factor $e^{(s-\frac12)\Lambda(\tau-t)}$ coming from the weight of the norm; see Lemma \ref{lemmabackward}. We solve
$$Z(t)=-\int_t^{+\infty}\mathcal U(t,\tau)\big(\mbox{forcing}+\mbox{remainders}\big)(\tau)\,d\tau$$
by a contraction argument in a space of trajectories with $\|Z(t)\|\lesssim\lambda(t)\ell(t)$; the resulting bound is integrable in $\tau$ if and only if $s<\frac32$. This proves Theorem \ref{thmmain}; see Section \ref{sectionbackward}.

\medskip

\noindent{\bf Step 6:} {\it Endpoint analysis}. To prove Theorem \ref{thmendpoint}, we add a second negative-frequency correction, which removes a logarithmic loss in the forcing, and we integrate the Duhamel formula mode by mode before taking the supremum over the modes. The leading part of the Fourier tail of the transverse coordinate $f$ is then an explicit oscillatory integral, which becomes a Mellin transform after the change of variables $\rho:=n\lambda_j(\tau)$; see Section \ref{sectionendpoint}.

\medskip

\noindent{\bf Step 7:} {\it Concentration of the kinetic energy}. To prove Theorem \ref{thmdefect}, we compute the limit of the momentum density $\Re(\overline SDS)$ of the peaks $S$, which is a local quantity, peak by peak. The commutator of $|D|^{\frac12}$ with the multiplication by a smooth function is bounded on $L^2$, while the peaks tend to zero in $L^2$, so that the kinetic density $\big||D|^{\frac12}S\big|^2$ has the same limit; see Section \ref{sectiondefect}.

\medskip

This paper is organized as follows. In Section \ref{sectionprofile} we derive the equation in the co-moving variables and we study the explicit profiles. In Section \ref{sectionlinear} we construct the adapted coordinates, we compute the linearized equation in these coordinates, and we perform the normal form transformation. In Section \ref{sectionbackward} we solve the final-value problem and we prove Theorems \ref{thmallmass} and \ref{thmmain}. Section \ref{sectionendpoint} is devoted to the endpoint regularity, Theorem \ref{thmendpoint}, and Section \ref{sectiondefect} to the concentration of the kinetic energy, Theorem \ref{thmdefect}. Appendix \ref{appendixblaschke} collects the properties of finite Blaschke products that we use, and Appendix \ref{appendixinvariants} describes the sharp restrictions on the conservation laws and the optimal growth rates within the explicit family.

\vskip 0.5cm


\noindent{\bf Notations}. All functions are complex valued and, unless stated otherwise, defined on $\T$, and $(f|g):=\int_\T f\bar g\,\frac{dx}{2\pi}$ is the scalar product of $L^2$. We let $H^s_+:=\Pi_+H^s$, and we identify $f\in L^2_+$ with its holomorphic extension $f(z):=\sum_{n\geq0}\hat f(n)z^n$ to the unit disc $\Disc$; in particular $f(0)=\hat f(0)$. The variable $z$ denotes both a point of $\C$ and the function $z:=e^{iy}$ on $\T$. For $b\in\W$, the Toeplitz operator with symbol $b$ is $\mathbf T_bf:=\Pi_+(bf)$, $f\in L^2_+$. A bar denotes complex conjugation; on $\T$, $\bar z=z^{-1}$. Operators containing complex conjugation are real-linear. For $f\in L^2$ with $\hat f(0)=0$, $D^{-1}f$ is the function with Fourier coefficients $\hat f(n)/n$, $n\neq0$, and $\widehat{D^{-1}f}(0):=0$. A prime denotes the time derivative of a scalar function of $t$ only, such as $p$, $\lambda$, $\mu$, $p_j$, $b_j$, $\lambda_j$, $\mu_j$ (defined in Section \ref{sectionprofile}), and the complex derivative of a holomorphic function of $z$ which does not depend on $t$, such as $B$; for the functions $\Theta$, $A$, $F$, $G_j$ (see Section \ref{sectionprofile}), which depend on $t$ and $z$, and for the unknowns $\eps$, $a$, $\Phi$, $Z$, we write $\pa_t$ and $\pa_z$. Primed letters such as $t'$, $\tau'$, $\eps'$, $Z'$, $a'$, $\Phi'$ denote a second variable or a second object of the same kind.

Unless stated otherwise, the implicit constants in the estimates are independent of time and of the functions involved, and they may depend on the parameters $K,N,M,Q,B$ and on the Sobolev indices. Asymptotic equivalence as $t\to+\infty$ is always written as $f(t)=g(t)(1+o(1))$. In heuristic statements, $x\ll y$ means that $x/y$ is small, and $x\gg y$ that $y\ll x$. The number $T_0\geq1$ denotes a large time, depending only on the parameters, after which all the statements of Lemma \ref{lemmapeaks} hold. It may be increased finitely many times; from Lemma \ref{lemmatransformation} on, it may also depend on the Sobolev index $s$ under consideration.


\section{Explicit concentrating solutions of the perturbed Szeg\H{o} equation}
\label{sectionprofile}


This section is devoted to the profiles. We first write \eqref{halfwave} in the frame moving with unit speed and modulated by the carrier $e^{-iKx}$; in this frame the nonnegative modes obey the perturbed Szeg\H{o} equation \eqref{alphaszego} with $\varkappa=2K$, up to their coupling with the negative modes (Lemma \ref{lemmacomoving}). We then check that the functions $A$ of \eqref{defTheta} solve the projected equation exactly (Proposition \ref{propalgebra}), we decompose them into separated peaks (Lemma \ref{lemmapeaks}), and we compute their norms and their conservation laws (Lemma \ref{lemmanorms}). Throughout the paper, the parameters $K,N,M,Q$ satisfy \eqref{admissible}, the angle $\vartheta$ and the rate $\Lambda$ are given by \eqref{deftheta}, and $B$ is a Blaschke product \eqref{defblaschke} of degree $J$ with $B(0)=0$.


\subsection{Co-moving variables}


Given a function $v(t,y)$ on $\R\times\T$, we let
\begin{equation}
\label{comoving}
u(t,x):=e^{-iKx-i(M+K)t}\,v\big(t,N(x-t)\big).
\end{equation}
The variable $y:=N(x-t)$ follows the positive frequencies, which are transported at unit speed, and it unfolds the $N$-fold covering; the factor $e^{-iKx}$ is the carrier; the phase $e^{-i(M+K)t}$ is chosen so that the explicit profile $A$, without any further phase rotation, solves the projected equation \eqref{eqA} below. In Fourier variables, \eqref{comoving} reads
\begin{equation}
\label{comovingfourier}
\begin{gathered}
\hat u(t,Nn-K)=e^{-i(M+K+Nn)t}\,\hat v(t,n),\qquad n\in\Z,\\
\hat u(t,m)=0\qquad\mbox{for}\ \ m\notin-K+N\Z.
\end{gathered}
\end{equation}

\begin{lemma}[Co-moving variables]
\label{lemmacomoving}
Let $\mathcal I\subset\R$ be an interval, $s>\frac12$, and $v\in\mathcal C(\mathcal I,H^s(\T))$, and let $u$ be defined by \eqref{comoving}.
\begin{enumerate}
\item {\rm (}Equation{\rm ).} $u$ solves \eqref{halfwave} on $\mathcal I$ if and only if $v$ solves
\begin{equation}
\label{eqv}
i\pa_tv=(L-M-2K)v+|v|^2v,\qquad L:=|ND-K|-(ND-K),
\end{equation}
where $D:=-i\pa_y$.
\item {\rm (}The operator $L${\rm ).} The Fourier multiplier $L$ has the symbol $2(K-Nn)_+$, $n\in\Z$. In particular,
\begin{equation}
\label{actionL}
Lf=2K\hat f(0)\ \ \mbox{for}\ \ f\in L^2_+,\qquad \widehat{Lf}(-n)=2(K+Nn)\hat f(-n)\ \ \mbox{for}\ \ n\geq1.
\end{equation}
\item {\rm (}Norms{\rm ).} $\|u(t)\|_\W=\|v(t)\|_\W$, $\|u(t)\|_{L^q}=\|v(t)\|_{L^q}$ for $1\leq q\leq\infty$, and, for every $r\in\R$, $\|u(t)\|_{H^r}^2=\sum_n\la Nn-K\ra^{2r}|\hat v(t,n)|^2$, where
\begin{equation}
\label{comparable}
(1+K)^{-1}\la n\ra\leq\la Nn-K\ra\leq(N+K)\la n\ra,\qquad n\in\Z.
\end{equation}
\end{enumerate}
\end{lemma}

\begin{proof}
Since $x\mapsto N(x-t)$ is an $N$-fold covering of $\T$, it maps the normalized measure to itself, which gives the identity for the $L^q$ norms, and it intertwines the products: $|u|^2u(t,x)=e^{-iKx-i(M+K)t}(|v|^2v)(t,N(x-t))$. The identities for the Wiener and Sobolev norms follow from \eqref{comovingfourier}, and \eqref{comparable} follows from $|n|\leq|Nn-K|+K$, from $|Nn-K|\leq(N+K)|n|$ for $n\neq0$, and from $\la a+b\ra\leq\la a\ra+|b|$ for $a,b\in\R$. Both sides of \eqref{halfwave} have Fourier coefficients supported in $-K+N\Z$. By \eqref{comovingfourier}, the mode $Nn-K$ of $i\pa_tu-|D|u$ equals
$$e^{-i(M+K+Nn)t}\Big[i\pa_t\hat v(t,n)+\big(M+K+Nn-|Nn-K|\big)\hat v(t,n)\Big].$$
Since $|Nn-K|-Nn-M-K=\big(|Nn-K|-(Nn-K)\big)-M-2K$, the equation \eqref{halfwave} is equivalent to \eqref{eqv}, mode by mode. Finally $|m|-m=2(-m)_+$ for $m=Nn-K$, and the assumption $N\geq K$ gives $(K-Nn)_+=0$ for $n\geq1$, which yields \eqref{actionL}.
\end{proof}

Lemma \ref{lemmacomoving} explains the role of the carrier. On the nonnegative modes, $L$ only acts on the constant mode, and the projection of \eqref{eqv} reads
$$i\pa_t\Pi_+v=\Pi_+(|v|^2v)+2K\hat v(0)-(M+2K)\Pi_+v.$$
If $\Pi_-v$ is neglected in the nonlinearity, the function $w:=e^{-i(M+2K)t}\Pi_+v$ solves exactly the perturbed Szeg\H{o} equation \eqref{alphaszego} with $\varkappa=2K$. The linear term $2K\hat v(0)$ is the difference $|m|-m=2K$ at the carrier frequency $m=-K$, which is the only negative frequency of the form $Nn-K$ with $n\geq0$. On the negative modes, $L$ is the large operator $2(K+N|n|)$: in the frame of $y$, the negative frequencies rotate fast, which is the origin of the gain of one derivative described in Section \ref{sectionlinear}.


\subsection{Algebra of the profiles}


We first record the elementary properties of the functions $p$ and $\lambda$ of \eqref{defplambda}. Let
\begin{equation}
\label{defmu}
\mu(t):=-\frac{\lambda'(t)}{\lambda(t)}=\Lambda\tanh\Big(\frac{\Lambda t}2\Big)
\end{equation}
be the instantaneous rate of concentration, which increases from $0$ to $\Lambda$ on $[0,+\infty)$.

\begin{lemma}[The functions $p$, $\lambda$, $\mu$ and $\ell$]
\label{lemmaplambda}
\begin{enumerate}
\item {\rm (}Derivatives{\rm ).} For $t\in\R$,
\begin{equation}
\label{derivp}
p'(t)=-i\sqrt{2KQ}\,\lambda(t),\qquad \sqrt{2KQ}\,\big(p(t)-\overline{p(t)}\big)=-i\mu(t),
\end{equation}
and $\lambda''=(\mu^2-\mu')\lambda$, with $\mu'=\frac{\Lambda^2}{2\sin^2\vartheta}\lambda$.
\item {\rm (}The scale $\lambda${\rm ).} For $0\leq t\leq\tau$,
\begin{equation}
\label{elementarylambda}
\begin{gathered}
\sin^2\vartheta\, e^{-\Lambda t}\leq\lambda(t)\leq\min\big(1,4\sin^2\vartheta\, e^{-\Lambda t}\big),\\
e^{-\Lambda(\tau-t)}\leq\frac{\lambda(\tau)}{\lambda(t)}\leq 4e^{-\Lambda(\tau-t)},
\end{gathered}
\end{equation}
and $\int_t^{+\infty}\lambda(\tau')\,d\tau'\leq\frac4\Lambda\lambda(t)$.
\item {\rm (}The point $p$ and the rate $\mu${\rm ).} For $t\geq0$, $|p(t)-e^{-i\vartheta}|\leq\frac{\lambda(t)}{\sin\vartheta}$ and $|\mu(t)-\Lambda|\leq\frac{2\Lambda}{\sin^2\vartheta}\lambda(t)$.
\item {\rm (}The logarithm $\ell${\rm ).} The function $\lambda\ell$ is decreasing on $[0,+\infty)$, and for $0\leq t\leq\tau$,
\begin{equation}
\label{ratio}
\begin{gathered}
\frac{\ell(\tau)}{\ell(t)}\leq1+\Lambda(\tau-t),\qquad \frac{\lambda(\tau)\ell(\tau)}{\lambda(t)\ell(t)}\lesssim\la\tau-t\ra\,e^{-\Lambda(\tau-t)},\\
\int_t^{+\infty}\lambda\ell^2\,d\tau'\lesssim\lambda(t)\ell(t)^2.
\end{gathered}
\end{equation}
\item {\rm (}Time integrals{\rm ).} Let $\kappa>0$. For $n\geq1$, $t\geq1$ and $q\in\{1,2\}$,
\begin{equation}
\label{timeintegrals}
\begin{gathered}
\int_t^{+\infty}\lambda(\tau)^qe^{-\kappa\lambda(\tau)n}\,d\tau\lesssim_\kappa\frac{1}{n^q},\qquad \int_t^{+\infty}\lambda(\tau)^q\ell(\tau)e^{-\kappa\lambda(\tau)n}\,d\tau\lesssim_\kappa\frac{1+\log n}{n^q},\\
\int_t^{+\infty}(\tau-t)\lambda(\tau)e^{-\kappa\lambda(\tau)n}\,d\tau\lesssim_\kappa\frac{1+\log n}{n}.
\end{gathered}
\end{equation}
\end{enumerate}
\end{lemma}

The time integrals of item (5) express that the function $e^{-\kappa\lambda(\tau)n}$ is exponentially small as long as $\lambda(\tau)\geq\frac1n$, while $\lambda(\tau)$ decays exponentially once $\lambda(\tau)\leq\frac1n$: the integrals are concentrated on a time interval of length $O(1)$ around the time where $\lambda(\tau)=\frac1n$.

\begin{proof}
\noindent{\bf Step 1:} {\it Derivatives.} Differentiating \eqref{defplambda}, we get
$$p'=-i\frac{\Lambda\sin\vartheta}{2\cosh^2(\Lambda t/2)}=-i\frac{\Lambda}{2\sin\vartheta}\lambda,$$
and $\frac{\Lambda}{2\sin\vartheta}=\sqrt{2KQ}$ by \eqref{deftheta}; similarly, $p-\bar p=-2i\sin\vartheta\tanh(\frac{\Lambda t}2)$ gives the second identity of \eqref{derivp}. Differentiating \eqref{defmu}, $\mu'=\frac{\Lambda^2}{2\cosh^2(\Lambda t/2)}=\frac{\Lambda^2}{2\sin^2\vartheta}\lambda$, and $\lambda''=(-\mu\lambda)'=(\mu^2-\mu')\lambda$.

\noindent{\bf Step 2:} {\it Bounds.} We write $\lambda(t)=4\sin^2\vartheta\,e^{-\Lambda t}(1+e^{-\Lambda t})^{-2}$, which gives the first bound of \eqref{elementarylambda}, and $\lambda\leq\sin^2\vartheta\leq1$ is clear. The bounds on $\frac{\lambda(\tau)}{\lambda(t)}=\frac{\cosh^2(\Lambda t/2)}{\cosh^2(\Lambda\tau/2)}$ follow from $\cosh(a+b)\leq e^b\cosh a$ and $\frac{e^a}{2}\leq\cosh a\leq e^a$ for $a,b\geq0$, and they give the bound on the integral of $\lambda$. Next, $p-e^{-i\vartheta}=i\sin\vartheta\,(1-\tanh(\frac{\Lambda t}2))$ and $\mu-\Lambda=-\Lambda(1-\tanh(\frac{\Lambda t}2))$, with $1-\tanh(\frac{\Lambda t}2)=\frac{2}{e^{\Lambda t}+1}\leq2e^{-\Lambda t}$ and $e^{-\Lambda t}\leq\frac{\lambda}{\sin^2\vartheta}$; for the bound on $p$, we also use $\frac{2}{e^{\Lambda t}+1}\leq\frac{4e^{\Lambda t}}{(e^{\Lambda t}+1)^2}=\frac{\lambda}{\sin^2\vartheta}$, which holds for $t\geq0$. This proves items (2) and (3).

\noindent{\bf Step 3:} {\it The logarithm.} Since $\lambda\leq1$, we have $\ell=1+\log\frac1\lambda$, and $\frac{d}{d\lambda}\big(\lambda(1+\log\frac1\lambda)\big)=\log\frac1\lambda\geq0$; since $\lambda$ is decreasing on $[0,+\infty)$, so is $\lambda\ell$. By \eqref{elementarylambda}, $\ell(\tau)\leq\ell(t)+\Lambda(\tau-t)\leq\ell(t)(1+\Lambda(\tau-t))$, since $\ell(t)\geq1$. Together with $\frac{\lambda(\tau)}{\lambda(t)}\leq4e^{-\Lambda(\tau-t)}$, this gives the second bound of \eqref{ratio}, and $\int_t^{+\infty}\lambda\ell^2d\tau'\leq4\lambda(t)\ell(t)^2\int_t^{+\infty}e^{-\Lambda(\tau'-t)}\big(1+\Lambda(\tau'-t)\big)^2d\tau'=\frac{20}{\Lambda}\lambda(t)\ell(t)^2$.

\noindent{\bf Step 4:} {\it Time integrals.} For $\tau\geq1$, we have $\pa_\tau\lambda=-\mu\lambda$, with $\mu(\tau)\geq\mu(1)>0$ by \eqref{defmu}. Changing variables $\tau\mapsto\lambda(\tau)$, and then $\lambda=\frac\rho n$, we obtain, since $\ell=1+|\log\lambda|$,
\begin{align*}
\int_t^{+\infty}\lambda^q\ell\, e^{-\kappa\lambda n}\,d\tau&\leq\frac{1}{\mu(1)}\int_0^{1}\lambda^{q-1}(1+|\log\lambda|)e^{-\kappa\lambda n}\,d\lambda\\
&\leq\frac{1}{\mu(1)n^q}\int_0^{+\infty}\rho^{q-1}\big(1+|\log\rho|+\log n\big)e^{-\kappa\rho}\,d\rho,
\end{align*}
which gives the second bound of \eqref{timeintegrals}; the same computation without the factor $\ell$ gives the first one. For the third one, $\lambda(\tau)\leq4e^{-\Lambda(\tau-t)}$ by \eqref{elementarylambda}, hence $\tau-t\leq\frac1\Lambda\log\frac{4}{\lambda(\tau)}\lesssim\ell(\tau)$, and the second bound with $q=1$ applies.
\end{proof}

The profiles are built from the Blaschke product $\Theta(t,\cdot)$ of \eqref{defTheta}. Its zero-mean part
\begin{equation}
\label{defF}
F(t,z):=\Theta(t,z)+\overline{p(t)}=\frac{\lambda(t)B(z)}{1-p(t)B(z)}
\end{equation}
is the part of the profile which concentrates, and the constant
\begin{equation}
\label{defc}
c(t):=\sqrt{2K}-\sqrt Q\,\overline{p(t)},\qquad c_\infty:=\sqrt{2K}-\sqrt Q\,e^{i\vartheta},
\end{equation}
is the amplitude of the carrier. We shall also use the function obtained by exchanging the roles of $\sqrt{2K}$ and $\sqrt Q$,
\begin{equation}
\label{defAtilde}
\begin{gathered}
\tilde A(t,z):=\sqrt Q+\sqrt{2K}\,\Theta(t,z),\\
\tilde c(t):=\tilde A(t,0)=\sqrt Q-\sqrt{2K}\,\overline{p(t)},\qquad \tilde c_\infty:=\sqrt Q-\sqrt{2K}e^{i\vartheta},
\end{gathered}
\end{equation}
which will describe the coupling between the positive and the negative frequencies in Section \ref{sectionlinear}.

\begin{proposition}[Algebra of the profiles]
\label{propalgebra}
For every $t\in\R$, the following holds.
\begin{enumerate}
\item {\rm (}Structure{\rm ).} $F(t)$ and $A(t)$ are holomorphic in a neighborhood of $\overline\Disc$, $F(t,0)=0$, $A=c+\sqrt Q\,F$, and $|\Theta|=1$ on $\T$. In particular $|A|\leq\sqrt{2K}+\sqrt Q$ on $\T$, $\tilde A=\Theta\bar A$ on $\T$, $\|F(t)\|_{L^\infty}=1+|p(t)|$, and
\begin{equation}
\label{modc}
|c(t)|^2=M-Q\lambda(t),\qquad |\tilde c(t)|^2=M-2K\lambda(t),\qquad |c_\infty|^2=|\tilde c_\infty|^2=M.
\end{equation}
\item {\rm (}Quadratic identity{\rm ).} On $\T$,
\begin{equation}
\label{quadraticF}
|F|^2=\lambda+pF+\bar p\,\overline F,
\end{equation}
so that $\Pi_-(|F|^2)=\bar p\,\overline F$ and $\|F\|_{L^2}^2=\lambda$.
\item {\rm (}Time derivative{\rm ).} On $\T$,
\begin{equation}
\label{derivF}
i\pa_t\Theta=2\sqrt{2KQ}\,(\Re F)\,\Theta.
\end{equation}
\item {\rm (}Modulus{\rm ).} On $\T$,
\begin{equation}
\label{modA}
|A|^2=M+2\sqrt{2KQ}\,\Re F.
\end{equation}
\item {\rm (}Projected equation{\rm ).} $\Pi_-(|A|^2A)=2K\sqrt Q\,\overline F$, and
\begin{equation}
\label{eqA}
i\pa_tA=(L-M-2K)A+|A|^2A-2K\sqrt Q\,\overline F.
\end{equation}
Equivalently, $w(t):=e^{-i(M+2K)t}A(t)$ solves the perturbed Szeg\H{o} equation \eqref{alphaszego} with $\varkappa=2K$.
\end{enumerate}
\end{proposition}

\begin{proof}
\noindent{\bf Step 1:} {\it Structure.} The formula \eqref{defF} follows from $\lambda=1-|p|^2$:
$$\Theta+\bar p=\frac{B-\bar p+\bar p-|p|^2B}{1-pB}=\frac{\lambda B}{1-pB}.$$
Since $|B|\leq1$ on $\overline\Disc$ and $|p|<1$, the denominator does not vanish near $\overline\Disc$, and $F(t,0)=0$ because $B(0)=0$. On $\T$ we have $|B|=1$, hence $1-pB=B\,\overline{(B-\bar p)}$ and $|\Theta|=1$. Therefore $|A|\leq\sqrt{2K}+\sqrt Q$ and $\Theta\bar A=\sqrt{2K}\,\Theta+\sqrt Q=\tilde A$. Moreover, on $\T$, $|F|=\frac{\lambda}{|1-pB|}\leq\frac{\lambda}{1-|p|}=1+|p|$, with equality at the points where $pB=|p|$, which exist since $B$ maps $\T$ onto $\T$ by Lemma \ref{lemmablaschke}, item (1). Next, $A=\sqrt{2K}+\sqrt Q(F-\bar p)=c+\sqrt Q F$. Since $\Re p=\cos\vartheta$ and $|p|^2=1-\lambda$, \eqref{deftheta} gives
$$|c|^2=2K-2\sqrt{2KQ}\cos\vartheta+Q(1-\lambda)=M-Q\lambda,$$
$$|\tilde c|^2=Q-2\sqrt{2KQ}\cos\vartheta+2K(1-\lambda)=M-2K\lambda,$$
and \eqref{modc} follows, the limits being obtained with $\lambda=0$.

\noindent{\bf Step 2:} {\it Quadratic identity.} On $\T$, $\overline F=\frac{\lambda\overline B}{1-\bar p\overline B}$ and $|B|=1$, hence
\begin{align*}
\lambda+pF+\bar p\overline F&=\frac{\lambda}{|1-pB|^2}\Big(|1-pB|^2+pB\big(1-\bar p\overline B\big)+\bar p\overline B\big(1-pB\big)\Big)\\
&=\frac{\lambda\,(1-|p|^2)}{|1-pB|^2}=|F|^2.
\end{align*}
Since $F$ has only positive Fourier modes and $\overline F$ only negative ones, we deduce $\Pi_-(|F|^2)=\bar p\overline F$ and $\int_\T|F|^2\frac{dx}{2\pi}=\lambda$.

\noindent{\bf Step 3:} {\it Time derivative.} From \eqref{defF}, $\pa_tF=\frac{\lambda'}{\lambda}F+\frac{p'}{\lambda}F^2$ near $\overline\Disc$, and \eqref{derivp} gives $i\pa_tF=-i\mu F+\sqrt{2KQ}\,F^2$. For \eqref{derivF}, we work on $\T$: $\Theta=F-\bar p$ and $i\pa_t\bar p=i\overline{p'}=-\sqrt{2KQ}\lambda$ give
$$i\pa_t\Theta=-i\mu F+\sqrt{2KQ}\,F^2+\sqrt{2KQ}\,\lambda.$$
On the other hand, by \eqref{quadraticF} and \eqref{derivp},
\begin{align*}
2\sqrt{2KQ}\,(\Re F)\,\Theta&=\sqrt{2KQ}\,(F+\overline F)(F-\bar p)=\sqrt{2KQ}\,\big(F^2+(p-\bar p)F+\lambda\big)\\
&=\sqrt{2KQ}\,F^2-i\mu F+\sqrt{2KQ}\,\lambda.
\end{align*}

\noindent{\bf Step 4:} {\it Modulus.} Expanding $|c+\sqrt QF|^2$ with \eqref{quadraticF},
$$|A|^2=|c|^2+Q\lambda+\sqrt Q\,(\bar c+\sqrt Q\,p)F+\sqrt Q\,(c+\sqrt Q\,\bar p)\overline F.$$
Since $c+\sqrt Q\,\bar p=\sqrt{2K}$ and $|c|^2+Q\lambda=M$ by \eqref{modc}, we obtain \eqref{modA}.

\noindent{\bf Step 5:} {\it Projected equation.} By \eqref{modA} and $A=c+\sqrt QF$, the negative modes of $|A|^2A$ come from $\overline F A$ only, and Step 2 gives
$$\Pi_-(|A|^2A)=\sqrt{2KQ}\,\Pi_-\big(\overline F(c+\sqrt QF)\big)=\sqrt{2KQ}\,(c+\sqrt Q\,\bar p)\,\overline F=2K\sqrt Q\,\overline F.$$
We now check \eqref{eqA} on $\T$. By \eqref{derivF}, $i\pa_tA=\sqrt Q\,i\pa_t\Theta=2\sqrt{2K}\,Q\,(\Re F)\Theta$. By \eqref{actionL}, $LA=2Kc$, so that $(L-2K)A=-2K\sqrt QF$. Using \eqref{modA} and $A=\sqrt{2K}+\sqrt Q\Theta$, the right hand side of \eqref{eqA} is
$$-2K\sqrt QF+2\sqrt{2KQ}(\Re F)(\sqrt{2K}+\sqrt Q\Theta)-2K\sqrt Q\overline F=2\sqrt{2K}Q(\Re F)\Theta,$$
since $2\sqrt{2KQ}\sqrt{2K}\,\Re F=2K\sqrt Q(F+\overline F)$, and $2\sqrt{2KQ}\sqrt Q=2\sqrt{2K}\,Q$. This proves \eqref{eqA}. Finally, \eqref{eqA} and the first identity of item (5) mean that $i\pa_tA=\Pi_+\big[(L-M-2K)A+|A|^2A\big]$, which, by \eqref{actionL}, is equivalent to \eqref{alphaszego} for $w:=e^{-i(M+2K)t}A$.
\end{proof}

\begin{remark}
\label{remarkxu}
The function $w:=e^{-i(M+2K)t}A$ belongs to the family of rational trajectories of \eqref{alphaszego} studied by Xu \cite{Xu2014} for $B(z)=z$; the composition with a Blaschke product vanishing at the origin is established in \cite[Section 6]{Xu2015}. Proposition \ref{propalgebra} verifies the trajectory directly. By \eqref{modc}, \eqref{quadraticF}, and the identity $\|A\|_{L^4}^4=M^2+4KQ\lambda$ of Lemma \ref{lemmanorms}, item (3), below, the mass and the Hamiltonian $\frac14\|w\|_{L^4}^4+K|\hat w(0)|^2$ of \eqref{alphaszego} with $\varkappa=2K$ are equal to $M$ and to $\frac{M^2}{4}+KM$ along the trajectory.
\end{remark}


\subsection{The peaks}


As $t\to+\infty$, the $J$ zeros of $\Theta(t)$ converge to the boundary points $\zeta_j$, and $F$ splits into $J$ separated peaks. Each peak is a multiple of the function $g_j:=\frac{\lambda_jz}{1-p_jz}$, where $\overline{p_j}$ is the zero of $\Theta(t)$ near $\zeta_j$ and $\lambda_j:=1-|p_j|^2$. On $\T$, $|g_j|^2=\lambda_j\,\frac{1-|p_j|^2}{|z-\overline{p_j}|^2}$ is $\lambda_j$ times the Poisson kernel of the point $\overline{p_j}$, which concentrates near $\zeta_j$ at the scale $\lambda_j$, and the Fourier coefficients $\lambda_jp_j^{n-1}$ of $g_j$ are spread over the frequencies $n\lesssim\lambda_j^{-1}$. The complex number $\chi$ of \eqref{lambdaj} below describes how $\overline{p_j}$ approaches $\zeta_j$: since $\Re\chi=\frac12$ and $\Im\chi=-\frac12\cot\vartheta$, the zero $\overline{p_j}$ lies at the distance $\frac{\lambda_j}2(1+O(\lambda))$ from the circle, and its argument differs from that of $\zeta_j$ by $-\frac{\cot\vartheta}{2}\lambda_j+O(\lambda^2)$. We use the properties of finite Blaschke products recalled in Lemma \ref{lemmablaschke}.

\begin{lemma}[Peaks]
\label{lemmapeaks}
There exists $T_0\geq1$ such that, for all $t\geq T_0$, the following holds.
\begin{enumerate}
\item {\rm (}Zeros{\rm ).} The zeros of $\Theta(t,\cdot)$ are $J$ simple zeros $\overline{p_1(t)},\dots,\overline{p_J(t)}\in\Disc$, which are smooth functions of $t$ with $|p_j(t)|\geq\frac12$. The numbers $\lambda_j(t):=1-|p_j(t)|^2$ satisfy $|\frac{\lambda_j'}{\lambda_j}+\Lambda|\lesssim\lambda$ and
\begin{equation}
\label{lambdaj}
\Big|\lambda_j-\frac{\lambda}{|B'(\zeta_j)|}\Big|\lesssim\lambda^2,\qquad \Big|\frac{p_j}{\overline{\zeta_j}}-1+\chi\lambda_j\Big|\lesssim\lambda^2,\qquad \chi:=\frac{-ie^{i\vartheta}}{2\sin\vartheta}.
\end{equation}
\item {\rm (}Partial fractions{\rm ).} With $b_j(t):=\lambda(t)\lambda_j(t)^{-1}\overline{B'(\overline{p_j(t)})}^{\,-1}$,
\begin{equation}
\label{partialfractions}
F(t,z)=\sum_{j=1}^JG_j(t,z),\qquad G_j(t,z):=b_j(t)\frac{\lambda_j(t)z}{1-p_j(t)z}.
\end{equation}
The amplitudes satisfy $|b_j-e^{i\vartheta}\overline{\zeta_j}|+|b_j'|\lesssim\lambda$, $\big||b_j|-1\big|+\big|\overline{b_j}p_j-e^{-i\vartheta}\big|\lesssim\lambda$ and $\frac12\leq|b_j|\leq2$. Moreover $p_j'=-i\sqrt{2KQ}\,b_j\lambda_j$, and the complex rate $\mu_j:=-\frac{(b_j\lambda_j)'}{b_j\lambda_j}$ satisfies $|\mu_j-\Lambda|\lesssim\lambda$ and
\begin{equation}
\label{derivGj}
i\pa_tG_j=-i\mu_jG_j+\sqrt{2KQ}\,G_j^2.
\end{equation}
\item {\rm (}Local quadratic identity{\rm ).} On $\T$,
\begin{equation}
\label{quadraticGj}
|G_j|^2=|b_j|^2\lambda_j+\overline{b_j}p_j\,G_j+b_j\overline{p_j}\,\overline{G_j}.
\end{equation}
\item {\rm (}Separation{\rm ).} $\|G_j\|_\W\leq2|b_j|$, and for $j\neq l$,
$$G_jG_l=\frac{b_l\lambda_l}{p_j-p_l}\,G_j-\frac{b_j\lambda_j}{p_j-p_l}\,G_l,\qquad \|G_jG_l\|_\W+\|G_j\overline{G_l}\|_\W\lesssim\lambda.$$
\item {\rm (}Fourier coefficients{\rm ).} For $m\geq1$, $\widehat{G_j^m}(n)=(b_j\lambda_j)^m\binom{n-1}{m-1}p_j^{n-m}$ for $n\geq m$ and $\widehat{G_j^m}(n)=0$ otherwise. With $\kappa_*:=(8\max_j|B'(\zeta_j)|)^{-1}$,
\begin{equation}
\label{envelopeGj}
|\widehat{G_j^m}(n)|+|\pa_t\widehat{G_j^m}(n)|+|\pa_t^2\widehat{G_j^m}(n)|\lesssim_m\lambda\,e^{-\kappa_*\lambda n},\qquad n\geq0.
\end{equation}
Consequently, the functions $F$, $\bar zF$ and $F^2$, and their first two time derivatives, have Fourier coefficients $\lesssim\lambda e^{-\kappa_*\lambda n}$, $n\geq0$.
\end{enumerate}
\end{lemma}

\begin{proof}
We use the function $\varrho$ of Lemma \ref{lemmablaschke}, item (3), which is smooth and positive near $\T$ and satisfies $1-|B(z)|^2=(1-|z|^2)\varrho(z)$ and $\varrho(\zeta)=|B'(\zeta)|$ for $\zeta\in\T$.

\noindent{\bf Step 1:} {\it Zeros.} The zeros of $\Theta(t)$ are the solutions of $B(z)=\overline{p(t)}$. Since $B(\zeta_j)=e^{i\vartheta}$, $\overline{p(t)}\to e^{i\vartheta}$ and $B'(\zeta_j)\neq0$, the inverse function theorem provides, for $t$ large, a solution of $B(z)=\overline{p(t)}$ near $\zeta_j$, smooth in $t$, at a distance $\lesssim|p(t)-e^{-i\vartheta}|\lesssim\lambda(t)$ from $\zeta_j$ by Lemma \ref{lemmaplambda}, item (3); we denote it by $\overline{p_j(t)}$. Since $|B(\overline{p_j})|=|p|<1$ and $|B|\geq1$ outside $\Disc$, $\overline{p_j}\in\Disc$. The $J$ points $\overline{p_j}$ are distinct, and $B(z)=\bar p$ has exactly $J$ solutions in $\Disc$ counted with multiplicity, by Lemma \ref{lemmablaschke}, item (2), so these are all the zeros, and they are simple. Since $|p_j|\to1$, we have $|p_j|\geq\frac12$ after increasing $T_0$. Then
$$\lambda=1-|B(\overline{p_j})|^2=\lambda_j\,\varrho(\overline{p_j}),$$
and $|\varrho(\overline{p_j})-|B'(\zeta_j)||\lesssim|\overline{p_j}-\zeta_j|\lesssim\lambda$ gives the first bound of \eqref{lambdaj}. In particular $\lambda_j\geq\frac{\lambda}{2|B'(\zeta_j)|}$ for $t\geq T_0$.

\noindent{\bf Step 2:} {\it Partial fractions.} For $t\geq T_0$ we have $p\neq0$. Since $F=\lambda\,(B^{-1}-p)^{-1}$, $F$ is bounded near the poles of $B$ and near infinity, where $B$ is infinite by Lemma \ref{lemmablaschke}, item (2). Hence the poles of $F$ in $\C$ are the solutions of $B(z)=1/p$. By the reflection identity $B(1/\bar z)=1/\overline{B(z)}$ of Lemma \ref{lemmablaschke}, item (2), these are the $J$ points $1/p_j$, and they are simple. Near $z=1/p_j$, $1-pB(z)=-pB'(1/p_j)(z-1/p_j)+O(|z-1/p_j|^2)$, and the derivative of the reflection identity at the point $\overline{p_j}$ gives $B'(1/p_j)=p_j^2\,\overline{B'(\overline{p_j})}\,p^{-2}$. Hence the residue of $F$ at $1/p_j$ is
$$\frac{\lambda B(1/p_j)}{-pB'(1/p_j)}=-\frac{\lambda}{p_j^2\,\overline{B'(\overline{p_j})}},$$
which is also the residue of $b_j\frac{\lambda_jz}{1-p_jz}$ at $1/p_j$, by the definition of $b_j$. The difference $F-\sum_jG_j$ is therefore a rational function without poles, bounded at infinity, hence a constant, which vanishes since it vanishes at $z=0$. Next, $b_j=\varrho(\overline{p_j})\,\overline{B'(\overline{p_j})}^{\,-1}$ is a smooth function of $\overline{p_j}$ near $\zeta_j$, and at $\overline{p_j}=\zeta_j$ it equals $|B'(\zeta_j)|\,\overline{B'(\zeta_j)}^{\,-1}=e^{i\vartheta}\overline{\zeta_j}$, since $\zeta_jB'(\zeta_j)=|B'(\zeta_j)|B(\zeta_j)=|B'(\zeta_j)|e^{i\vartheta}$ by Lemma \ref{lemmablaschke}, item (1). This gives $|b_j-e^{i\vartheta}\overline{\zeta_j}|\lesssim\lambda$, and, with $|p_j-\overline{\zeta_j}|\lesssim\lambda$, the bounds on $|b_j|$ and $\overline{b_j}p_j$, after increasing $T_0$.

\noindent{\bf Step 3:} {\it Dynamics.} Differentiating $B(\overline{p_j})=\bar p$ and using \eqref{derivp}, $B'(\overline{p_j})\,\overline{p_j}'=i\sqrt{2KQ}\lambda$, hence $p_j'=-i\sqrt{2KQ}\lambda\,\overline{B'(\overline{p_j})}^{\,-1}=-i\sqrt{2KQ}\,b_j\lambda_j$. In particular $|p_j'|\lesssim\lambda$, and differentiating once more, $p_j''=-i\sqrt{2KQ}(b_j\lambda_j)'=i\sqrt{2KQ}\,\mu_jb_j\lambda_j$. Since $b_j$ and $\varrho(\overline{p_j})$ are smooth functions of $\overline{p_j}$, we get $|b_j'|\lesssim\lambda$ and, from $\lambda_j=\lambda/\varrho(\overline{p_j})$, $|\frac{\lambda_j'}{\lambda_j}+\mu|\lesssim\lambda$. Hence $|\mu_j-\mu|\lesssim\lambda$, and Lemma \ref{lemmaplambda}, item (3), gives $|\mu_j-\Lambda|\lesssim\lambda$ and $|\frac{\lambda_j'}{\lambda_j}+\Lambda|\lesssim\lambda$. In particular $\mu_j$ is bounded, and $|p_j''|\lesssim\lambda$. Differentiating twice the smooth functions $b_j$ and $\lambda_j=\lambda/\varrho(\overline{p_j})$ of $\overline{p_j}$ and $\lambda$, and using $|p_j'|+|p_j''|\lesssim\lambda$ and $\lambda''=(\mu^2-\mu')\lambda=O(\lambda)$ by Lemma \ref{lemmaplambda}, item (1), we obtain $|b_j''|+|\lambda_j''|\lesssim\lambda$; these bounds are used in Step 6. Differentiating $G_j$ and using $p_j'=-i\sqrt{2KQ}b_j\lambda_j$ yields
$$\pa_tG_j=\frac{(b_j\lambda_j)'}{b_j\lambda_j}\,G_j+b_j\lambda_j\frac{p_j'z^2}{(1-p_jz)^2}=-\mu_jG_j-i\sqrt{2KQ}\,G_j^2,$$
which is \eqref{derivGj}. For the second bound of \eqref{lambdaj}, $p_j(\tau)\to\overline{\zeta_j}$ as $\tau\to+\infty$, hence
$$\overline{\zeta_j}-p_j(t)=\int_t^{+\infty}p_j'(\tau)\,d\tau=-i\sqrt{2KQ}\int_t^{+\infty}b_j\lambda_j\,d\tau.$$
Integrating $\frac{\lambda_j'}{\lambda_j}=-\Lambda+O(\lambda)$ on $[t,\tau]$ and using $\int_t^{+\infty}\lambda\,d\tau'\leq\frac4\Lambda\lambda(t)$ (Lemma \ref{lemmaplambda}, item (2)), we get $\lambda_j(\tau)=\lambda_j(t)e^{-\Lambda(\tau-t)}(1+O(\lambda(t)))$ for $\tau\geq t$. Since $b_j=e^{i\vartheta}\overline{\zeta_j}+O(\lambda)$, we obtain $\overline{\zeta_j}-p_j=-i\frac{\sqrt{2KQ}}{\Lambda}e^{i\vartheta}\overline{\zeta_j}\lambda_j+O(\lambda^2)$, and $\frac{\sqrt{2KQ}}{\Lambda}=\frac{1}{2\sin\vartheta}$ gives the claim.

\noindent{\bf Step 4:} {\it Local quadratic identity.} On $\T$, $|z|=1$ and $\lambda_j=1-|p_j|^2$, hence, with $g_j:=\frac{G_j}{b_j}=\frac{\lambda_jz}{1-p_jz}$,
\begin{align*}
\lambda_j+p_jg_j+\overline{p_j}\,\overline{g_j}&=\frac{\lambda_j}{|1-p_jz|^2}\Big(|1-p_jz|^2+p_jz\big(1-\overline{p_jz}\big)+\overline{p_jz}\big(1-p_jz\big)\Big)\\
&=\frac{\lambda_j(1-|p_j|^2)}{|1-p_jz|^2}=|g_j|^2.
\end{align*}
Multiplying by $|b_j|^2$ gives \eqref{quadraticGj}.

\noindent{\bf Step 5:} {\it Separation.} Since $\lambda_j=(1-|p_j|)(1+|p_j|)$, we have $\|G_j\|_\W=|b_j|\lambda_j\sum_{n\geq1}|p_j|^{n-1}=|b_j|(1+|p_j|)\leq2|b_j|$. For $j\neq l$, the partial fraction decomposition
$$\frac{z^2}{(1-p_jz)(1-p_lz)}=\frac{1}{p_j-p_l}\Big(\frac{z}{1-p_jz}-\frac{z}{1-p_lz}\Big)$$
gives the formula for $G_jG_l$. Moreover, on $\T$,
$$G_j\overline{G_l}=\frac{b_j\overline{b_l}\lambda_j\lambda_l}{(1-p_jz)(1-\overline{p_lz})},\qquad \widehat{G_j\overline{G_l}}(n)=\frac{b_j\overline{b_l}\lambda_j\lambda_l}{1-p_j\overline{p_l}}\times\begin{cases}p_j^n,&n\geq0,\\ \overline{p_l}^{\,|n|},&n<0.\end{cases}$$
Since $p_j-p_l\to\overline{\zeta_j}-\overline{\zeta_l}\neq0$ and $1-p_j\overline{p_l}\to1-\overline{\zeta_j}\zeta_l\neq0$, and since $\sum_{n\geq0}|p_j|^n\leq\frac2{\lambda_j}$, both Wiener norms are $O(\lambda)$.

\noindent{\bf Step 6:} {\it Fourier coefficients.} The coefficient of $z^n$ in $z^m(1-p_jz)^{-m}$ is $\binom{n-1}{m-1}p_j^{n-m}$. Since $|p_j|^{n-m}=(1-\lambda_j)^{\frac{n-m}2}\leq e^{\frac m2}e^{-\frac{\lambda_j}2n}$ and $\lambda_j\geq\frac{\lambda}{2|B'(\zeta_j)|}$, we get $|\widehat{G_j^m}(n)|\lesssim_m\lambda(\lambda n)^{m-1}e^{-2\kappa_*\lambda n}$. The time derivative of $(b_j\lambda_j)^mp_j^{n-m}$ equals this quantity times the factor $-m\mu_j+(n-m)\frac{p_j'}{p_j}=O(1+\lambda n)$, and its second time derivative equals this quantity times the sum of the square of this factor and of its time derivative $-m\mu_j'+(n-m)\big(\frac{p_j''}{p_j}-\frac{p_j'^2}{p_j^2}\big)$, which is $O((1+\lambda n)^2)$. Indeed, $|p_j|\geq\frac12$, $|p_j'|+|p_j''|\lesssim\lambda$, and $\mu_j'=-\frac{\lambda_j''}{\lambda_j}+\frac{\lambda_j'^2}{\lambda_j^2}-\frac{b_j''}{b_j}+\frac{b_j'^2}{b_j^2}=O(1)$ by Step 3, since $\lambda_j\gtrsim\lambda$ and $|b_j|\geq\frac12$. Since $(1+\lambda n)^{m+1}e^{-\kappa_*\lambda n}\lesssim_m1$, this gives \eqref{envelopeGj}.
\end{proof}


\subsection{Norms and conservation laws of the profiles}


The concentration of the peaks makes the Sobolev norms above $H^{\frac12}$ grow, while the Wiener norm stays bounded and the conservation laws are constant.

\begin{lemma}[Norms and conservation laws of the profiles]
\label{lemmanorms}
The following holds.
\begin{enumerate}
\item {\rm (}Wiener norm{\rm ).} For $t\geq T_0$, the functions $F(t)$, $A(t)$ and $\tilde A(t)$ are bounded in $\W$, uniformly in $t$.
\item {\rm (}Sobolev norms{\rm ).} For every $s\geq0$, as $t\to+\infty$,
\begin{equation}
\label{normsA}
\sum_{n\geq1}\la Nn-K\ra^{2s}|\hat A(t,n)|^2=N^{2s}Q\,\Gamma(2s+1)\sum_{j=1}^J|B'(\zeta_j)|^{2s-1}\,\lambda(t)^{1-2s}\,(1+o(1)).
\end{equation}
Moreover, for $s>0$ and $\sigma\in\R$, the same holds with the weight $\la Nn-K\ra^{2s}$ replaced by $\la Nn-K\ra^{2s}\log^{-2\sigma}(e+|Nn-K|)$ and the right hand side multiplied by $\ell(t)^{-2\sigma}$.
\item {\rm (}Conservation laws{\rm ).} For every $t\in\R$, the function $U(t)$ of \eqref{defU} satisfies $\|U(t)\|_{L^4}^4=\|A(t)\|_{L^4}^4=M^2+4KQ\lambda(t)$ and
\begin{equation}
\label{invariantsU}
\begin{gathered}
M(U(t))=M,\qquad P(U(t))=NJQ-KM,\\
E(U(t))=\frac{NJQ+KM}2+\frac{M^2}{4}.
\end{gathered}
\end{equation}
\end{enumerate}
\end{lemma}

\begin{proof}
The Wiener bound follows from \eqref{partialfractions} and $\|G_j\|_\W\leq2|b_j|$, by Lemma \ref{lemmapeaks}, item (4), since $A=c+\sqrt QF$ and $\tilde A=\tilde c+\sqrt{2K}F$.

For \eqref{normsA}, we have $\hat A(n)=\sqrt Q\sum_jb_j\lambda_jp_j^{n-1}$ for $n\geq1$, hence
$$\sum_{n\geq1}\la Nn-K\ra^{2s}|\hat A(n)|^2=Q\sum_{j,l}b_j\overline{b_l}\lambda_j\lambda_l\sum_{n\geq1}\la Nn-K\ra^{2s}(p_j\overline{p_l})^{n-1}.$$
For $j=l$, $p_j\overline{p_l}=1-\lambda_j$, and the Riemann sum $\lambda_j\sum_{n\geq1}(\lambda_jn)^{2s}(1-\lambda_j)^{n-1}$ converges to $\int_0^{+\infty}\rho^{2s}e^{-\rho}d\rho=\Gamma(2s+1)$; since $\la Nn-K\ra^{2s}=(Nn)^{2s}(1+O(n^{-1}))$ and, for every fixed $\delta>0$, the frequencies $n\leq\delta\lambda_j^{-1}$ contribute $O(\delta^{2s+1}\lambda_j^{-2s-1})$, we obtain
$$\lambda_j^2\sum_{n\geq1}\la Nn-K\ra^{2s}(1-\lambda_j)^{n-1}=N^{2s}\Gamma(2s+1)\lambda_j^{1-2s}(1+o(1)),$$
and \eqref{lambdaj}, $|b_j|\to1$ give the right hand side of \eqref{normsA}. For $j\neq l$, let $q:=|p_j\overline{p_l}|$, so that $1-q\geq1-|p_j|\geq\frac{\lambda_j}2\gtrsim\lambda$, while $p_j\overline{p_l}/q\to\overline{\zeta_j}\zeta_l\neq1$. Hence the partial sums of the geometric series $\sum_n(p_j\overline{p_l}/q)^{n-1}$ are bounded, and the total variation of the sequence $\big(\la Nn-K\ra^{2s}q^{n-1}\big)_{n\geq1}$ is bounded by
$$\sum_{n\geq1}\Big(\big|\la N(n+1)-K\ra^{2s}-\la Nn-K\ra^{2s}\big|q^n+\la Nn-K\ra^{2s}q^{n-1}(1-q)\Big)\lesssim(1-q)^{-2s}.$$
By summation by parts, the cross term is $O(\lambda^2\lambda^{-2s})=O(\lambda\cdot\lambda^{1-2s})$, which is negligible.

With the additional weight $\log^{-2\sigma}(e+|Nn-K|)$, $s>0$, the same proof applies, with the following changes. Let $\delta\in(0,1)$ be fixed. In the diagonal terms, $\log(e+|Nn-K|)=\ell(t)(1+o(1))$ uniformly for $\delta\lambda_j^{-1}\leq n\leq\delta^{-1}\lambda_j^{-1}$. For $n\geq\lambda^{-\frac13}$ and $t$ large, we have $\frac\ell4\leq\log(e+|Nn-K|)\leq\ell+C+\log(1+\lambda_jn)$. Hence, relative to the main term, the frequencies $\lambda^{-\frac13}\leq n\leq\delta\lambda_j^{-1}$ contribute $O(\delta^{2s+1})$, and the frequencies $n\geq\delta^{-1}\lambda_j^{-1}$ contribute $O\big(\int_{\delta^{-1}}^{+\infty}\rho^{2s}(1+\log(1+\rho))^{2|\sigma|}e^{-\rho}d\rho\big)$, for both signs of $\sigma$, while the frequencies $n\leq\lambda^{-\frac13}$ contribute $O(\lambda^{2-\frac{2s+1}3}\ell^{2|\sigma|})$, which is negligible. In the cross terms, the total variation of the weighted sequence acquires at most a factor $\ell^{2|\sigma|}$, up to a constant, which is again negligible.

We now prove \eqref{invariantsU}. By Lemma \ref{lemmacomoving}, \eqref{modc} and \eqref{quadraticF},
$$M(U)=|c|^2+Q\|F\|_{L^2}^2=M-Q\lambda+Q\lambda=M.$$
By Lemma \ref{lemmablaschke}, item (4), $\sum_nn|\hat F(n)|^2=\sum_nn|\hat\Theta(n)|^2=J$, since $\Theta(t)$ is a Blaschke product of degree $J$. By \eqref{comovingfourier},
$$P(U)=\sum_{n\geq0}(Nn-K)|\hat A(n)|^2=NQ\sum_{n\geq1}n|\hat F(n)|^2-K\|A\|_{L^2}^2=NJQ-KM.$$
For the energy, $|Nn-K|=Nn-K$ for $n\geq1$ and $|{-K}|=-K+2K$, hence
$$\sum_m|m||\hat U(m)|^2=P(U)+2K|c|^2=NJQ+KM-2KQ\lambda.$$
Since $\hat F(0)=0$ and $\int_\T(F+\overline F)^2\frac{dx}{2\pi}=2\|F\|_{L^2}^2=2\lambda$ by \eqref{quadraticF}, Lemma \ref{lemmacomoving} and \eqref{modA} give
$$\int_\T|U|^4\frac{dx}{2\pi}=\int_\T\big(M+\sqrt{2KQ}(F+\overline F)\big)^2\frac{dx}{2\pi}=M^2+4KQ\lambda.$$
Therefore $E(U)=\frac12(NJQ+KM-2KQ\lambda)+\frac14(M^2+4KQ\lambda)$, and the terms in $\lambda$ cancel.
\end{proof}


\section{The linearized flow in coordinates adapted to the concentration}
\label{sectionlinear}


We look for a solution of \eqref{eqv} of the form $v=A+\eps$. By \eqref{eqA}, $v$ solves \eqref{eqv} if and only if
\begin{equation}
\label{eqeps}
i\pa_t\eps=\mathcal L_A\eps+2K\sqrt Q\,\overline F+\mathcal N_A(\eps),
\end{equation}
where the linearized operator and the nonlinear terms are
\begin{align}
\mathcal L_A\eps&:=(L-M-2K)\eps+2|A|^2\eps+A^2\overline\eps\notag\\
&=L\eps+(M-2K)\eps+4\sqrt{2KQ}\,(\Re F)\,\eps+A^2\overline\eps,\label{defLA}\\
\mathcal N_A(\eps)&:=2A|\eps|^2+\overline A\eps^2+|\eps|^2\eps.\label{defNA}
\end{align}
The second expression of $\mathcal L_A$ follows from \eqref{modA}. Two features of \eqref{eqeps} prevent a direct perturbative treatment. The forcing term $2K\sqrt Q\,\overline F$ is not small: its Wiener norm stays of order one, since $\|F\|_\W\geq\|F\|_{L^\infty}=1+|p|$ by Proposition \ref{propalgebra}, item (1). And the linearized operator $\mathcal L_A$ contains the terms $4\sqrt{2KQ}(\Re F)\eps$ and $A^2\overline\eps$, which are not small either. This section is devoted to the description of the linearized flow $i\pa_t\eps=\mathcal L_A\eps$ in coordinates adapted to the concentration, in which it becomes a triangular system with an explicit principal part.

We first introduce a norm adapted to the concentration and an explicit correction which removes the forcing at leading order (Subsection \ref{subsectionnorm}, Lemmas \ref{lemmasmoothing} and \ref{lemmah1}). We then decompose $\eps$ into a finite-dimensional part, a transverse part, and a negative-frequency part (Subsection \ref{subsectioncoordinates}), and we compute the linearized equation in these coordinates (Proposition \ref{proplinearized}). We study its finite-dimensional block (Lemma \ref{lemmafiniteblock}), and we remove the concentrated part of its infinite-dimensional block by a normal form transformation (Lemma \ref{lemmatransformation} and Proposition \ref{propnormalform}). The remaining constant-coefficient propagator is computed explicitly in Lemma \ref{lemmahardygroup}.


\subsection{A norm adapted to the concentration}
\label{subsectionnorm}


By Lemma \ref{lemmanorms}, the profile $A(t)$ is bounded in the Wiener algebra, while its $H^s$ norm grows like $\lambda(t)^{\frac12-s}$ for $s>\frac12$. In order to treat the profile as a bounded object, we measure functions with the time-dependent norm
\begin{equation}
\label{defXst}
\|f\|_{X^s_t}:=\|f\|_\W+\lambda(t)^{s-\frac12}\|f\|_{H^s},\qquad \frac12<s<\frac32,
\end{equation}
and we denote by $X^s_t$ the space $H^s$ endowed with this norm. The weight $\lambda^{s-\frac12}$ normalizes the $H^s$ norm of a function concentrated at the scale $\lambda$, in the sense of \eqref{envelope} below, so that the peaks $G_j$ are bounded in $X^s_t$ (Lemma \ref{lemmasmoothing} below). The Wiener norm controls the values of the coordinates introduced in Subsection \ref{subsectioncoordinates}, and it is the part of the norm which only loses a polynomial factor when the error is propagated backwards in time.

The negative frequencies will be described through the reflection
\begin{equation}
\label{defflat}
g^\flat:=\bar z\,\overline{\Pi_-g}\in L^2_+,\qquad \widehat{g^\flat}(n)=\overline{\hat g(-n-1)},\quad n\geq0,
\end{equation}
which maps the negative modes of $g$ onto the nonnegative ones. Equivalently, $\Pi_-g=\bar z\,\overline{g^\flat}$, and $(\bar z\bar q)^\flat=q$ for $q\in L^2_+$. The reflection is real-linear, with $(\gamma g)^\flat=\bar\gamma g^\flat$ for $\gamma\in\C$. By \eqref{actionL}, the reflection transforms the multiplier $L$ into the positive operator
\begin{equation}
\label{defDcal}
\begin{gathered}
\D:=N(D+1)+K\ \ \mbox{on}\ \ L^2_+,\qquad \widehat{\D f}(n)=\D_n\hat f(n),\\
\D_n:=N(n+1)+K,\qquad (Lg)^\flat=2\D g^\flat.
\end{gathered}
\end{equation}
The operator $\D^{-1}$ gains one derivative, and this is the mechanism which controls the negative frequencies: they rotate at the frequencies $2\D_n\geq2(N+K)$ in the co-moving frame, so that their response to a forcing is smaller by one power of the frequency. We shall apply $\D^{-1}$ to products $\varpi f$, where $\varpi\in L^2_+$ is concentrated at the scale $\lambda$, in the sense that, for some constants $C_0,\kappa_0>0$,
\begin{equation}
\label{envelope}
|\hat\varpi(n)|\leq C_0\,\lambda(t)\,e^{-\kappa_0\lambda(t)n},\qquad n\geq0.
\end{equation}
By Lemma \ref{lemmapeaks}, item (5), the functions $F$, $\bar zF$, $F^2$ and their first two time derivatives satisfy \eqref{envelope} with $\kappa_0=\kappa_*$.

\begin{lemma}[Products and smoothing]
\label{lemmasmoothing}
Let $\frac12<s<\frac32$ and $t\geq T_0$.
\begin{enumerate}
\item {\rm (}Products{\rm ).} For $f,g\in H^s$ and $q\in H^s_+$,
\begin{gather*}
\|fg\|_{X^s_t}\leq2^s\|f\|_{X^s_t}\|g\|_{X^s_t},\qquad \|\overline f\|_{X^s_t}=\|f\|_{X^s_t},\\
\|\Pi_\pm f\|_{X^s_t}+\|f^\flat\|_{X^s_t}\leq2\|f\|_{X^s_t},\qquad \|\bar z\bar q\|_{X^s_t}\leq2^s\|q\|_{X^s_t}.
\end{gather*}
\item {\rm (}Concentrated symbols{\rm ).} Let $\varpi\in L^2_+$ satisfy \eqref{envelope}. Then $\|\varpi\|_{X^s_t}\lesssim1$, $\|\D^{-1}\varpi\|_{X^s_t}\lesssim\lambda\ell$, and for every $f\in H^s_+$,
\begin{equation}
\label{smoothing}
\|\D^{-1}(\varpi f)\|_{X^s_t}+\|\mathbf T_{\overline\varpi}\D^{-1}f\|_{X^s_t}\lesssim\lambda\ell\,\|f\|_{X^s_t}.
\end{equation}
If moreover $\hat\varpi(0)=0$, then
$$\|D^{-1}(\varpi+\overline\varpi)\|_{X^s_t}\lesssim\lambda\ell,\qquad \|\mathbf T_{D^{-1}(\varpi+\overline\varpi)}f\|_{X^s_t}\lesssim\lambda\ell\,\|f\|_{X^s_t}.$$
\item {\rm (}Wiener and $L^2$ versions{\rm ).} The estimates of item {\rm (2)} for the three operators $f\mapsto\D^{-1}(\varpi f)$, $f\mapsto \mathbf T_{\overline\varpi}\D^{-1}f$ and $f\mapsto \mathbf T_{D^{-1}(\varpi+\overline\varpi)}f$ also hold with $\|\cdot\|_{X^s_t}$ replaced by $\|\cdot\|_\W$, or by $\|\cdot\|_{L^2}$, on both sides.
\end{enumerate}
In items {\rm (2)} and {\rm (3)}, the implicit constants are of the form $C_0\,C(s,\kappa_0)$.
\end{lemma}

The smoothing estimate \eqref{smoothing} expresses the gain of one derivative: the factor $\frac{1}{n+1}$ of $\D^{-1}$ is small on the frequencies $n\gg\lambda^{-1}$ of $f$, while on the frequencies $n\lesssim\lambda^{-1}$ the concentrated function has the small coefficients $O(\lambda)$, and the sum over these frequencies produces the factor $\ell$.

\begin{proof}
\noindent{\bf Step 1:} {\it Products.} Since $\la n\ra\leq\la n-m\ra+\la m\ra$, we have $\la n\ra^s\leq2^s(\la n-m\ra^s+\la m\ra^s)$, and Young's inequality for convolutions gives
$$\|fg\|_{H^s}\leq2^s\big(\|f\|_\W\|g\|_{H^s}+\|f\|_{H^s}\|g\|_\W\big),\qquad \|fg\|_\W\leq\|f\|_\W\|g\|_\W.$$
Multiplying the first inequality by $\lambda^{s-\frac12}$ and adding the second one gives the product estimate. The identity for $\overline f$ is obvious, and the bounds for $\Pi_\pm f$ and $f^\flat$ follow from $\widehat{f^\flat}(n)=\overline{\hat f(-n-1)}$ and $\la n\ra\leq\la n+1\ra$ for $n\geq0$. Finally, $\bar z\bar q$ has the coefficient $\overline{\hat q(n)}$ at the mode $-n-1$, and $\la n+1\ra\leq2\la n\ra$.

\noindent{\bf Step 2:} {\it Sums.} Comparing the sums with integrals in the variable $\lambda n$, the bound \eqref{envelope} implies
\begin{equation}
\label{sumsenvelope}
\begin{gathered}
\sum_{n\geq0}|\hat\varpi(n)|\lesssim1,\qquad \sum_{n\geq0}\la n\ra^{2s}|\hat\varpi(n)|^2\lesssim\lambda^{1-2s},\\
\sum_{n\geq0}\frac{|\hat\varpi(n)|}{n+1}\lesssim\lambda\ell,\qquad \sum_{n\geq0}\la n\ra^{2s-2}|\hat\varpi(n)|^2\lesssim\lambda^{3-2s}.
\end{gathered}
\end{equation}
The logarithm in the third sum comes from the range $0\leq n\leq\lambda^{-1}$, where $|\hat\varpi(n)|\lesssim\lambda$, and the last sum is compared with $\lambda^{3-2s}\int_0^{+\infty}x^{2s-2}e^{-2\kappa_0x}\,dx$, which converges at $x=0$ because $2s-2>-1$. The first two sums give $\|\varpi\|_{X^s_t}\lesssim1$. Since $N(n+1)+K\geq n+1\geq\la n\ra$, the last two sums give $\|\D^{-1}\varpi\|_\W\lesssim\lambda\ell$ and $\lambda^{s-\frac12}\|\D^{-1}\varpi\|_{H^s}\lesssim\lambda^{s-\frac12}\lambda^{\frac32-s}=\lambda$.

\noindent{\bf Step 3:} {\it Multiplication followed by $\D^{-1}$.} For $n\geq0$,
$$|\widehat{\D^{-1}(\varpi f)}(n)|\leq\frac{1}{n+1}\sum_{m=0}^n|\hat\varpi(n-m)||\hat f(m)|\leq\sum_{m=0}^n\frac{|\hat\varpi(n-m)|}{n-m+1}|\hat f(m)|.$$
By Young's inequality and the third sum of \eqref{sumsenvelope}, the Wiener and $L^2$ norms of $\D^{-1}(\varpi f)$ are bounded by $\lambda\ell\|f\|_\W$ and $\lambda\ell\|f\|_{L^2}$. For the $H^s$ norm, we split $\la n\ra^s\leq2^s(\la m\ra^s+\la n-m\ra^s)$. The first part is bounded by the convolution of the sequences $\frac{|\hat\varpi(j)|}{j+1}$ and $\la m\ra^s|\hat f(m)|$; since $\frac{\la j\ra^s}{j+1}\leq\la j\ra^{s-1}$, the second part is bounded by the convolution of $\la j\ra^{s-1}|\hat\varpi(j)|$ and $|\hat f(m)|$. By Young's inequality and \eqref{sumsenvelope},
$$\|\D^{-1}(\varpi f)\|_{H^s}\lesssim\lambda\ell\,\|f\|_{H^s}+\lambda^{\frac32-s}\|f\|_\W.$$
Multiplying by $\lambda^{s-\frac12}$ gives the bound \eqref{smoothing} for the first term.

\noindent{\bf Step 4:} {\it Conjugate Toeplitz operators.} For $g\in L^2_+$ and $n\geq0$, $\widehat{\mathbf T_{\overline\varpi}g}(n)=\sum_{j\geq0}\overline{\hat\varpi(j)}\hat g(n+j)$. With $g:=\D^{-1}f$, we have $|\hat g(n+j)|\leq\frac{|\hat f(n+j)|}{j+1}$, and therefore
$$|\widehat{\mathbf T_{\overline\varpi}\D^{-1}f}(n)|\leq\sum_{j\geq0}\frac{|\hat\varpi(j)|}{j+1}\,|\hat f(n+j)|.$$
Since $\la n\ra\leq\la n+j\ra$, each shifted sequence $(\hat f(n+j))_{n\geq0}$ has Wiener, $L^2$ and $H^s$ norms bounded by those of $f$, and the third sum of \eqref{sumsenvelope} gives the bound \eqref{smoothing} for the second term, as well as the corresponding bounds in $\W$ and $L^2$.

\noindent{\bf Step 5:} {\it Primitives.} If $\hat\varpi(0)=0$, the symbol $b:=D^{-1}(\varpi+\overline\varpi)$ has the coefficients $\hat b(n)=\frac{\hat\varpi(n)}n$ and $\hat b(-n)=-\frac{\overline{\hat\varpi(n)}}{n}$ for $n\geq1$. By \eqref{sumsenvelope}, $\|b\|_\W\lesssim\lambda\ell$ and $\|b\|_{H^s}\lesssim\lambda^{\frac32-s}$, hence $\|b\|_{X^s_t}\lesssim\lambda\ell$, and Step 1 gives $\|\mathbf T_bf\|_{X^s_t}\leq\|bf\|_{X^s_t}\lesssim\lambda\ell\|f\|_{X^s_t}$. Finally, $\|\mathbf T_bf\|_\W\leq\|b\|_\W\|f\|_\W$ and $\|\mathbf T_bf\|_{L^2}\leq\|b\|_{L^\infty}\|f\|_{L^2}\leq\|b\|_\W\|f\|_{L^2}$.
\end{proof}

We can now remove the forcing term of \eqref{eqeps} at leading order. This term has only negative Fourier modes, and its reflection is $(2K\sqrt Q\,\overline F)^\flat=2K\sqrt Q\,\bar zF$, where $\bar zF\in L^2_+$ since $F(0)=0$. We look for a correction of $\eps$ of the form $\bar z\bar h$, with $h\in\mathcal C^1([T_0,+\infty),H^3_+)$, which has only negative modes. By \eqref{defLA} and \eqref{defDcal}, and since $\big(i\pa_t(\bar z\bar h)\big)^\flat=-i\pa_th$ and $\Pi_-\big((\Re F)\bar z\bar h\big)=\bar z\,\overline{\mathbf T_{\Re F}h}$, the function $\Re F$ being real, while $A^2\,\overline{\bar z\bar h}=A^2zh$ has no negative modes, the reflection of the negative part of $i\pa_t(\bar z\bar h)-\mathcal L_A(\bar z\bar h)-2K\sqrt Q\,\overline F$ is the residual forcing
\begin{equation}
\label{defrh}
r_h:=-i\pa_th-(2\D+M-2K)h-4\sqrt{2KQ}\,\mathbf T_{\Re F}h-2K\sqrt Q\,\bar zF.
\end{equation}
The operator $2\D$ is large, since the negative modes rotate at the frequencies $2\D_n\geq2(N+K)$, and we choose $h$ so that the term $2\D h$ balances the forcing: we let
\begin{equation}
\label{defh1}
h_1:=-K\sqrt Q\,\D^{-1}(\bar zF),\qquad\mbox{so that}\qquad 2\D h_1+2K\sqrt Q\,\bar zF=0.
\end{equation}
By \eqref{defDcal}, this means $L(\bar z\bar h_1)=-2K\sqrt Q\,\overline F$: the function $\bar z\bar h_1$ is the response of the negative modes to the forcing, and it balances the forcing with the large part $L$ of the linearized operator. The residual forcing of the first correction is
\begin{equation}
\label{defr1}
r_1:=r_{h_1}=-i\pa_th_1-(M-2K)h_1-4\sqrt{2KQ}\,\mathbf T_{\Re F}h_1.
\end{equation}

\begin{lemma}[The first correction]
\label{lemmah1}
For $t\geq T_0$ and $\frac12<s<\frac32$, the following holds.
\begin{enumerate}
\item {\rm (}The correction $h_1${\rm ).} $\hat h_1(n)=-K\sqrt Q\,\frac{\hat F(n+1)}{N(n+1)+K}$, and
\begin{equation}
\label{envelopeh1}
|\hat h_1(n)|+|\pa_t\hat h_1(n)|+|\pa_t^2\hat h_1(n)|\lesssim\frac{\lambda}{n+1}\,e^{-\kappa_*\lambda n},\qquad n\geq0.
\end{equation}
Consequently, $|\hat h_1(n)|\lesssim\min(\lambda\ell,\la n\ra^{-2})$, $h_1\in\mathcal C^1([T_0,+\infty),H^r_+)$ for every $r\geq0$, and
\begin{equation}
\label{boundh1}
\|h_1\|_{X^s_t}\lesssim\lambda\ell.
\end{equation}
\item {\rm (}The residual forcing $r_1${\rm ).} $\|r_1\|_{X^s_t}\lesssim\lambda\ell$.
\end{enumerate}
\end{lemma}

The coefficients of the correction are thus of size at most $\lambda$, and they decay like $n^{-1}$ up to the frequency $\lambda^{-1}$.

\begin{proof}
The formula for $\hat h_1(n)$ follows from $\widehat{\bar zF}(n)=\hat F(n+1)$ and \eqref{defDcal}. Since $\hat F(n+1)=\sum_j\widehat{G_j}(n+1)$ by \eqref{partialfractions}, the bound \eqref{envelopeh1} follows from \eqref{envelopeGj} with $m=1$. Since $\frac{\lambda}{n+1}e^{-\kappa_*\lambda n}\leq\lambda$ and $\frac{\lambda}{n+1}e^{-\kappa_*\lambda n}=\frac{1}{(n+1)^2}\,\lambda(n+1)e^{-\kappa_*\lambda n}\lesssim\la n\ra^{-2}$, we get $|\hat h_1(n)|\lesssim\min(\lambda\ell,\la n\ra^{-2})$. On a bounded time interval, $\lambda$ is bounded from below by a positive constant, so that \eqref{envelopeh1} and the dominated convergence theorem give $h_1\in\mathcal C^1([T_0,+\infty),H^r_+)$ for every $r\geq0$. By Lemma \ref{lemmasmoothing}, item (2), applied with $\varpi=\bar zF$ and with $\varpi=\bar z\pa_tF$, we get $\|h_1\|_{X^s_t}+\|\pa_th_1\|_{X^s_t}\lesssim\lambda\ell$. Finally, $\Re F$ is bounded in $X^s_t$ by Lemma \ref{lemmasmoothing}, items (1) and (2), so that the product estimate of item (1) gives $\|\mathbf T_{\Re F}h_1\|_{X^s_t}\lesssim\|h_1\|_{X^s_t}$, and the bound on $r_1$ follows.
\end{proof}


\subsection{Coordinates adapted to the concentration}
\label{subsectioncoordinates}


We now decompose a function $\eps$ into three parts: a finite-dimensional part, tangent to the set of rational functions which contains the profile; a positive transverse part; and the negative part. For $t\geq T_0$, the function $z\Theta(t)^2$ is a Blaschke product of degree $2J+1$, with a simple zero at the origin and double zeros at the points $\overline{p_j(t)}$, and we consider its model space
\begin{equation}
\label{defmodelspace}
\Kt:=L^2_+\ominus z\Theta(t)^2L^2_+,
\end{equation}
the orthogonal complement of $z\Theta(t)^2L^2_+$ in $L^2_+$; we refer to \cite[Chapter 5]{GMR2016} for the general theory of model spaces. The space $\Kt$ is the tangent space, at $A(t)$, of the set of rational functions with $J$ poles which are bounded at infinity: differentiating such a function with respect to the positions of its poles produces double poles. We parametrize $\Kt$ by $\C^{2J+1}$ through the map
\begin{equation}
\label{defiota}
\iota_t(a):=a_0+\sum_{j=1}^J\big(a_{j,1}G_j(t)+a_{j,2}G_j(t)^2\big),
\end{equation}
where $a:=\big(a_0,(a_{j,1},a_{j,2})_{1\leq j\leq J}\big)\in\C^{2J+1}$, and we endow $\C^{2J+1}$ with the Euclidean norm $|a|$. The coordinate $a_0$ describes the variations of the carrier amplitude $c$, and the coordinates $a_{j,1},a_{j,2}$ describe the variations of the amplitude, position and width of the $j$-th peak. We denote by $P_t$ the orthogonal projection of $L^2_+$ onto $\Kt$.

\begin{lemma}[Adapted coordinates]
\label{lemmacoordinates}
After increasing $T_0$, the following holds for $t\geq T_0$.
\begin{enumerate}
\item {\rm (}Model space{\rm ).} A function $g\in L^2_+$ belongs to $z\Theta^2L^2_+$ if and only if $g(0)=0$ and $g(\overline{p_j})=g'(\overline{p_j})=0$ for $1\leq j\leq J$. The space $\Kt$ consists of the rational functions whose numerator is a polynomial of degree at most $2J$ and whose denominator is $\prod_j(1-p_jz)^2$, and $\iota_t$ is an isomorphism of $\C^{2J+1}$ onto $\Kt$. The orthogonal projection of $L^2_+$ onto $z\Theta^2L^2_+$ is $g\mapsto z\Theta^2\,\mathbf T_{\overline{z\Theta^2}}\,g$.
\item {\rm (}Decomposition{\rm ).} For $r\in\R$, every $\eps\in H^r$ can be written uniquely as
\begin{equation}
\label{decomposition}
\eps=\iota_t(a)+z\Theta^2f+\bar z\bar k,\qquad a\in\C^{2J+1},\ \ f,k\in H^r_+,
\end{equation}
namely $k=\eps^\flat$, $f=\mathbf T_{\overline{z\Theta^2}}\Pi_+\eps$ and $a=\iota_t^{-1}P_t\Pi_+\eps$.
\item {\rm (}Coordinates{\rm ).} For $g\in L^2_+\cap\W$, $|\iota_t^{-1}P_tg|\lesssim\|g\|_\W$.
\item {\rm (}Norms{\rm ).} For $\frac12<s<\frac32$, $\|\eps\|_{X^s_t}\sim|a|+\|f\|_{X^s_t}+\|k\|_{X^s_t}$.
\end{enumerate}
\end{lemma}

\begin{proof}
\noindent{\bf Step 1:} {\it Model space.} Since $|z\Theta^2|=1$ on $\T$, a function $g\in L^2_+$ belongs to $z\Theta^2L^2_+$ if and only if $g/(z\Theta^2)$ is holomorphic in $\Disc$, that is, if and only if $g$ vanishes at $0$ and vanishes to second order at the points $\overline{p_j}$; recall that $|p_j|\geq\frac12$ by Lemma \ref{lemmapeaks}, item (1). For $w\in\Disc$, the Cauchy formula, written in Fourier variables, gives
$$\Big(g\,\Big|\,\frac{1}{1-\bar wz}\Big)=\sum_{n\geq0}\hat g(n)w^n=g(w),\qquad \Big(g\,\Big|\,\frac{z}{(1-\bar wz)^2}\Big)=g'(w).$$
Hence $z\Theta^2L^2_+$ is the orthogonal complement in $L^2_+$ of the span of the $2J+1$ functions $1$, $\frac{1}{1-p_jz}$, $\frac{z}{(1-p_jz)^2}$, and $\Kt$ is this span, which is the space of rational functions of the lemma. The functions $1,G_j,G_j^2$ belong to this space, and they are linearly independent by Step 3 below, so $\iota_t$ is an isomorphism. Finally, multiplication by $z\Theta^2$ is an isometry of $L^2_+$ onto $z\Theta^2L^2_+$, whose adjoint is $\mathbf T_{\overline{z\Theta^2}}$, which gives the formula for the projection.

\noindent{\bf Step 2:} {\it Decomposition.} The negative part of \eqref{decomposition} is $\Pi_-\eps=\bar z\bar k$, that is $k=\eps^\flat$. On $L^2_+$, the decomposition $\Pi_+\eps=P_t\Pi_+\eps+z\Theta^2\mathbf T_{\overline{z\Theta^2}}\Pi_+\eps$ is orthogonal. All these operators are bounded on $H^r$, since $\Theta(t)$ is smooth on $\T$ and $\Kt$ is a finite-dimensional space of smooth functions.

\noindent{\bf Step 3:} {\it Coordinates.} Let $g\in L^2_+\cap\W$ and $a:=\iota_t^{-1}P_tg$. Since $g-\iota_t(a)\in z\Theta^2L^2_+$, Step 1 shows that $\iota_t(a)$ has the same value as $g$ at $0$, and the same values and derivatives as $g$ at the points $\overline{p_j}$. We compute these data for the basis functions. All the functions $G_l^m$ vanish at $0$, hence $a_0=g(0)$. At $z=\overline{p_j}$, since $1-|p_j|^2=\lambda_j$ and $\pa_zG_j=b_j\lambda_j(1-p_jz)^{-2}$,
\begin{gather*}
G_j(\overline{p_j})=b_j\overline{p_j},\qquad G_j^2(\overline{p_j})=b_j^2\overline{p_j}^{\,2},\\
\lambda_j\pa_zG_j(\overline{p_j})=b_j,\qquad \lambda_j\pa_z(G_j^2)(\overline{p_j})=2b_j^2\overline{p_j},
\end{gather*}
while for $l\neq j$ the corresponding quantities for $G_l$ and $G_l^2$ are $O(\lambda)$, because $|1-p_l\overline{p_j}|\gtrsim1$. Therefore the vector $(a_{j,1},a_{j,2})_j$ solves a linear system whose right hand side is $\big(g(\overline{p_j})-g(0),\lambda_jg'(\overline{p_j})\big)_j$ and whose matrix is block diagonal up to $O(\lambda)$, with the diagonal blocks
$$\begin{pmatrix}b_j\overline{p_j}&b_j^2\overline{p_j}^{\,2}\\ b_j&2b_j^2\overline{p_j}\end{pmatrix},\qquad \det=b_j^3\overline{p_j}^{\,2}\longrightarrow e^{3i\vartheta}\,\overline{\zeta_j}\neq0.$$
After increasing $T_0$, this matrix is invertible with a uniformly bounded inverse. Finally, $|g(\overline{p_j})|\leq\|g\|_\W$ and
$$\lambda_j|g'(\overline{p_j})|\leq\|g\|_\W\sup_{n\geq1}\lambda_jn|p_j|^{n-1}\leq\frac{2}{e|p_j|}\|g\|_\W\leq2\|g\|_\W,$$
since $nx^{n-1}\leq(ex\log\frac1x)^{-1}$ for $0<x<1$, $\log\frac1{|p_j|}\geq\frac{\lambda_j}{2}$, and $|p_j|\geq\frac12\geq\frac1e$.

\noindent{\bf Step 4:} {\it Norms.} By \eqref{envelopeGj} and Lemma \ref{lemmasmoothing}, item (2), $\|G_j^m\|_{X^s_t}\lesssim1$ for $m=1,2$, hence $\|\iota_t(a)\|_{X^s_t}\lesssim|a|$. Since $\Theta=F-\bar p$, we have $\|z\Theta^2\|_{X^s_t}\lesssim1$, and the product estimate gives $\|z\Theta^2f\|_{X^s_t}\lesssim\|f\|_{X^s_t}$; also $\|\bar z\bar k\|_{X^s_t}\leq2^s\|k\|_{X^s_t}$. Conversely, $\|k\|_{X^s_t}=\|\eps^\flat\|_{X^s_t}\leq2\|\eps\|_{X^s_t}$, $\|f\|_{X^s_t}\leq\|\overline{z\Theta^2}\,\Pi_+\eps\|_{X^s_t}\lesssim\|\eps\|_{X^s_t}$, and $|a|\lesssim\|\Pi_+\eps\|_\W\leq\|\eps\|_\W$ by Step 3.
\end{proof}


\subsection{The linearized equation in adapted coordinates}


We now compute the equation $i\pa_t\eps=\mathcal L_A\eps+g$ in the coordinates \eqref{decomposition}, for a given forcing term $g$. The computation relies on three exact facts, which we explain before stating the result.

First, the positive part of $\mathcal L_A$ maps the tangent space $\Kt$ into itself, up to the time derivative of the basis functions. This is expected, since $\Kt$ is tangent to a set of rational functions on which the projected equation is defined, and it is checked below by counting the orders of the poles.

Second, the transverse coordinate $f$ is not forced by the finite-dimensional part, and its own evolution is the rotation $e^{-i(M-2K)t}$. Indeed, by \eqref{derivF},
$$i\pa_t(z\Theta^2)=4\sqrt{2KQ}\,(\Re F)\,z\Theta^2,$$
so that the concentrated term $4\sqrt{2KQ}(\Re F)z\Theta^2f$ of $\mathcal L_A(z\Theta^2f)$ is exactly compensated by the time derivative of the factor $z\Theta^2$. The transverse coordinate interacts with the negative part only, through the function $\tilde A$ of \eqref{defAtilde}.

Third, the reflected negative coordinate $k$ is driven by the large operator $-2\D$, which rotates the negative modes at the frequencies $2\D_n\geq2(N+K)$.

We define the three real-linear maps
\begin{align}
\label{defOmega}
\Omega(t)a&:=\iota_t^{-1}P_t\Big[\Pi_+\mathcal L_A\iota_t(a)-\sum_{j=1}^J\big(a_{j,1}\,i\pa_tG_j+a_{j,2}\,i\pa_t(G_j^2)\big)\Big],\qquad a\in\C^{2J+1},\\
\label{defUpsilon}
\Upsilon(t)k&:=\iota_t^{-1}P_t\,\Pi_+\mathcal L_A(\bar z\bar k),\qquad k\in L^2_+,\\
\label{defSigma}
\Sigma(t)a&:=-\big(\mathcal L_A\iota_t(a)\big)^\flat,\qquad a\in\C^{2J+1}.
\end{align}
The map $\Omega(t)$ is the generator of the linearized flow on the tangent space $\Kt$. By the first fact above, the bracket in \eqref{defOmega} belongs to $\Kt$, so that the projection $P_t$ does not modify it; it only makes the definition of $\Omega(t)$ independent of this fact. The map $\Upsilon(t)$ describes the action of the negative part on the finite-dimensional part, and $\Sigma(t)$ the action of the finite-dimensional part on the negative part. In the limit $t\to+\infty$, the finite-dimensional block becomes the real-linear map $\Omega_\infty$ of $\C^{2J+1}$ defined by
\begin{equation}
\label{Omegainfty}
\begin{split}
(\Omega_\infty a)_0&:=Ma_0+c_\infty^2\,\overline{a_0},\\
(\Omega_\infty a)_{j,1}&:=2\sqrt{2KQ}\,a_0+2\sqrt Q\,c_\infty\,\overline{a_0}+Qa_{j,1}+\big(2\sqrt{2KQ}\,e^{-i\vartheta}-Q\big)\overline{a_{j,1}}\\
&\quad+2\sqrt{2KQ}\,a_{j,2}+2\sqrt{2KQ}\,e^{-2i\vartheta}\,\overline{a_{j,2}},\\
(\Omega_\infty a)_{j,2}&:=Q\,\overline{a_0}+\sqrt{2KQ}\,a_{j,1}+Qe^{-i\vartheta}\,\overline{a_{j,1}}+(Q+i\Lambda)a_{j,2}+Qe^{-2i\vartheta}\,\overline{a_{j,2}}.
\end{split}
\end{equation}
In the limit, the blocks of different peaks do not interact with each other, and they only interact with the constant coordinate $a_0$, which evolves autonomously.

\begin{proposition}[Linearized equation in adapted coordinates]
\label{proplinearized}
Let $\mathcal I\subset[T_0,+\infty)$ be an interval, $r\geq0$, $g\in\mathcal C(\mathcal I,H^r)$, and $\eps\in\mathcal C^1(\mathcal I,H^r)\cap\mathcal C(\mathcal I,H^{r+1})$ with coordinates $(a,f,k)$ given by \eqref{decomposition}. Then $\eps$ solves $i\pa_t\eps=\mathcal L_A\eps+g$ on $\mathcal I$ if and only if
\begin{equation}
\label{system}
\left\{\begin{array}{l}
i\pa_ta=\Omega(t)a+\Upsilon(t)k+\iota_t^{-1}P_t\Pi_+g,\\
i\pa_tf=(M-2K)f+\mathbf T_{\overline{\tilde A}^2}k+\mathbf T_{\overline{z\Theta^2}}\Pi_+g,\\
i\pa_tk=-(2\D+M-2K)k-4\sqrt{2KQ}\,\mathbf T_{\Re F}k-\tilde A^2f+\Sigma(t)a-g^\flat.
\end{array}\right.
\end{equation}
Moreover, for $t\geq T_0$, the following holds.
\begin{enumerate}
\item {\rm (}Action of the negative part{\rm ).} $|\Upsilon(t)k|\lesssim\|k\|_\W$ for $k\in L^2_+\cap\W$.
\item {\rm (}Action on the negative part{\rm ).} There are real-linear forms $\Sigma_{j,1}(t)$ and $\Sigma_{j,2}(t)$ on $\C^{2J+1}$, $1\leq j\leq J$, of class $\mathcal C^1$ in $t$, with
$$|\Sigma_{j,1}(t)a|+|\Sigma_{j,2}(t)a|+|\pa_t[\Sigma_{j,1}(t)]a|+|\pa_t[\Sigma_{j,2}(t)]a|\lesssim|a|,$$
such that
\begin{equation}
\label{structureSigma}
\Sigma(t)a=\sum_{j=1}^J\Big(\Sigma_{j,1}(t)a\,\frac{\lambda_j}{1-p_jz}+\Sigma_{j,2}(t)a\,\frac{\lambda_j^2z}{(1-p_jz)^2}\Big).
\end{equation}
In particular, $\Sigma(t)a$ and $\pa_t[\Sigma(t)]a$ satisfy \eqref{envelope} with $\kappa_0=\kappa_*$ and $C_0\lesssim|a|$.
\item {\rm (}Finite-dimensional block{\rm ).} $|\Omega(t)a-\Omega_\infty a|\lesssim\lambda(t)|a|$.
\end{enumerate}
\end{proposition}

\begin{proof}
The map $\eps\mapsto(a,f,k)$ of Lemma \ref{lemmacoordinates}, item (2), and its inverse are bounded on $H^r$ and $H^{r+1}$ and of class $\mathcal C^1$ in $t$, since $\Theta(t)$ and $F(t)$ are smooth on $\T$; hence $\eps$ and $(a,f,k)$ have the same regularity in time, and, by linearity, it is enough to compute $i\pa_t\eps-\mathcal L_A\eps$ for each of the three parts of \eqref{decomposition}. We use the following elementary fact. For a rational function without poles on $\T$ and bounded at infinity, the negative part on $\T$ is the sum of its principal parts at its poles in $\Disc$, and the positive part is a rational function bounded at infinity, whose poles are its poles outside $\overline\Disc$, with the same principal parts. Indeed, for $|w|<1$ the function $(z-w)^{-m}=z^{-m}(1-w\bar z)^{-m}$ has only negative modes on $\T$, while the remainder is holomorphic in a neighborhood of $\overline\Disc$. All the rational functions below have their poles in the set $\{\overline{p_j}\}\cup\{1/p_j\}$; for instance $\overline{G_j}=\overline{b_j}\lambda_j\,(z-\overline{p_j})^{-1}$ on $\T$.

\noindent{\bf Step 1:} {\it The transverse part.} By \eqref{derivF}, $i\pa_t(z\Theta^2)=4\sqrt{2KQ}(\Re F)\,z\Theta^2$. Since $(z\Theta^2f)(0)=0$, \eqref{actionL} gives $L(z\Theta^2f)=0$, and since $\overline\Theta A=\overline{\tilde A}$ on $\T$, we have $A^2\,\overline{z\Theta^2f}=\bar z\,\overline{\tilde A^2f}$. Therefore
$$i\pa_t(z\Theta^2f)-\mathcal L_A(z\Theta^2f)=z\Theta^2\big(i\pa_tf-(M-2K)f\big)-\bar z\,\overline{\tilde A^2f}.$$
The first term belongs to $z\Theta^2L^2_+$, and the second one has only negative modes, with reflection $-\tilde A^2f$.

\noindent{\bf Step 2:} {\it The negative part.} By the computation leading to \eqref{defrh}, with $h$ replaced by $k$ and without the forcing term,
$$\big(i\pa_t(\bar z\bar k)-\mathcal L_A(\bar z\bar k)\big)^\flat=-i\pa_tk-2\D k-(M-2K)k-4\sqrt{2KQ}\,\mathbf T_{\Re F}k.$$
The positive part of $\mathcal L_A(\bar z\bar k)$ is $2\sqrt{2KQ}\,\Pi_+(F\bar z\bar k)+A^2zk$, since $\overline F\bar z\bar k$ has only negative modes. Expanding $(1-p_jz)^{-1}$ and $\bar k$ in Fourier series, we get
\begin{equation}
\label{projectionGjk}
\Pi_+(G_j\bar z\bar k)=b_j\lambda_j\,\Pi_+\Big(\frac{\bar k}{1-p_jz}\Big)=\overline{k(\overline{p_j})}\,\frac{b_j\lambda_j}{1-p_jz}\in\Kt.
\end{equation}
Hence the projection of $\Pi_+\mathcal L_A(\bar z\bar k)$ onto $z\Theta^2L^2_+$ is $z\Theta^2\,\mathbf T_{\overline{z\Theta^2}}(A^2zk)=z\Theta^2\,\mathbf T_{\overline{\tilde A}^2}k$, again by $\overline\Theta A=\overline{\tilde A}$, and its projection onto $\Kt$ is $\iota_t(\Upsilon(t)k)$. Item (1) follows from Lemma \ref{lemmacoordinates}, item (3), and $\|\Pi_+\mathcal L_A(\bar z\bar k)\|_\W\lesssim\|k\|_\W$.

\noindent{\bf Step 3:} {\it The finite-dimensional part.} Let $a\in\C^{2J+1}$. Then $i\pa_t\iota_t(a)=\iota_t(i\pa_ta)+\sum_j\big(a_{j,1}i\pa_tG_j+a_{j,2}i\pa_t(G_j^2)\big)$, so that
\begin{multline*}
i\pa_t\iota_t(a)-\mathcal L_A\iota_t(a)=\iota_t(i\pa_ta)-\Big[\Pi_+\mathcal L_A\iota_t(a)-\sum_{j=1}^J\big(a_{j,1}\,i\pa_tG_j+a_{j,2}\,i\pa_t(G_j^2)\big)\Big]\\
-\Pi_-\mathcal L_A\iota_t(a).
\end{multline*}
We claim that the bracket, which is also the bracket of \eqref{defOmega}, belongs to $\Kt$, so that it equals $\iota_t(\Omega(t)a)$. By the elementary fact above, the bracket is a rational function bounded at infinity, whose poles are among the points $1/p_j$, and by Lemma \ref{lemmacoordinates}, item (1), it is enough to show that these poles have order at most two. The functions $\iota_t(a)$, $A^2$ and $i\pa_tG_j$ have poles of order at most two at $1/p_j$, while $\overline F$ and $\overline{\iota_t(a)}$ have no pole outside $\overline\Disc$. The only terms of the bracket with a pole of order three at $1/p_j$ are therefore $2\sqrt{2KQ}\,a_{j,2}FG_j^2$, which comes from $4\sqrt{2KQ}(\Re F)\iota_t(a)$, and $-a_{j,2}i\pa_t(G_j^2)$. By \eqref{derivGj}, $i\pa_t(G_j^2)=-2i\mu_jG_j^2+2\sqrt{2KQ}\,G_j^3$, and $FG_j^2-G_j^3=(F-G_j)G_j^2$ has a pole of order at most two at $1/p_j$. The poles of order three cancel, which proves the claim. Consequently, the $\Kt$-component of $i\pa_t\iota_t(a)-\mathcal L_A\iota_t(a)$ is $\iota_t\big(i\pa_ta-\Omega(t)a\big)$, its component in $z\Theta^2L^2_+$ vanishes, and the reflection of its negative part is $\Sigma(t)a$.

We now project the equation $i\pa_t\eps-\mathcal L_A\eps=g$ onto $\Kt$, onto $z\Theta^2L^2_+$, and onto the negative modes. Collecting the contributions of Steps 1, 2 and 3, and using that $\iota_t$ and $f\mapsto z\Theta^2f$ are injective, we obtain the three equations of \eqref{system}.

\noindent{\bf Step 4:} {\it Structure of $\Sigma$.} The negative part of $\mathcal L_A\iota_t(a)$ is $\Pi_-\big(2\sqrt{2KQ}\,\overline F\iota_t(a)+A^2\overline{\iota_t(a)}\big)$, which is the sum of the principal parts of $2\sqrt{2KQ}\,\overline F\iota_t(a)+A^2\overline{\iota_t(a)}$ at the points $\overline{p_j}$. On $\T$, $\overline F=\sum_l\overline{b_l}\lambda_l(z-\overline{p_l})^{-1}$ and $\overline{\iota_t(a)}=\overline{a_0}+\sum_l\big(\overline{a_{l,1}}\,\overline{G_l}+\overline{a_{l,2}}\,\overline{G_l}^{\,2}\big)$, while $\iota_t(a)$ and $A$ are holomorphic near $\overline{p_j}$. We compute the values of $A$ and of its derivative at $\overline{p_j}$. Since $\Theta(\overline{p_j})=0$ and $\pa_z\Theta=\frac{\lambda B'}{(1-pB)^2}$, with $B(\overline{p_j})=\bar p$, \eqref{defTheta} gives $A(\overline{p_j})=\sqrt{2K}$ and $\pa_z(A^2)(\overline{p_j})=2\sqrt{2KQ}\,\frac{B'(\overline{p_j})}{\lambda}$, that is, by the definition of $b_j$ in Lemma \ref{lemmapeaks}, item (2), $\lambda_j\pa_z(A^2)(\overline{p_j})=\frac{2\sqrt{2KQ}}{\overline{b_j}}$. Expanding $A^2$ at $\overline{p_j}$ to first order, we find that the principal part at $\overline{p_j}$ is
$$\Big(2\sqrt{2KQ}\,\overline{b_j}\big(\iota_t(a)(\overline{p_j})+\overline{a_{j,2}}\big)+2K\,\overline{b_j}\,\overline{a_{j,1}}\Big)\frac{\lambda_j}{z-\overline{p_j}}+2K\,\overline{b_j}^{\,2}\,\overline{a_{j,2}}\,\frac{\lambda_j^2}{(z-\overline{p_j})^2}.$$
Since $\Pi_-g=\bar z\,\overline{g^\flat}$, and since on $\T$ the functions $\bar z\,\overline{\lambda_j(z-\overline{p_j})^{-1}}$ and $\bar z\,\overline{\lambda_j^2(z-\overline{p_j})^{-2}}$ are $\frac{\lambda_j}{1-p_jz}$ and $\frac{\lambda_j^2z}{(1-p_jz)^2}$, the function $\Sigma(t)a=-(\mathcal L_A\iota_t(a))^\flat$ is given by \eqref{structureSigma}, with the real-linear forms
\begin{gather*}
\Sigma_{j,1}(t)a:=-2\sqrt{2KQ}\,b_j\Big(\overline{\iota_t(a)(\overline{p_j})}+a_{j,2}\Big)-2K\,b_j\,a_{j,1},\\
\Sigma_{j,2}(t)a:=-2K\,b_j^2\,a_{j,2}.
\end{gather*} By Lemma \ref{lemmapeaks}, item (2), we have $\frac12\leq|b_j|\leq2$ and $|b_j'|\lesssim\lambda$. The value $\iota_t(a)(\overline{p_j})$ is bounded by $\|\iota_t(a)\|_\W\lesssim|a|$, and its time derivative is bounded by $|a|$, since $G_l(\overline{p_j})$ equals $b_j\overline{p_j}$ for $l=j$ and $\frac{b_l\lambda_l\overline{p_j}}{1-p_l\overline{p_j}}$ for $l\neq j$. This gives the bounds of item (2) on $\Sigma_{j,1}(t)a$, $\Sigma_{j,2}(t)a$ and their time derivatives. Finally, $\frac{\lambda_j}{1-p_jz}=\frac{\bar zG_j}{b_j}$ and $\frac{\lambda_j^2z}{(1-p_jz)^2}=\frac{\bar zG_j^2}{b_j^2}$, so that the bound \eqref{envelope} for $\Sigma(t)a$ and $\pa_t[\Sigma(t)]a$ follows from \eqref{envelopeGj} with $m=1,2$.

\noindent{\bf Step 5:} {\it The limit of $\Omega(t)$.} We write $O_\W(\lambda)$ for any function whose Wiener norm is $\lesssim\lambda$. By Lemma \ref{lemmapeaks}, $\lambda_j=O(\lambda)$, $\overline{b_j}p_j=e^{-i\vartheta}+O(\lambda)$, $|b_j|^2=1+O(\lambda)$, $\mu_j=\Lambda+O(\lambda)$, and $G_jG_l$, $G_j\overline{G_l}$ are $O_\W(\lambda)$ for $j\neq l$; moreover $c=c_\infty+O(\lambda)$ by \eqref{defc} and Lemma \ref{lemmaplambda}, item (3). Since $\W$ is an algebra, multiplication by a function bounded in $\W$ and the projection $\Pi_+$ map $O_\W(\lambda)$ into $O_\W(\lambda)$. We first derive four projection formulas. By \eqref{quadraticGj}, $|G_j|^2=|b_j|^2\lambda_j+\overline{b_j}p_jG_j+b_j\overline{p_j}\,\overline{G_j}$, hence, using this identity once for the first formula and twice for the second one,
\begin{gather*}
\Pi_+(\overline FG_j)=e^{-i\vartheta}G_j+O_\W(\lambda),\\
\Pi_+(\overline FG_j^2)=\Pi_+(|G_j|^2G_j)+O_\W(\lambda)=G_j+e^{-i\vartheta}G_j^2+O_\W(\lambda).
\end{gather*}
Next, $A^2=c^2+2\sqrt Q\,cF+QF^2$ and $\Pi_+(c^2\overline{G_j})=0$ give
\begin{align*}
\Pi_+(A^2\overline{G_j})&=2\sqrt Q\,c_\infty e^{-i\vartheta}G_j+Q\big(G_j+e^{-i\vartheta}G_j^2\big)+O_\W(\lambda)\\
&=\big(2\sqrt{2KQ}\,e^{-i\vartheta}-Q\big)G_j+Qe^{-i\vartheta}G_j^2+O_\W(\lambda),
\end{align*}
since $\sqrt Q\,c_\infty e^{-i\vartheta}=\sqrt{2KQ}\,e^{-i\vartheta}-Q$. Finally, $\Pi_-(A^2\overline{G_j})$ is the principal part $A(\overline{p_j})^2\,\overline{G_j}$ of $A^2\overline{G_j}$ at $\overline{p_j}$, and $\Pi_+(\overline{G_j}^{\,2})=0$, hence
$$\Pi_+(A^2\overline{G_j}^{\,2})=\Pi_+\big(\overline{G_j}\,\Pi_+(A^2\overline{G_j})\big)=2\sqrt{2KQ}\,e^{-2i\vartheta}G_j+Qe^{-2i\vartheta}G_j^2+O_\W(\lambda).$$
We now compute the bracket $\iota_t(\Omega(t)a)$ of \eqref{defOmega}, using $L\iota_t(a)=2Ka_0$, $\Pi_+\mathcal L_A\iota_t(a)=2Ka_0+(M-2K)\iota_t(a)+2\sqrt{2KQ}\big(F\iota_t(a)+\Pi_+(\overline F\iota_t(a))\big)+\Pi_+(A^2\overline{\iota_t(a)})$, the identity $i\pa_tG_j=-i\mu_jG_j+\sqrt{2KQ}\,G_j^2$ of \eqref{derivGj}, and $FG_j^m=G_j^{m+1}+O_\W(\lambda)$. The terms containing $a_0$ are
$$Ma_0+c^2\overline{a_0}+\sum_j\Big(\big(2\sqrt{2KQ}\,a_0+2\sqrt Q\,c\,\overline{a_0}\big)G_j+Q\,\overline{a_0}\,G_j^2\Big)+O_\W(\lambda)|a_0|.$$
The terms containing $a_{j,1}$ are
\begin{align*}
&a_{j,1}\Big(\big(M-2K+2\sqrt{2KQ}\,e^{-i\vartheta}+i\Lambda\big)G_j+\sqrt{2KQ}\,G_j^2\Big)\\
&\qquad+\overline{a_{j,1}}\Big(\big(2\sqrt{2KQ}\,e^{-i\vartheta}-Q\big)G_j+Qe^{-i\vartheta}G_j^2\Big)+O_\W(\lambda)|a|.
\end{align*}
The terms containing $a_{j,2}$ are
\begin{align*}
&a_{j,2}\Big(2\sqrt{2KQ}\,G_j+\big(M-2K+2\sqrt{2KQ}\,e^{-i\vartheta}+2i\Lambda\big)G_j^2\Big)\\
&\qquad+\overline{a_{j,2}}\Big(2\sqrt{2KQ}\,e^{-2i\vartheta}G_j+Qe^{-2i\vartheta}G_j^2\Big)+O_\W(\lambda)|a|,
\end{align*}
where the terms in $G_j^3$, namely $2\sqrt{2KQ}\,a_{j,2}G_j^3$ and $-2\sqrt{2KQ}\,a_{j,2}G_j^3$, have cancelled. By \eqref{deftheta}, $M-2K+2\sqrt{2KQ}\,e^{-i\vartheta}=M-2K+(2K+Q-M)-i\Lambda=Q-i\Lambda$. Collecting the coefficients of $1$, $G_j$ and $G_j^2$, we have proved that $\iota_t(\Omega(t)a)=\iota_t(\Omega_\infty a)+O_\W(\lambda)|a|$. Since both functions belong to $\Kt$, Lemma \ref{lemmacoordinates}, item (3), gives $|\Omega(t)a-\Omega_\infty a|=|\iota_t^{-1}P_t(\iota_t(\Omega(t)a)-\iota_t(\Omega_\infty a))|\lesssim\lambda|a|$, which is item (3).
\end{proof}


\subsection{The finite-dimensional block}


In the limit $t\to+\infty$, the finite-dimensional part of the linearized flow is the linear system $\pa_ta=-i\Omega_\infty a$. For $t\in\R$, we denote by $e^{-it\Omega_\infty}$ the flow at time $t$ of this real-linear system, so that $e^{it\Omega_\infty}$ is its flow at time $-t$. Since we shall integrate backwards in time, the relevant question is whether $e^{it\Omega_\infty}$, $t\geq0$, grows: an eigenvalue of $-i\Omega_\infty$ with negative real part would produce an exponential growth backwards in time. We show that the only eigenvalues are $0$ and $\Lambda$, by exhibiting real coordinates in which $-i\Omega_\infty$ is triangular.

\begin{lemma}[The finite-dimensional block]
\label{lemmafiniteblock}
The real-linear map $-i\Omega_\infty$ of $\C^{2J+1}\simeq\R^{4J+2}$ has the eigenvalue $0$, with algebraic multiplicity $2J+2$, and the eigenvalue $\Lambda$, with algebraic multiplicity $2J$. Moreover,
\begin{equation}
\label{backwardfinite}
|e^{it\Omega_\infty}a|\lesssim\la t\ra^2|a|,\qquad t\geq0,\ \ a\in\C^{2J+1}.
\end{equation}
\end{lemma}

\begin{proof}
\noindent{\bf Step 1:} {\it Triangular coordinates.} We normalize the constant coordinate and rotate the second coordinate of each peak, and we introduce one combination of the coordinates of each peak:
\begin{equation}
\label{coordinatesfinite}
\tilde a_0:=\frac{a_0}{c_\infty},\qquad \tilde a_{j,2}:=e^{i\vartheta}a_{j,2},\qquad \rho_j:=\Re a_{j,1}+2\Re\tilde a_{j,2},\qquad 1\leq j\leq J.
\end{equation}
The combination $\rho_j$ is chosen so that it is conserved, as we now show. Let $a(t)$ solve $\pa_ta=-i\Omega_\infty a$. For the constant coordinate, \eqref{Omegainfty} and $|c_\infty|^2=M$ give $ic_\infty\pa_t\tilde a_0=Mc_\infty\tilde a_0+c_\infty^2\overline{c_\infty}\,\overline{\tilde a_0}=Mc_\infty(\tilde a_0+\overline{\tilde a_0})$, that is
\begin{equation}
\label{triangulara0}
\pa_t\Re\tilde a_0=0,\qquad \pa_t\Im\tilde a_0=-2M\Re\tilde a_0.
\end{equation}
We fix $j$, and we write $a_1:=a_{j,1}$, $\tilde a_2:=\tilde a_{j,2}$ and $\rho:=\rho_j$. By \eqref{Omegainfty}, the constant coordinate enters the equations of $a_1$ and $\tilde a_2$ only through the additive terms $2\sqrt{2KQ}\,a_0+2\sqrt Q\,c_\infty\,\overline{a_0}$ and $Qe^{i\vartheta}\,\overline{a_0}$, respectively, and we first set $a_0=0$. Multiplying the equation of $a_{j,2}$ by $e^{i\vartheta}$, \eqref{Omegainfty} becomes
\begin{gather*}
i\pa_ta_1=Q(a_1-\overline{a_1})+2\sqrt{2KQ}\,e^{-i\vartheta}\big(\overline{a_1}+\tilde a_2+\overline{\tilde a_2}\big),\\
i\pa_t\tilde a_2=\sqrt{2KQ}\,e^{i\vartheta}a_1+Q\big(\overline{a_1}+\tilde a_2+\overline{\tilde a_2}\big)+i\Lambda\tilde a_2.
\end{gather*}
We have $\overline{a_1}+\tilde a_2+\overline{\tilde a_2}=\rho-i\Im a_1$ and $a_1-\overline{a_1}=2i\Im a_1$, and $\sqrt{2KQ}\,e^{\mp i\vartheta}=\frac{2K+Q-M}{2}\mp\frac{i}{2}\Lambda$ by \eqref{deftheta}. Multiplying the two equations by $-i$ and taking the real and imaginary parts, we obtain
\begin{align*}
\pa_t\Re a_1&=(M+Q-2K)\Im a_1-\Lambda\rho,\\
\pa_t\Im a_1&=\Lambda\Im a_1-(2K+Q-M)\rho,\\
\pa_t\Re\tilde a_2&=\Lambda\Re\tilde a_2+\frac\Lambda2\Re a_1-\frac{M+Q-2K}{2}\Im a_1,\\
\pa_t\Im\tilde a_2&=\Lambda\Im\tilde a_2-\frac{2K+Q-M}{2}\Re a_1+\frac\Lambda2\Im a_1-Q\rho.
\end{align*}
Hence $\pa_t\rho=\pa_t\Re a_1+2\pa_t\Re\tilde a_2=\Lambda(\Re a_1+2\Re\tilde a_2)-\Lambda\rho=0$. For a general $a_0$, the two additive terms add real-linear functions of $a_0$ to the right hand sides of these four equations. Their contribution to $\pa_t\rho$ is the imaginary part of $2\sqrt{2KQ}\,a_0+2\sqrt Q\,c_\infty\,\overline{a_0}+2Qe^{i\vartheta}\,\overline{a_0}$, which vanishes, since this number equals $2\sqrt{2KQ}\,(a_0+\overline{a_0})$ by \eqref{defc}. Therefore
\begin{equation}
\label{triangularpeak}
\pa_t\rho_j=0,\qquad 1\leq j\leq J,
\end{equation}
while the equations of $\Im a_{j,1}$, $\Re a_{j,1}$ and $\Im\tilde a_{j,2}$ are the ones displayed above, plus real-linear functions of $a_0$.

\noindent{\bf Step 2:} {\it Eigenvalues.} We order the real coordinates as
$$\Re\tilde a_0,\ \Im\tilde a_0,\ \big(\rho_j,\ \Im a_{j,1},\ \Re a_{j,1},\ \Im\tilde a_{j,2}\big)_{1\leq j\leq J}.$$
They determine $a$, since $\Re\tilde a_{j,2}=\frac12(\rho_j-\Re a_{j,1})$, so that they depend on $a$ through a real-linear isomorphism of $\C^{2J+1}\simeq\R^{4J+2}$. By \eqref{triangulara0}, \eqref{triangularpeak} and the equations of Step 1, the time derivative of each coordinate is the sum of a multiple of this coordinate and of a linear combination of the previous ones. In other words, the matrix of $-i\Omega_\infty$ in these coordinates is lower triangular, with the diagonal entries $0,0$, and $0,\Lambda,0,\Lambda$ for each peak. This gives the eigenvalues and their multiplicities.

\noindent{\bf Step 3:} {\it Backward growth.} For $t\geq0$, $e^{it\Omega_\infty}a$ is the value at the time $t$ of the solution of the system obtained by changing the sign of the right hand side of the equations of Step 1, with the initial value $a$. The matrix of this system in the coordinates of Step 2 is lower triangular, with the diagonal entries $0$ and $-\Lambda$. We integrate this system coordinate by coordinate, in the order of Step 2. If the off-diagonal part of the equation of a coordinate with diagonal entry $d\in\{0,-\Lambda\}$ is $O(\la t\ra^m)|a|$, then this coordinate is $O(\la t\ra^{m+1})|a|$ if $d=0$, and $O(\la t\ra^m)|a|$ if $d=-\Lambda$, since $\int_0^te^{-\Lambda(t-t')}\la t'\ra^m\,dt'\lesssim\la t\ra^m$. Hence, in units of $|a|$, $\Re\tilde a_0=O(1)$, $\Im\tilde a_0=O(\la t\ra)$, $\rho_j=O(1)$, $\Im a_{j,1}=O(\la t\ra)$, $\Re a_{j,1}=O(\la t\ra^2)$, and $\Im\tilde a_{j,2}$ and $\Re\tilde a_{j,2}$ are $O(\la t\ra^2)$, which proves \eqref{backwardfinite}.
\end{proof}

\begin{remark}
\label{remarkexponents}
The eigenvalues of the finite-dimensional block have a dynamical meaning. The infinitesimal phase rotation $iA=ic+i\sqrt QF$ has the coordinates $a_0=ic$, $a_{j,1}=i\sqrt Q$, $a_{j,2}=0$. The time translation $\pa_tA=-i\sqrt{2K}Q\lambda-\sqrt Q\sum_j\big(\mu_jG_j+i\sqrt{2KQ}\,G_j^2\big)$ has the coordinates $a_0=-i\sqrt{2K}Q\lambda$, $a_{j,1}=-\sqrt Q\mu_j$, $a_{j,2}=-i\sqrt Q\sqrt{2KQ}$. One checks on \eqref{Omegainfty} that the limits of these two vectors as $t\to+\infty$ belong to the kernel of $\Omega_\infty$: they are eigenvectors of $-i\Omega_\infty$ for the eigenvalue $0$. On the other hand, the infinitesimal rotation of the circle, $iz\pa_zA$, has the coordinates $a_0=0$, $a_{j,1}=i\sqrt Q$ and $a_{j,2}=\frac{i\sqrt Q\,p_j}{b_j\lambda_j}$, since $z\pa_zG_j=G_j+\frac{p_j}{b_j\lambda_j}G_j^2$. The coordinate $a_{j,2}$ grows like $\lambda^{-1}\sim e^{\Lambda t}$, and $\lambda_ja_{j,2}\to i\sqrt Q\,e^{-i\vartheta}$; one checks on \eqref{Omegainfty} that the vector whose only nonzero coordinate is $a_{j,2}=i\sqrt Q\,e^{-i\vartheta}$ is an eigenvector of $-i\Omega_\infty$ for the eigenvalue $\Lambda$: moving a concentration point corresponds to the eigenvalue $\Lambda$. No eigenvalue has a negative real part, so that the finite-dimensional block is stable backwards in time, up to polynomial factors.
\end{remark}


\subsection{A normal form for the infinite-dimensional block}
\label{subsectionnormal}


We now apply Proposition \ref{proplinearized} to the solution $\eps$ of \eqref{eqeps}, that is with $g:=2K\sqrt Q\,\overline F+\mathcal N$, where $\mathcal N:=\mathcal N_A(\eps)$; then $\Pi_+g=\Pi_+\mathcal N$ and $g^\flat=2K\sqrt Q\,\bar zF+\mathcal N^\flat$. We remove the forcing at leading order by a correction of the negative part. In Sections \ref{sectionlinear} and \ref{sectionbackward}, the correction is the function $h_1$ of \eqref{defh1}; in Section \ref{sectionendpoint} we shall use a refined correction. We therefore fix $\frac12<s<\frac32$ and a function $h$ such that, for a constant $C_h$,
\begin{equation}
\label{assumptionh}
\begin{gathered}
h\in\mathcal C^1([T_0,+\infty),H^3_+),\\
\|h\|_{X^{s'}_t}+\|r_h\|_{X^{s'}_t}\leq C_h\lambda(t)\ell(t),\qquad t\geq T_0,\ \ s'\in\{1,s\},
\end{gathered}
\end{equation}
where $r_h$ is the residual forcing \eqref{defrh}. By Lemma \ref{lemmah1}, the first correction $h=h_1$ satisfies \eqref{assumptionh} for every $\frac12<s<\frac32$, since $r_{h_1}=r_1$.

In the coordinates $(a,f,k)$ of $\eps$, with $k=\eps^\flat$, three terms of \eqref{system} are not small. The source $\Sigma(t)a$ in the equation of $k$ is a concentrated function of size $|a|$. The coupling between $f$ and $k$ is the multiplication by $\tilde A^2$ in the equation of $k$ and its adjoint $\mathbf T_{\overline{\tilde A}^2}$ in the equation of $f$. And the equation of $k$ contains the Toeplitz operator $4\sqrt{2KQ}\,\mathbf T_{\Re F}$. We remove the source and the concentrated parts of the two other terms by two changes of unknowns, which use the fact that the negative modes rotate at the frequencies $2\D_n\geq2(N+K)$; the constant part of the coupling, given by the value $\tilde c^{\,2}$ of $\tilde A^2$ at the origin, remains. First, we subtract from $k$ the correction $h$ and the response of the negative modes to the source $\Sigma(t)a$:
\begin{equation}
\label{defcheckk}
\check k:=k-h-\frac12\D^{-1}\Sigma(t)a.
\end{equation}
Second, we split $\tilde A^2$ into its value at the origin and a concentrated part: since $\tilde A=\tilde c+\sqrt{2K}F$,
\begin{equation}
\label{concentratedpart}
\tilde A^2-\tilde c^{\,2}=2\sqrt{2K}\,\tilde c\,F+2KF^2,
\end{equation}
which satisfies \eqref{envelope} with $\kappa_0=\kappa_*$, together with its first two time derivatives, by Lemma \ref{lemmapeaks}, item (5), since $\tilde c=\sqrt Q-\sqrt{2K}\,\bar p$ is bounded and $|\tilde c'|+|\tilde c''|\lesssim\lambda$ by Lemma \ref{lemmaplambda}, item (1). Accordingly, we shall see that the pair $(f,\check k)$ solves an equation of the form $i\pa_t(f,\check k)=\mathcal H(t)(f,\check k)+\mathcal Y$, where
\begin{equation}
\label{defHt}
\begin{gathered}
\mathcal H(t):=\mathcal H_0+\mathcal V(t)+\mathcal B(t),\\
\mathcal H_0:=\begin{pmatrix}0&0\\0&-2\D\end{pmatrix},\qquad \mathcal V(t):=\begin{pmatrix}M-2K&\overline{\tilde c}^{\,2}\\-\tilde c^{\,2}&-(M-2K)\end{pmatrix},
\end{gathered}
\end{equation}
\begin{equation}
\label{defBt}
\mathcal B(t):=\begin{pmatrix}0&\mathbf T_{\overline{\tilde A^2-\tilde c^2}}\\-(\tilde A^2-\tilde c^{\,2})&-4\sqrt{2KQ}\,\mathbf T_{\Re F}\end{pmatrix},
\end{equation}
and where the forcing $\mathcal Y:=(\mathcal Y_1,\mathcal Y_2)$ is
\begin{equation}
\label{defY}
\begin{aligned}
\mathcal Y_1&:=\mathbf T_{\overline{\tilde A}^2}\big(h+\tfrac12\D^{-1}\Sigma(t)a\big)+\mathbf T_{\overline{z\Theta^2}}\Pi_+\mathcal N,\\
\mathcal Y_2&:=r_h-\mathcal N^\flat-\tfrac12\big(M-2K+4\sqrt{2KQ}\,\mathbf T_{\Re F}\big)\D^{-1}\Sigma(t)a\\
&\quad-\tfrac i2\D^{-1}\big(\pa_t[\Sigma(t)]a+\Sigma(t)\pa_ta\big),
\end{aligned}
\end{equation}
with $\pa_ta$ given by the first equation of \eqref{system}. Recall that $\Sigma(t)$ is only real-linear: in $\mathcal Y_2$, the factor $-\frac i2$ cannot be moved inside $\Sigma(t)\pa_ta$.

The operator $\mathcal H_0$ is the large diagonal part. The matrix $\mathcal V(t)$ is the constant-coefficient coupling, which converges to the matrix $\mathcal V_\infty$ obtained by replacing $\tilde c$ by $\tilde c_\infty$. The operator $\mathcal B(t)$ is the concentrated coupling. Since the second component rotates at the frequencies $2\D_n$, the operator $\mathcal B(t)$ can be removed by a change of variables $(f,\check k)=(I+\mathcal T(t))\Phi$, where $\mathcal T(t)$ solves the commutator equation $[\mathcal H_0,\mathcal T(t)]=-\mathcal B(t)$. This equation can be solved exactly, because $\mathcal H_0$ is diagonal: we let
\begin{equation}
\label{defTt}
\mathcal T(t):=\begin{pmatrix}0&-\frac12\mathbf T_{\overline{\tilde A^2-\tilde c^2}}\,\D^{-1}\\-\frac12\D^{-1}\,(\tilde A^2-\tilde c^{\,2})&-\frac{2\sqrt{2KQ}}{N}\,\mathbf T_{D^{-1}\Re F}\end{pmatrix}.
\end{equation}
Each nonzero entry contains $\D^{-1}$ or the primitive $D^{-1}$, so that $\mathcal T(t)$ is small by Lemma \ref{lemmasmoothing}. The new unknown is $Z:=(a,\Phi)$, $\Phi:=(\Phi_1,\Phi_2)$. For $r\geq0$, we denote by $\mathcal Z^r:=\C^{2J+1}\times(H^r_+)^2$ the space of these unknowns. A pair is measured by the sum of the norms of its components, as in $\|\Phi\|_{X^s_t}:=\|\Phi_1\|_{X^s_t}+\|\Phi_2\|_{X^s_t}$, and similarly in the other spaces of functions, and we let
\begin{equation}
\label{defZ}
\|Z\|_{\mathcal Z^s_t}:=|a|+\|\Phi\|_{X^s_t},\qquad \mathcal H_\infty:=\mathcal H_0+\mathcal V_\infty.
\end{equation}

\begin{lemma}[The normal form transformation]
\label{lemmatransformation}
After increasing $T_0$, depending on $s$, the following holds for $t\geq T_0$ and $s'\in\{1,s\}$.
\begin{enumerate}
\item {\rm (}The transformation{\rm ).} $[\mathcal H_0,\mathcal T(t)]=-\mathcal B(t)$ on $H^1_+\times H^1_+$, and, as operators on $X^{s'}_t\times X^{s'}_t$, on $\W\times\W$ and on $L^2_+\times L^2_+$,
$$\|\mathcal T(t)\|+\|\pa_t\mathcal T(t)\|\lesssim\lambda\ell,\qquad \|(I+\mathcal T(t))^{-1}\|\leq2.$$
\item {\rm (}Coordinates{\rm ).} The map $Z\mapsto\eps-\bar z\bar h$ defined by \eqref{decomposition}, $k=h+\check k+\frac12\D^{-1}\Sigma(t)a$ and $(f,\check k)=(I+\mathcal T(t))\Phi$ is a real-linear isomorphism of $\mathcal Z^0$ onto $L^2$ and of $\mathcal Z^{s'}$ onto $H^{s'}$, and $\|\eps-\bar z\bar h\|_{X^{s'}_t}\sim\|Z\|_{\mathcal Z^{s'}_t}$.
\end{enumerate}
\end{lemma}

\begin{proof}
\noindent{\bf Step 1:} {\it Commutator equation.} For an operator matrix $\mathcal T=\begin{pmatrix}\mathcal T_{11}&\mathcal T_{12}\\\mathcal T_{21}&\mathcal T_{22}\end{pmatrix}$,
$$[\mathcal H_0,\mathcal T]=\begin{pmatrix}0&2\mathcal T_{12}\D\\-2\D\mathcal T_{21}&-2[\D,\mathcal T_{22}]\end{pmatrix}.$$
With \eqref{defTt}, the off-diagonal entries are $-\mathbf T_{\overline{\tilde A^2-\tilde c^2}}=-\mathcal B_{12}$ and $\tilde A^2-\tilde c^{\,2}=-\mathcal B_{21}$. For the diagonal entry, we use $\D=N(D+1)+K$ and $[D,\mathbf T_b]=\mathbf T_{Db}$, which follows from $D(bg)=(Db)g+bDg$ and $D\Pi_+=\Pi_+D$. Since $\widehat{\Re F}(0)=0$, we have $DD^{-1}\Re F=\Re F$, and $-2[\D,\mathcal T_{22}]=4\sqrt{2KQ}\,\mathbf T_{\Re F}=-\mathcal B_{22}$.

\noindent{\bf Step 2:} {\it Bounds on the transformation.} By \eqref{concentratedpart} and Lemma \ref{lemmapeaks}, item (5), Lemma \ref{lemmasmoothing}, items (2) and (3), applies to the entries of $\mathcal T(t)$ and of $\pa_t\mathcal T(t)$, with $\varpi$ equal to $\tilde A^2-\tilde c^{\,2}$, $\frac12F$ and their time derivatives, in $X^1_t$, in $X^s_t$, in $\W$ and in $L^2$. This gives $\|\mathcal T\|+\|\pa_t\mathcal T\|\lesssim\lambda\ell$ on these four spaces. After increasing $T_0$, depending on $s$, $\|\mathcal T(t)\|\leq\frac12$ on these spaces, and the Neumann series gives $\|(I+\mathcal T)^{-1}\|\leq2$.

\noindent{\bf Step 3:} {\it Coordinates.} The function $\eps-\bar z\bar h$ has the coordinates $(a,f,k-h)$, and $k-h=\check k+\frac12\D^{-1}\Sigma(t)a$. By Lemma \ref{lemmacoordinates}, item (4), $\|\eps-\bar z\bar h\|_{X^{s'}_t}\sim|a|+\|f\|_{X^{s'}_t}+\|k-h\|_{X^{s'}_t}$. By Proposition \ref{proplinearized}, item (2), and Lemma \ref{lemmasmoothing}, item (2), $\|\frac12\D^{-1}\Sigma(t)a\|_{X^{s'}_t}\lesssim\lambda\ell|a|$, and by Step 2, $\|(f,\check k)\|_{X^{s'}_t}\sim\|\Phi\|_{X^{s'}_t}$. This proves the equivalence of norms of item (2). The isomorphism property follows from Lemma \ref{lemmacoordinates}, item (2), and from the invertibility of $I+\mathcal T(t)$ on $L^2_+\times L^2_+$ and on $X^{s'}_t\times X^{s'}_t$, given by Step 2.
\end{proof}

In the new coordinates, the concentrated coupling $\mathcal B$ only appears multiplied by the small operator $\mathcal T$, and no derivative of $\Phi$ is lost. We now write the equation \eqref{eqeps} in the coordinates $Z$, and we isolate its principal part, which is triangular: the pair $\Phi$ evolves by the constant-coefficient operator $\mathcal H_\infty$, and the finite-dimensional coordinate $a$ evolves by $\Omega_\infty$ and is driven by $\Phi_2$.

\begin{proposition}[Normal form]
\label{propnormalform}
Let $h$ satisfy \eqref{assumptionh}. After increasing $T_0$, depending on $s$, the following holds for $t\geq T_0$ and $s'\in\{1,s\}$.
\begin{enumerate}
\item {\rm (}Structure{\rm ).} In the coordinates $(a,\Phi)$ of Lemma \ref{lemmatransformation}, item (2), the equation \eqref{eqeps} reads
$$\pa_ta=-i\Omega(t)a-i\Upsilon(t)k-i\iota_t^{-1}P_t\Pi_+\mathcal N$$
and
\begin{equation}
\label{structurephi}
\pa_t\Phi=-i\big(\mathcal H_0+\mathcal V(t)\big)\Phi-i(I+\mathcal T)^{-1}\Big(\big([\mathcal V,\mathcal T]+\mathcal B\mathcal T-i\pa_t\mathcal T\big)\Phi+\mathcal Y\Big),
\end{equation}
with $\mathcal Y$ given by \eqref{defY}. After the substitution of $\pa_ta$ in $\mathcal Y_2$, a function $\mathcal N$ enters the equation of $\Phi$ only through the term
\begin{equation}
\label{contributionN}
-i(I+\mathcal T)^{-1}\Big(\mathbf T_{\overline{z\Theta^2}}\Pi_+\mathcal N,\,-\mathcal N^\flat-\frac i2\D^{-1}\Sigma(t)\big(-i\,\iota_t^{-1}P_t\Pi_+\mathcal N\big)\Big),
\end{equation}
and the terms of the right hand sides of the two equations which are linear in $Z$ and contain neither $h$ nor $\mathcal N$ are of the form $\mathcal A(t)Z+\mathcal R(t)Z$, where
\begin{equation}
\label{defAt}
\mathcal A(t)Z:=\big(-i\Omega_\infty a-i\Upsilon(t)\Phi_2,\ -i\mathcal H_\infty\Phi\big)
\end{equation}
is the principal part, and $\mathcal R(t)$ is real-linear and does not depend on $h$. The terms which contain $h$ but not $Z$, together with the terms obtained by replacing $\mathcal N$ by $\mathcal N_A(\bar z\bar h)$, form a forcing $\mathcal E_h(t)$, which does not depend on $Z$. The terms obtained by replacing $\mathcal N$ by $\mathcal N_A(\eps)-\mathcal N_A(\bar z\bar h)$ form a term $\mathcal N_h(t,Z)$, with $\mathcal N_h(t,0)=0$.
\item {\rm (}Equation{\rm ).} Let $\mathcal I\subset[T_0,+\infty)$ be an interval. Then $\eps\in\mathcal C^1(\mathcal I,L^2)\cap\mathcal C(\mathcal I,H^1)$ solves \eqref{eqeps} on $\mathcal I$ if and only if $Z\in\mathcal C^1(\mathcal I,\mathcal Z^0)\cap\mathcal C(\mathcal I,\mathcal Z^1)$ solves
\begin{equation}
\label{eqZ}
\pa_tZ=\mathcal A(t)Z+\mathcal R(t)Z+\mathcal E_h(t)+\mathcal N_h(t,Z).
\end{equation}
\item {\rm (}Estimates{\rm ).} For $\|Z\|_{\mathcal Z^{s'}_t}\leq1$ and $\|Z'\|_{\mathcal Z^{s'}_t}\leq1$,
\begin{equation}
\label{boundsZ}
\begin{gathered}
\|\mathcal R(t)Z\|_{\mathcal Z^{s'}_t}\lesssim\lambda\ell\|Z\|_{\mathcal Z^{s'}_t},\qquad \|\mathcal E_h(t)\|_{\mathcal Z^{s'}_t}\lesssim\lambda\ell,\\
\|\mathcal N_h(t,Z)-\mathcal N_h(t,Z')\|_{\mathcal Z^{s'}_t}\lesssim\big(\lambda\ell+\|Z\|_{\mathcal Z^{s'}_t}+\|Z'\|_{\mathcal Z^{s'}_t}\big)\|Z-Z'\|_{\mathcal Z^{s'}_t}.
\end{gathered}
\end{equation}
Moreover, in the components $(a,\Phi_1,\Phi_2)$,
\begin{equation}
\label{mainforcing}
\mathcal E_h(t)=-i\big(\Upsilon(t)h,\ \mathbf T_{\overline{\tilde A}^2}h,\ r_h\big)+\mathcal E_h^{\rm rem}(t),\qquad \|\mathcal E_h^{\rm rem}(t)\|_{\mathcal Z^{s'}_t}\lesssim(\lambda\ell)^2,
\end{equation}
where the $\Phi$-component of $\mathcal E_h^{\rm rem}(t)$ is the sum of the term \eqref{contributionN} with $\mathcal N=\mathcal N_A(\bar z\bar h)$ and of
$$-i\big((I+\mathcal T)^{-1}-I\big)\big(\mathbf T_{\overline{\tilde A}^2}h,r_h\big)+\frac12(I+\mathcal T)^{-1}\Big(0,\D^{-1}\Sigma(t)\big(i\Upsilon(t)h\big)\Big).$$
The implicit constants only depend on $s$ and on $C_h$.
\end{enumerate}
\end{proposition}

\begin{proof}
\noindent{\bf Step 1:} {\it Structure.} The equation of $a$ is the first equation of \eqref{system}. We substitute $k=h+\check k+\frac12\D^{-1}\Sigma(t)a$ into the two other equations of \eqref{system}. In the equation of $\check k$, we have $i\pa_t\check k=i\pa_tk-i\pa_th-\frac i2\D^{-1}\big(\pa_t[\Sigma(t)]a+\Sigma(t)\pa_ta\big)$; the term $-2\D\cdot\frac12\D^{-1}\Sigma(t)a$ coming from $-2\D k$ cancels the source $\Sigma(t)a$ exactly, and the terms containing $h$, together with the term $-2K\sqrt Q\,\bar zF$ of $-g^\flat$, form $r_h$, by \eqref{defrh}. Since $\mathbf T_{\overline{\tilde A}^2}=\overline{\tilde c}^{\,2}+\mathbf T_{\overline{\tilde A^2-\tilde c^2}}$ and $\tilde A^2=\tilde c^{\,2}+(\tilde A^2-\tilde c^{\,2})$, the pair $(f,\check k)$ solves $i\pa_t(f,\check k)=\mathcal H(t)(f,\check k)+\mathcal Y$, with $\mathcal Y$ given by \eqref{defY}. Next, $(f,\check k)=(I+\mathcal T)\Phi$ gives $(I+\mathcal T)\,i\pa_t\Phi=\mathcal H(I+\mathcal T)\Phi-i(\pa_t\mathcal T)\Phi+\mathcal Y$. By Lemma \ref{lemmatransformation}, item (1), $\mathcal H_0\mathcal T=\mathcal T\mathcal H_0-\mathcal B$, hence $\mathcal H(I+\mathcal T)=(I+\mathcal T)(\mathcal H_0+\mathcal V)+[\mathcal V,\mathcal T]+\mathcal B\mathcal T$, and we obtain \eqref{structurephi} after multiplication by $-i(I+\mathcal T)^{-1}$. This proves item (1).

\noindent{\bf Step 2:} {\it Equation.} All the transformations of Lemma \ref{lemmatransformation} are continuous on $L^2$ and on $H^1$ and of class $\mathcal C^1$ in time, and $\mathcal N_A$ maps $H^1$ into itself since $H^1$ is an algebra. Hence item (2) follows from item (1) and from Proposition \ref{proplinearized} with $r=0$.

\noindent{\bf Step 3:} {\it Linear terms.} We estimate the terms of item (1) in $\mathcal Z^{s'}_t$. By Lemma \ref{lemmatransformation}, item (1), Lemma \ref{lemmasmoothing}, items (1) and (2), and \eqref{concentratedpart}, the operator $\mathcal B(t)$ is bounded on $X^{s'}_t\times X^{s'}_t$, and the operator $[\mathcal V,\mathcal T]+\mathcal B\mathcal T-i\pa_t\mathcal T$ has norm $O(\lambda\ell)$ on this space. Since $|\tilde c-\tilde c_\infty|=\sqrt{2K}|p-e^{-i\vartheta}|\lesssim\lambda$ by Lemma \ref{lemmaplambda}, item (3), we have $\|(\mathcal V(t)-\mathcal V_\infty)\Phi\|_{X^{s'}_t}\lesssim\lambda\|\Phi\|_{X^{s'}_t}$. By Proposition \ref{proplinearized}, item (3), $|\Omega(t)a-\Omega_\infty a|\lesssim\lambda|a|$. Since $k-h-\Phi_2=[\mathcal T\Phi]_2+\frac12\D^{-1}\Sigma(t)a$, Proposition \ref{proplinearized}, items (1) and (2), Lemma \ref{lemmatransformation}, item (1), and Lemma \ref{lemmasmoothing}, item (2), give $|\Upsilon(t)(k-h-\Phi_2)|\lesssim\lambda\ell(|a|+\|\Phi\|_\W)$. By Proposition \ref{proplinearized}, item (2), and Lemma \ref{lemmasmoothing}, item (2), the functions $\D^{-1}\Sigma(t)a$, $\D^{-1}(\pa_t[\Sigma(t)]a)$ and $\D^{-1}\Sigma(t)a'$, $a'\in\C^{2J+1}$, have $X^{s'}_t$ norms, hence Wiener norms, $\lesssim\lambda\ell(|a|+|a'|)$. Since $\Sigma(t)$ is additive, the part of $\Sigma(t)\pa_ta$ which is linear in $Z$ and contains neither $h$ nor $\mathcal N$ is $\Sigma(t)a'$, with $a':=-i\Omega(t)a-i\Upsilon(t)(k-h)$, and $|a'|\lesssim|a|+\|k-h\|_\W\lesssim|a|+\|\Phi\|_\W$, by Proposition \ref{proplinearized}, items (1) and (3), and since $k-h=\Phi_2+[\mathcal T\Phi]_2+\frac12\D^{-1}\Sigma(t)a$. Together with Lemma \ref{lemmasmoothing}, item (1), and $\|(I+\mathcal T)^{-1}\|\leq2$, this proves the bound on $\mathcal R(t)Z$ in \eqref{boundsZ}.

\noindent{\bf Step 4:} {\it Forcing and nonlinear terms.} The terms of $\mathcal E_h(t)$ which do not contain $\mathcal N$ are $-i\Upsilon(t)h$, $-i(I+\mathcal T)^{-1}(\mathbf T_{\overline{\tilde A}^2}h,r_h)$, and the contribution of $\Sigma(t)\pa_ta$ through the term $-i\Upsilon(t)h$ of $\pa_ta$, which is
$$-i(I+\mathcal T)^{-1}\Big(0,-\frac i2\D^{-1}\Sigma(t)\big(-i\Upsilon(t)h\big)\Big)=\frac12(I+\mathcal T)^{-1}\Big(0,\D^{-1}\Sigma(t)\big(i\Upsilon(t)h\big)\Big),$$
since $\Sigma(t)$ is real-linear. Since $\|(I+\mathcal T)^{-1}-I\|\lesssim\lambda\ell$, $|\Upsilon(t)h|\lesssim\|h\|_\W$, and $\|\mathcal N_A(\bar z\bar h)\|_{X^{s'}_t}\lesssim\|h\|_{X^{s'}_t}^2$ by the product estimate, we obtain \eqref{mainforcing}, and the bound on $\mathcal E_h(t)$ in \eqref{boundsZ}. The function $\mathcal N_A(\eps)-\mathcal N_A(\bar z\bar h)$ is, by \eqref{defNA}, a sum of products of three factors among $\eps-\bar z\bar h$, $\bar z\bar h$, $A$ and their conjugates, each product containing at least one factor $\eps-\bar z\bar h$ or its conjugate, and at most one factor $A$ or $\overline A$. By the product estimate of Lemma \ref{lemmasmoothing}, item (1), by Lemma \ref{lemmatransformation}, item (2), and by $\|\bar z\bar h\|_{X^{s'}_t}\lesssim\lambda\ell$, this gives the Lipschitz bound of \eqref{boundsZ}, since the maps which send $\mathcal N$ to the corresponding terms of item (1) are bounded from $X^{s'}_t$ to $\mathcal Z^{s'}_t$, by Lemma \ref{lemmacoordinates}, item (3), and Lemma \ref{lemmasmoothing}, items (1) and (2).
\end{proof}


\subsection{The constant-coefficient propagator}


The operator $\mathcal H_\infty$ acts on each Fourier mode $n\geq0$ of the pair $\Phi$ by a $2\times2$ matrix. Its eigenvalues are real: the off-diagonal entries have modulus $|\tilde c_\infty|^2=M$, while the half-difference of the diagonal entries is $\D_n+M-2K\geq M$, since $\D_n\geq N+K\geq2K$.

\begin{lemma}[The constant-coefficient propagator]
\label{lemmahardygroup}
For $n\geq0$, let
\begin{equation}
\label{defHn}
\begin{gathered}
\mathcal H_\infty(n):=\begin{pmatrix}M-2K&\overline{\tilde c_\infty}^{\,2}\\-\tilde c_\infty^{\,2}&-2\D_n-(M-2K)\end{pmatrix},\\
\nu_n:=\sqrt{\big(N(n+1)-K\big)\big(N(n+1)-K+2M\big)},
\end{gathered}
\end{equation}
so that $\widehat{(\mathcal H_\infty\Phi)}(n)=\mathcal H_\infty(n)\hat\Phi(n)$.
\begin{enumerate}
\item {\rm (}Explicit propagator{\rm ).} $(\mathcal H_\infty(n)+\D_n)^2=\nu_n^2I$, the eigenvalues of $\mathcal H_\infty(n)$ are $-\D_n\pm\nu_n$, and for $t\in\R$,
\begin{equation}
\label{explicitgroup}
e^{-it\mathcal H_\infty(n)}=e^{it\D_n}\Big(\cos(\nu_nt)\,I-i\,\frac{\sin(\nu_nt)}{\nu_n}\big(\mathcal H_\infty(n)+\D_n\big)\Big),
\end{equation}
where $\frac{\sin(\nu_nt)}{\nu_n}=t$ if $\nu_n=0$, which happens only for $n=0$ and $N=K$.
\item {\rm (}Bounds{\rm ).} For $t\in\R$,
\begin{equation}
\label{boundgroup}
\sup_{n\geq0}\big|e^{-it\mathcal H_\infty(n)}\big|\lesssim\la t\ra,\qquad \sup_{n\geq1}\big|e^{-it\mathcal H_\infty(n)}\big|\lesssim1.
\end{equation}
Consequently, $\|e^{-it\mathcal H_\infty}\Phi\|_\W\lesssim\la t\ra\|\Phi\|_\W$ and $\|e^{-it\mathcal H_\infty}\Phi\|_{H^s}\lesssim\la t\ra\|\Phi\|_{H^s}$.
\item {\rm (}High modes{\rm ).} For $n\geq1$ and $t\in\R$, the entries of $e^{it\mathcal H_\infty(n)}$ satisfy
\begin{equation}
\label{highmodes}
\big|\big(e^{it\mathcal H_\infty(n)}\big)_{11}-e^{i(M-2K)t}\big|\lesssim\frac{|t|}{n}+\frac1{n^2},\qquad \big|\big(e^{it\mathcal H_\infty(n)}\big)_{12}\big|\lesssim\frac1n.
\end{equation}
\end{enumerate}
\end{lemma}

\begin{proof}
The matrix $\mathcal H_\infty(n)+\D_n$ has trace zero and determinant $-(\D_n+M-2K)^2+|\tilde c_\infty|^4=-(\D_n+M-2K)^2+M^2=-\nu_n^2$, since $(\D_n+M-2K)^2-M^2=(\D_n-2K)(\D_n-2K+2M)$ and $\D_n-2K=N(n+1)-K\geq0$. By the Cayley--Hamilton theorem, $(\mathcal H_\infty(n)+\D_n)^2=\nu_n^2I$, which gives the eigenvalues and \eqref{explicitgroup}. For $n\geq1$, $N(n+1)-K\geq\frac{N(n+1)}{2}$ since $K\leq N$, hence $\nu_n\geq\frac{N(n+1)}{2}$, while $|\mathcal H_\infty(n)+\D_n|\lesssim n+1$; therefore the matrices in \eqref{explicitgroup} are bounded uniformly in $n\geq1$ and $t\in\R$. For $n=0$, $|\frac{\sin(\nu_0t)}{\nu_0}|\leq|t|$. The bounds on $\W$ and $H^s$ follow, since $e^{-it\mathcal H_\infty}$ acts mode by mode. For item (3), let $n\geq1$ and $\delta_n:=\D_n+M-2K$, so that $\nu_n=\delta_n\big(1-\frac{M^2}{\delta_n^2}\big)^{\frac12}$ and $\delta_n\geq\frac{N(n+1)}2$. Then $\frac{\delta_n}{\nu_n}=1+O(n^{-2})$ and $\nu_n-\delta_n=O(n^{-1})$. By \eqref{explicitgroup} with $t$ replaced by $-t$,
\begin{align*}
\big(e^{it\mathcal H_\infty(n)}\big)_{11}&=e^{-it\D_n}\Big(\cos(\nu_nt)+i\,\frac{\delta_n}{\nu_n}\sin(\nu_nt)\Big)\\
&=e^{i(\nu_n-\D_n)t}+O(n^{-2}),
\end{align*}
and $e^{i(\nu_n-\D_n)t}=e^{i(M-2K)t}e^{i(\nu_n-\delta_n)t}=e^{i(M-2K)t}+O(\frac{|t|}{n})$. Finally, $\big(e^{it\mathcal H_\infty(n)}\big)_{12}=ie^{-it\D_n}\frac{\sin(\nu_nt)}{\nu_n}\,\overline{\tilde c_\infty}^{\,2}=O(\frac1n)$.
\end{proof}


\section{\texorpdfstring{Backward integration from $t=+\infty$}{Backward integration from t = infinity}}
\label{sectionbackward}


In this section we solve the equation \eqref{eqZ} with the first correction $h=h_1$ backwards from $t=+\infty$, and we prove Theorems \ref{thmmain} and \ref{thmallmass}. For a function $h$ satisfying \eqref{assumptionh}, we write the equation \eqref{eqZ} as
\begin{equation}
\label{eqZG}
\pa_tZ=\mathcal A(t)Z+\mathcal G_h(t,Z),\qquad \mathcal G_h(t,Z):=\mathcal R(t)Z+\mathcal E_h(t)+\mathcal N_h(t,Z).
\end{equation}
By \eqref{boundsZ}, $\mathcal G_h(t,Z)$ is of size $\lambda\ell$ when $Z$ is of size $\lambda\ell$. The principal part $\mathcal A(t)$ is triangular: the pair $\Phi$ evolves autonomously by the explicit group $e^{-it\mathcal H_\infty}$ of Lemma \ref{lemmahardygroup}, and the finite-dimensional coordinate $a$ evolves by $e^{-it\Omega_\infty}$, driven by $\Phi_2$ through the bounded map $\Upsilon(t)$. We first estimate the propagator of $\mathcal A(t)$ backwards in time (Lemma \ref{lemmabackward}) and the corresponding final-value integral (Lemma \ref{lemmaJ}), and we then solve the final-value problem for \eqref{eqZG} by the contraction principle (Proposition \ref{propfinalvalue}), as in the construction of multi-solitary waves \cite{merlemulti, Martelmulti, MMnls, KMRhartree}. The decay rate $\lambda\ell\sim\la t\ra e^{-\Lambda t}$ of the forcing must compensate the loss $e^{(s-\frac12)\Lambda(\tau-t)}$ of the propagator between the times $\tau$ and $t<\tau$, which is the origin of the restriction $s<\frac32$.


\subsection{The backward propagator and the final-value integral}


For $T_0\leq t\leq\tau$ and $\mathring Z:=(\mathring a,\mathring\Phi)\in\mathcal Z^r$, $r\geq0$, we denote by $\mathcal U(t,\tau)\mathring Z$ the value at time $t$ of the solution $Z$ of $\pa_tZ=\mathcal A(t)Z$ with $Z(\tau)=\mathring Z$. Since $\mathcal A(t)$ is triangular, the propagator is explicit: $\mathcal U(t,\tau)\mathring Z=(a(t),\Phi(t))$, with
\begin{equation}
\label{defpropagator}
\Phi(t)=e^{i(\tau-t)\mathcal H_\infty}\mathring\Phi,\qquad a(t)=e^{i(\tau-t)\Omega_\infty}\mathring a+\int_t^\tau e^{i(t'-t)\Omega_\infty}\big(i\,\Upsilon(t')\Phi_2(t')\big)\,dt',
\end{equation}
where we recall that $e^{it\Omega_\infty}$ is the real-linear flow defined before Lemma \ref{lemmafiniteblock}. The group $e^{-it\mathcal H_\infty}$ acts mode by mode by the matrices \eqref{explicitgroup}, and it is strongly continuous on $(H^r_+)^2$ for every $r$. The solution of $\pa_tZ=\mathcal A(t)Z$ with a given value at a given time is unique, since $\Phi$ is given by the group and $a$ then solves a linear differential equation in $\C^{2J+1}$ with a continuous forcing. Hence $\mathcal U(t,\tau)\mathcal U(\tau,\tau')=\mathcal U(t,\tau')$ for $t\leq\tau\leq\tau'$, and $\mathcal U(t,\tau)$ is strongly continuous in $(t,\tau)$ on $\mathcal Z^r$.

\begin{lemma}[Backward propagator]
\label{lemmabackward}
Let $\frac12<s<\frac32$. For $T_0\leq t\leq\tau$ and $\mathring Z\in\mathcal Z^s$,
\begin{equation}
\label{backward}
\|\mathcal U(t,\tau)\mathring Z\|_{\mathcal Z^s_t}\lesssim\la\tau-t\ra^4\,e^{(s-\frac12)\Lambda(\tau-t)}\,\|\mathring Z\|_{\mathcal Z^s_\tau}.
\end{equation}
\end{lemma}

\begin{proof}
Let $\Phi(t'):=e^{i(\tau-t')\mathcal H_\infty}\mathring\Phi$ for $t\leq t'\leq\tau$. By Lemma \ref{lemmahardygroup}, $\|\Phi(t')\|_\W\lesssim\la\tau-t'\ra\|\mathring\Phi\|_\W$ and $\|\Phi(t)\|_{H^s}\lesssim\la\tau-t\ra\|\mathring\Phi\|_{H^s}$. By \eqref{elementarylambda}, $\lambda(t)\leq e^{\Lambda(\tau-t)}\lambda(\tau)$, hence
\begin{align*}
\|\Phi(t)\|_{X^s_t}&\lesssim\la\tau-t\ra\Big(\|\mathring\Phi\|_\W+\Big(\frac{\lambda(t)}{\lambda(\tau)}\Big)^{s-\frac12}\lambda(\tau)^{s-\frac12}\|\mathring\Phi\|_{H^s}\Big)\\
&\leq\la\tau-t\ra\,e^{(s-\frac12)\Lambda(\tau-t)}\|\mathring\Phi\|_{X^s_\tau}.
\end{align*}
For the finite-dimensional component, \eqref{backwardfinite} and Proposition \ref{proplinearized}, item (1), give
$$|a(t)|\lesssim\la\tau-t\ra^2|\mathring a|+\int_t^\tau\la t'-t\ra^2\|\Phi_2(t')\|_\W\,dt'\lesssim\la\tau-t\ra^2|\mathring a|+\la\tau-t\ra^4\|\mathring\Phi\|_\W,$$
which proves \eqref{backward}.
\end{proof}

The Wiener part of the norm only loses a polynomial factor. The exponential factor in \eqref{backward} comes from the weight $\lambda(t)^{s-\frac12}$ in \eqref{defXst} alone: the propagator is bounded on $H^s$ up to a polynomial factor, but the weight grows backwards in time. By \eqref{ratio}, a forcing of size $\lambda(\tau)\ell(\tau)$ at time $\tau$ is smaller than $\lambda(t)\ell(t)$ by the factor $\la\tau-t\ra e^{-\Lambda(\tau-t)}$, which beats the exponential loss $e^{(s-\frac12)\Lambda(\tau-t)}$ of \eqref{backward}, in an integrable way with respect to $\tau\in[t,+\infty)$, if and only if $s<\frac32$.

We now fix $\frac12<s<\frac32$, and for $T\geq T_0$ we consider the Banach space $\mathcal X_T$ of the functions $Z\in\mathcal C([T,+\infty),\mathcal Z^{\max(s,1)})$ such that
\begin{equation}
\label{defET}
\|Z\|_{\mathcal X_T}:=\sup_{t\geq T}\frac{\max\big(\|Z(t)\|_{\mathcal Z^1_t},\|Z(t)\|_{\mathcal Z^s_t}\big)}{\lambda(t)\ell(t)}<+\infty.
\end{equation}
The index $1$ is included in the norm in order to compare the solutions obtained for different values of $s$. For $Y\in\mathcal X_T$, we define the final-value integral
\begin{equation}
\label{defJ}
(\mathcal JY)(t):=-\int_t^{+\infty}\mathcal U(t,\tau)Y(\tau)\,d\tau,\qquad t\geq T.
\end{equation}
Formally, $\mathcal JY$ is the solution of $\pa_tZ=\mathcal A(t)Z+Y$ which vanishes at $t=+\infty$. The decay of the elements of $\mathcal X_T$ makes this statement rigorous.

\begin{lemma}[The final-value integral]
\label{lemmaJ}
Let $\frac12<s<\frac32$, $T\geq T_0$ and $Y\in\mathcal X_T$. The integral \eqref{defJ} converges absolutely in $\mathcal Z^{\max(s,1)}$, and the following holds.
\begin{enumerate}
\item {\rm (}Bound{\rm ).} $\mathcal JY\in\mathcal X_T$ and $\|\mathcal JY\|_{\mathcal X_T}\leq C_1\|Y\|_{\mathcal X_T}$, where the constant $C_1$ depends on $s$, but not on $T$.
\item {\rm (}Duhamel formula{\rm ).} $\mathcal JY\in\mathcal C^1([T,+\infty),\mathcal Z^0)$, and $\pa_t(\mathcal JY)=\mathcal A(t)\,\mathcal JY+Y$ on $[T,+\infty)$.
\item {\rm (}Uniqueness{\rm ).} If $Z\in\mathcal X_T\cap\mathcal C^1([T,+\infty),\mathcal Z^0)$ solves $\pa_tZ=\mathcal A(t)Z+Y$ on $[T,+\infty)$, then $Z=\mathcal JY$.
\end{enumerate}
\end{lemma}

\begin{proof}
Let $s'\in\{1,s\}$ and $T\leq t\leq\tau$. By \eqref{backward} and \eqref{ratio},
\begin{equation}
\label{decayU}
\begin{aligned}
\|\mathcal U(t,\tau)Y(\tau)\|_{\mathcal Z^{s'}_t}&\lesssim\la\tau-t\ra^4e^{(s'-\frac12)\Lambda(\tau-t)}\lambda(\tau)\ell(\tau)\|Y\|_{\mathcal X_T}\\
&\lesssim\la\tau-t\ra^5e^{-(\frac32-s')\Lambda(\tau-t)}\lambda(t)\ell(t)\|Y\|_{\mathcal X_T}.
\end{aligned}
\end{equation}
Since $s'\leq\max(s,1)<\frac32$, the right hand side is integrable in $\tau\in[t,+\infty)$. Since the norms of $\mathcal Z^{s'}_t$ and $\mathcal Z^{s'}$ are equivalent for fixed $t$, uniformly for $t$ in bounded intervals, this proves the absolute convergence, locally uniformly in $t$, and the bound of item (1), with $C_1\lesssim\int_0^{+\infty}\la\tau'\ra^5e^{-(\frac32-\max(s,1))\Lambda\tau'}\,d\tau'$. For $T\leq t\leq t_1$, the identity $\mathcal U(t,\tau)=\mathcal U(t,t_1)\mathcal U(t_1,\tau)$ gives
$$(\mathcal JY)(t)=\mathcal U(t,t_1)(\mathcal JY)(t_1)-\int_t^{t_1}\mathcal U(t,\tau)Y(\tau)\,d\tau,$$
which is the Duhamel formula on $[t,t_1]$ for the equation $\pa_tZ=\mathcal A(t)Z+Y$. Since $\mathcal U$ is strongly continuous on $\mathcal Z^{\max(s,1)}$ and $Y$ is continuous with values in $\mathcal Z^{\max(s,1)}\subset\mathcal Z^1$, this identity shows that $\mathcal JY$ is continuous with values in $\mathcal Z^{\max(s,1)}$, which completes the proof of item (1), and it gives item (2). For item (3), the Duhamel formula for $Z$ on $[t,\tau]$ reads $Z(t)=\mathcal U(t,\tau)Z(\tau)-\int_t^\tau\mathcal U(t,\tau')Y(\tau')\,d\tau'$, and \eqref{decayU}, applied to $Z$ instead of $Y$, shows that $\mathcal U(t,\tau)Z(\tau)\to0$ in $\mathcal Z^1$ as $\tau\to+\infty$. Hence $Z(t)=(\mathcal JY)(t)$.
\end{proof}


\subsection{The final-value problem}


By Lemma \ref{lemmaJ}, a solution of \eqref{eqZG} which decays like $\lambda\ell$ is a fixed point of the map $Z\mapsto\mathcal J(\mathcal G_h(\cdot,Z))$, and by \eqref{boundsZ} this map is a contraction for large times, since the forcing is quadratic in $\lambda\ell$ up to the term $\mathcal E_h$.

\begin{proposition}[Final-value problem]
\label{propfinalvalue}
\begin{enumerate}
\item {\rm (}Existence{\rm ).} For every $\frac12<s<\frac32$, there exist $T_s\geq T_0$ and a solution $\eps\in\mathcal C^1([T_s,+\infty),L^2)\cap\mathcal C([T_s,+\infty),H^{\max(s,1)})$ of \eqref{eqeps} such that, for $t\geq T_s$,
\begin{equation}
\label{boundeps}
\|\eps(t)\|_{X^1_t}+\|\eps(t)\|_{X^s_t}\lesssim\lambda(t)\ell(t).
\end{equation}
\item {\rm (}Uniqueness{\rm ).} For every $R>0$, there exists $\tilde T_R\geq T_0$, where $T_0$ is the time of Proposition \ref{propnormalform} for $s=1$, such that the following holds. Let $T\geq\tilde T_R$, and let $\eps,\eps'\in\mathcal C^1([T,+\infty),L^2)\cap\mathcal C([T,+\infty),H^1)$ be two solutions of \eqref{eqeps} such that $\|\eps(t)\|_{X^1_t}+\|\eps'(t)\|_{X^1_t}\leq R\lambda(t)\ell(t)$ for $t\geq T$. Then $\eps=\eps'$.
\item {\rm (}Contraction{\rm ).} Let $\frac12<s<\frac32$, and let $h$ satisfy \eqref{assumptionh}. There exists $R_0\geq1$, depending on $s$ and $C_h$, and, for every $R\geq R_0$, there exists $T_{s,R}\geq T_0$, depending on $s$, $R$ and $C_h$, such that the following holds for $T\geq T_{s,R}$. The map $\mathcal J_h$ defined by $\mathcal J_h(Z):=\mathcal J\big(\mathcal G_h(\cdot,Z)\big)$ maps the closed ball of radius $R$ of $\mathcal X_T$ into the closed ball of radius $R_0$, and it is Lipschitz on the ball of radius $R$, with constant $\frac12$. A function $Z$ of this ball is a fixed point of $\mathcal J_h$ if and only if $Z\in\mathcal C^1([T,+\infty),\mathcal Z^0)$ and $Z$ solves \eqref{eqZG} on $[T,+\infty)$, that is, if and only if $Z$ is the coordinate vector, given by Lemma \ref{lemmatransformation}, item (2), of a solution $\eps\in\mathcal C^1([T,+\infty),L^2)\cap\mathcal C([T,+\infty),H^1)$ of \eqref{eqeps}.
\end{enumerate}
\end{proposition}

\begin{proof}
\noindent{\bf Step 1:} {\it The fixed point problem.} Let $s$ and $h$ be as in item (3), let $R\geq1$, and let $T\geq T_0$ be such that $R\lambda(T)\ell(T)\leq1$. By Proposition \ref{propnormalform}, item (2), with $\mathcal I=[T,+\infty)$, a function $\eps\in\mathcal C^1([T,+\infty),L^2)\cap\mathcal C([T,+\infty),H^1)$ solves \eqref{eqeps} if and only if its coordinates $Z$ solve \eqref{eqZG} in $\mathcal C^1([T,+\infty),\mathcal Z^0)\cap\mathcal C([T,+\infty),\mathcal Z^1)$. Let $Z$ belong to the closed ball of radius $R$ of $\mathcal X_T$. Since $\lambda\ell$ is decreasing on $[0,+\infty)$, by Lemma \ref{lemmaplambda}, item (4), we have $\|Z(t)\|_{\mathcal Z^{s'}_t}\leq R\lambda(t)\ell(t)\leq1$ for $s'\in\{1,s\}$ and $t\geq T$, so that, by \eqref{boundsZ}, $\mathcal G_h(\cdot,Z)$ is bounded in the norm of $\mathcal X_T$. It is continuous in time with values in $\mathcal Z^{\max(s,1)}$, since $h\in\mathcal C^1([T_0,+\infty),H^3_+)$, and since all the operators of Proposition \ref{propnormalform} are continuous in $t$; for $\pa_t\mathcal T(t)$ and $\pa_t[\Sigma(t)]$, this uses the second time derivatives in \eqref{envelopeGj}. Hence $\mathcal G_h(\cdot,Z)\in\mathcal X_T$, and Lemma \ref{lemmaJ}, items (2) and (3), shows that $Z\in\mathcal C^1([T,+\infty),\mathcal Z^0)$ solves \eqref{eqZG} if and only if $Z=\mathcal J_h(Z)$. This proves the last statement of item (3).

\noindent{\bf Step 2:} {\it Contraction.} Let $C_2\geq1$ be larger than the implicit constants of \eqref{boundsZ} for the indices $1$ and $s$, which only depend on $s$ and $C_h$, and let $R_0:=2C_1C_2+1$. For $R\geq R_0$, we choose $T_{s,R}\geq T_0$ such that $C_1C_2(2+2R+R^2)\lambda(T_{s,R})\ell(T_{s,R})\leq\frac12$ and $R\lambda(T_{s,R})\ell(T_{s,R})\leq1$; since $\lambda\ell$ is decreasing on $[0,+\infty)$, the same inequalities hold at every $T\geq T_{s,R}$, and Step 1 applies. Let $T\geq T_{s,R}$ and $Z,Z'\in\mathcal X_T$ with $\|Z\|_{\mathcal X_T},\|Z'\|_{\mathcal X_T}\leq R$, so that $\|Z(t)\|_{\mathcal Z^{s'}_t}\leq1$ for $s'\in\{1,s\}$ and $t\geq T$. By \eqref{boundsZ}, for $t\geq T$ and $s'\in\{1,s\}$,
\begin{gather*}
\|\mathcal G_h(t,Z(t))\|_{\mathcal Z^{s'}_t}\leq C_2\lambda\ell\big(1+(2R+R^2)\lambda\ell\big),\\
\|\mathcal G_h(t,Z(t))-\mathcal G_h(t,Z'(t))\|_{\mathcal Z^{s'}_t}\leq C_2(2+2R)(\lambda\ell)^2\|Z-Z'\|_{\mathcal X_T},
\end{gather*}
where $\lambda\ell=\lambda(t)\ell(t)\leq\lambda(T)\ell(T)$. By Lemma \ref{lemmaJ}, item (1), $\|\mathcal J_h(Z)\|_{\mathcal X_T}\leq C_1C_2+\frac12\leq R_0$ and $\|\mathcal J_h(Z)-\mathcal J_h(Z')\|_{\mathcal X_T}\leq\frac12\|Z-Z'\|_{\mathcal X_T}$. This proves item (3).

\noindent{\bf Step 3:} {\it Existence.} We apply item (3) with $h=h_1$, which satisfies \eqref{assumptionh}. Let $T_s:=T_{s,R_0}$ and let $Z$ be the fixed point of $\mathcal J_{h_1}$ in the ball of radius $R_0$ of $\mathcal X_{T_s}$. By item (3), $Z$ is the coordinate vector of a solution $\eps\in\mathcal C^1([T_s,+\infty),L^2)\cap\mathcal C([T_s,+\infty),H^1)$ of \eqref{eqeps}, and $\eps\in\mathcal C([T_s,+\infty),H^{\max(s,1)})$ by Lemma \ref{lemmatransformation}, item (2), since $Z\in\mathcal C([T_s,+\infty),\mathcal Z^{\max(s,1)})$ and $h_1\in\mathcal C([T_0,+\infty),H^2_+)$. Moreover, for $s'\in\{1,s\}$, $\|\eps(t)\|_{X^{s'}_t}\leq\|\eps(t)-\bar z\bar h_1(t)\|_{X^{s'}_t}+\|\bar z\bar h_1(t)\|_{X^{s'}_t}\lesssim\lambda(t)\ell(t)$, by Lemma \ref{lemmatransformation}, item (2), Lemma \ref{lemmasmoothing}, item (1), and \eqref{boundh1}.

\noindent{\bf Step 4:} {\it Uniqueness.} We apply item (3) with $s=1$ and $h=h_1$, so that $\mathcal X_T$ is defined with the norm $\sup_t\frac{\|Z(t)\|_{\mathcal Z^1_t}}{\lambda(t)\ell(t)}$. Let $\eps,\eps'$ be as in item (2), and let $Z,Z'$ be their coordinates. By Lemma \ref{lemmatransformation}, item (2), and \eqref{boundh1}, $\|Z\|_{\mathcal X_T}+\|Z'\|_{\mathcal X_T}\leq C_3(R+1)$ for a constant $C_3$. Let $R_1:=\max(R_0,C_3(R+1))$ and $\tilde T_R:=T_{1,R_1}$. If $T\geq\tilde T_R$, then, by item (3), $Z$ and $Z'$ are fixed points of $\mathcal J_{h_1}$ in the ball of radius $R_1$ of $\mathcal X_T$, on which $\mathcal J_{h_1}$ is Lipschitz with constant $\frac12$. Hence $Z=Z'$ and $\eps=\eps'$.
\end{proof}


\subsection{\texorpdfstring{Proof of Theorems \ref{thmallmass} and \ref{thmmain}}{Proof of Theorems 1.1 and 1.2}}


The solution of Theorem \ref{thmmain} is obtained from the function $\eps$ of Proposition \ref{propfinalvalue}, item (1), with $s=1$. The global well-posedness of \eqref{halfwave} extends it to all times, and the uniqueness statement of Proposition \ref{propfinalvalue}, item (2), identifies it with the solutions obtained for the other values of $s<\frac32$, which gives its regularity. The growth of its norms and its conservation laws are then inherited from those of $U$, through the approximation \eqref{approxmain}.

\begin{proof}[Proof of Theorem \ref{thmmain}]
\noindent{\bf Step 1:} {\it The solution.} Let $\eps$ be the solution given by Proposition \ref{propfinalvalue}, item (1), with $s=1$, defined on $[T_1,+\infty)$. Let $v:=A+\eps$, and let $u$ be defined by \eqref{comoving}. By Lemma \ref{lemmacomoving}, $u\in\mathcal C([T_1,+\infty),H^1)$ solves \eqref{halfwave}. By the global well-posedness of \eqref{halfwave} in $H^{\frac12}$ and the persistence of regularity \cite[Proposition 1]{GeGr2012}, $u$ extends to a global solution $u\in\mathcal C(\R,H^1)$. The Fourier coefficients of $u(T_1)$ vanish outside $-K+N\Z$, by \eqref{comovingfourier}. Therefore $\tilde u(t,x):=e^{\frac{2i\pi K}{N}}u(t,x+\frac{2\pi}{N})$, which solves \eqref{halfwave} by the invariance under translations and phase rotations, satisfies $\tilde u(T_1)=u(T_1)$, hence $\tilde u=u$ by uniqueness. In Fourier variables, this means $(e^{\frac{2i\pi(K+m)}{N}}-1)\hat u(t,m)=0$, so that $\hat u(t,m)=0$ for $m\notin-K+N\Z$ and all $t\in\R$. Hence \eqref{comovingfourier} defines $v(t)$ for all $t\in\R$, and we let $\eps:=v-A$.

\noindent{\bf Step 2:} {\it Regularity and approximation.} Let $\frac12<s<\frac32$, and let $\eps_s$ be the solution given by Proposition \ref{propfinalvalue}, item (1), on $[T_s,+\infty)$. By \eqref{boundeps}, there is $R>0$ such that $\|\eps(t)\|_{X^1_t}+\|\eps_s(t)\|_{X^1_t}\leq R\lambda(t)\ell(t)$ for $t\geq\max(T_1,T_s)$. By Proposition \ref{propfinalvalue}, item (2), $\eps=\eps_s$ on $[T,+\infty)$, with $T:=\max(T_1,T_s,\tilde T_R)$. Hence $u(T)\in H^{\max(s,1)}$, and by the persistence of regularity \cite[Proposition 1]{GeGr2012}, $u\in\mathcal C(\R,H^{\max(s,1)})$. Since $s<\frac32$ is arbitrary, $u\in\bigcap_{r<\frac32}\mathcal C(\R,H^r)$. Moreover, $\|\eps(t)\|_{X^s_t}\lesssim\lambda(t)\ell(t)$ for $t\geq T$ by \eqref{boundeps}, and also for $0\leq t\leq T$, since $t\mapsto\|\eps(t)\|_{X^s_t}$ is continuous and $\lambda\ell\geq\lambda(T)\ell(T)>0$ on $[0,T]$. By Lemma \ref{lemmacomoving} and \eqref{comparable}, $\|u(t)-U(t)\|_\W=\|\eps(t)\|_\W$ and $\|u(t)-U(t)\|_{H^s}\lesssim\|\eps(t)\|_{H^s}\leq\lambda^{\frac12-s}\|\eps(t)\|_{X^s_t}$. This proves \eqref{approxmain}.

\noindent{\bf Step 3:} {\it Growth.} By Lemma \ref{lemmacomoving} and \eqref{normsA},
\begin{align*}
\|U(t)\|_{H^s}^2&=\la K\ra^{2s}|c(t)|^2+\sum_{n\geq1}\la Nn-K\ra^{2s}|\hat A(t,n)|^2\\
&=N^{2s}Q\,\Gamma(2s+1)\sum_{j=1}^J|B'(\zeta_j)|^{2s-1}\lambda(t)^{1-2s}(1+o(1)),
\end{align*}
since $\lambda^{1-2s}\to+\infty$. By \eqref{approxmain}, $\|u(t)\|_{H^s}=\|U(t)\|_{H^s}+O(\lambda^{\frac32-s}\ell)=\|U(t)\|_{H^s}(1+o(1))$, and by \eqref{defplambda}, $\lambda(t)=4\sin^2\vartheta\,e^{-\Lambda t}(1+e^{-\Lambda t})^{-2}$. This proves \eqref{growthmain}.

\noindent{\bf Step 4:} {\it Conservation laws.} By Young's inequality for convolutions and H\"older's inequality,
$$\|f\|_{L^4}^2=\Big(\sum_n|\widehat{f^2}(n)|^2\Big)^{\frac12}\leq\Big(\sum_n|\hat f(n)|^{\frac43}\Big)^{\frac32}\leq\Big(\sum_n\la n\ra^{-2}\Big)^{\frac12}\|f\|_{H^{\frac12}}^2,$$
which is the Sobolev embedding $H^{\frac12}(\T)\subset L^4(\T)$. Hence the conservation laws \eqref{conservationlaws} are uniformly continuous on the bounded subsets of $H^{\frac12}$. By \eqref{boundH12} and \eqref{invariantsU}, $U(t)$ is bounded in $H^{\frac12}$, and by \eqref{approxmain} with $s=1$, $\|u(t)-U(t)\|_{H^{\frac12}}\lesssim\lambda(t)^{\frac12}\ell(t)\to0$. Since the conservation laws of $u(t)$ are independent of $t$, and those of $U(t)$ are given by \eqref{invariantsU}, we obtain \eqref{invariantsmain}.

\noindent{\bf Step 5:} {\it Uniqueness.} Let $u'\in\mathcal C(\R,H^1)$ be a solution of \eqref{halfwave} whose Fourier coefficients vanish outside $-K+N\Z$ and which satisfies \eqref{approxmain} for $s=1$, and let $v'$ and $\eps':=v'-A$ be defined by \eqref{comovingfourier}. Since $\pa_tu'=-i(|D|u'+|u'|^2u')\in\mathcal C(\R,L^2)$, we have $\eps'\in\mathcal C^1(\R,L^2)\cap\mathcal C(\R,H^1)$; by Lemma \ref{lemmacomoving}, $v'$ solves \eqref{eqv}, so that $\eps'$ solves \eqref{eqeps}. By Lemma \ref{lemmacomoving} and \eqref{comparable}, $\|\eps'(t)\|_\W=\|u'(t)-U(t)\|_\W\lesssim\lambda\ell$ and $\lambda^{\frac12}\|\eps'(t)\|_{H^1}\lesssim\lambda^{\frac12}\|u'(t)-U(t)\|_{H^1}\lesssim\lambda\ell$ for $t\geq0$. The same bounds hold for $\eps$, by Step 2. Hence there is $R>0$ such that $\|\eps(t)\|_{X^1_t}+\|\eps'(t)\|_{X^1_t}\leq R\lambda(t)\ell(t)$ for $t\geq0$, and Proposition \ref{propfinalvalue}, item (2), gives $\eps=\eps'$ on $[\tilde T_R,+\infty)$. Therefore $u=u'$ on $[\tilde T_R,+\infty)$, and $u=u'$ on $\R$ by uniqueness for the Cauchy problem in $H^{\frac12}$ \cite[Proposition 1]{GeGr2012}.
\end{proof}

\begin{proof}[Proof of Theorem \ref{thmallmass}]
Let $K=N=J=1$, $B(z)=z$, $M=M_0$ and $Q=M_0+2$. The condition \eqref{admissible} holds, since $\sqrt{M_0+2}-\sqrt2<\sqrt{M_0}$. By \eqref{deftheta}, $\cos\vartheta=\frac{2K+Q-M}{2\sqrt{2KQ}}=\sqrt{\frac{2}{M_0+2}}$, hence $\sin^2\vartheta=\frac{M_0}{M_0+2}$ and $\Lambda=2\sqrt{2KQ}\sin\vartheta=2\sqrt{2M_0}$. Since $B'=1$ and $\zeta_1=e^{i\vartheta}$, \eqref{growthmain} and \eqref{invariantsmain} give \eqref{growthallmass} and the values of $M(u)$, $P(u)=Q-M=2$ and $E(u)=\frac{Q+M}{2}+\frac{M^2}{4}=(1+\frac{M_0}2)^2$.
\end{proof}

The choice $Q=M+2K$ in the proof of Theorem \ref{thmallmass} maximizes the rate $\Lambda$ at fixed $M$ and $K$: by \eqref{deftheta},
$$\Lambda^2=8KQ-(2K+Q-M)^2=8KM-(Q-M-2K)^2.$$
Note also that the energy in Theorem \ref{thmallmass} tends to $1$ as $M_0\to0$: the solutions have small mass but not small energy. Appendix \ref{appendixinvariants} shows that the energies of the explicit family are bounded from below by a positive constant.


\subsection{Prescribed concentration points}


By Lemma \ref{lemmapeaks}, the $j$-th peak of the solution of Theorem \ref{thmmain} is located at the point $\zeta_j$ of the circle, in the variable $z:=e^{iy}$, and it concentrates at the scale $\lambda_j(t)=\frac{\lambda(t)}{|B'(\zeta_j)|}(1+O(\lambda))$; its half-width in the variable $x$ is computed in Section \ref{sectiondefect}. The Blaschke product $B$ can be chosen so that the points $\zeta_j$ and the relative scales $|B'(\zeta_j)|^{-1}$ are prescribed. The only constraint is the Clark identity $\sum_j|B'(\zeta_j)|^{-1}=1$ of Lemma \ref{lemmablaschke}, item (5), which fixes the sum of the relative scales.

\begin{corollary}[Prescribed concentration points and widths]
\label{corgeometry}
Let $K$, $N$, $M$, $Q$ satisfy \eqref{admissible}, let $\zeta_1,\dots,\zeta_J$ be distinct points of $\T$, and let $w_1,\dots,w_J>0$ with $\sum_jw_j=1$. Let
$$H(z):=\sum_{j=1}^Jw_j\,\frac{\zeta_j+z}{\zeta_j-z},\qquad B(z):=e^{i\vartheta}\,\frac{H(z)-1}{H(z)+1}.$$
Then $B$ is a Blaschke product of degree $J$ with $B(0)=0$, the solutions of $B(\zeta)=e^{i\vartheta}$ are the points $\zeta_j$, and $|B'(\zeta_j)|=w_j^{-1}$. Consequently, the $j$-th peak of the solution of Theorem \ref{thmmain} associated with $B$ is located at $\zeta_j$ and concentrates at the scale $\lambda_j(t)=w_j\lambda(t)(1+O(\lambda(t)))$.
\end{corollary}

\begin{proof}
The function $H$ is rational, with simple poles at the points $\zeta_j$, $H(0)=\sum_jw_j=1$, $\Re H(z)=\sum_jw_j\frac{1-|z|^2}{|\zeta_j-z|^2}>0$ in $\Disc$, and $\Re H=0$ on $\T\setminus\{\zeta_1,\dots,\zeta_J\}$. Since $\frac{H-1}{H+1}=1-\frac{2}{H+1}$ and $\frac1{H+1}$ vanishes at the poles of $H$, the function $B$ is rational and holomorphic in a neighborhood of $\overline\Disc$, $|B|<1$ in $\Disc$, $|B|=1$ on $\T$, and $B(0)=0$. Dividing $B$ by the Blaschke product formed with its zeros in $\Disc$, we obtain a function which is holomorphic and zero-free near $\overline\Disc$ and unimodular on $\T$, hence a unimodular constant, by the maximum principle applied to this function and to its inverse. Thus $B$ is a Blaschke product. Its degree as a rational function is the degree $J$ of $H$, since $H$ has $J$ simple poles in $\C$ and is bounded at infinity, and since $B$ is obtained from $H$ by a M\"obius transformation; hence $B$ has $J$ zeros in $\Disc$. Near $\zeta_j$, $H(z)=\frac{2w_j\zeta_j}{\zeta_j-z}(1+O(|z-\zeta_j|))$, hence
$$B(z)=e^{i\vartheta}-\frac{e^{i\vartheta}(\zeta_j-z)}{w_j\zeta_j}+O(|z-\zeta_j|^2),\qquad B(\zeta_j)=e^{i\vartheta},\qquad B'(\zeta_j)=\frac{e^{i\vartheta}}{w_j\zeta_j}.$$
By Lemma \ref{lemmablaschke}, item (1), $B$ takes the value $e^{i\vartheta}$ at exactly $J$ points of $\T$, so that the points $\zeta_j$ are all the solutions. Finally, $|B'(\zeta_j)|=w_j^{-1}$, and \eqref{lambdaj} gives the scale $\lambda_j$.
\end{proof}

The number of concentration points can be arbitrarily large on one fixed level set of the conservation laws.

\begin{corollary}[Arbitrarily many concentration points on one level set]
\label{corsameshell}
Let $J\geq1$, and let $\zeta_j$, $w_j$ and $B$ be as in Corollary \ref{corgeometry}, with $K=N=1$, $M=2$ and $Q=\frac2J$. Then the solution of Theorem \ref{thmmain} associated with $B$ satisfies
$$M(u)=2,\qquad P(u)=0,\qquad E(u)=3,$$
and its $H^s$ norm grows at the exponential rate $(s-\frac12)\Lambda_J$, $\Lambda_J:=\frac{2\sqrt{4J-1}}{J}$, for $\frac12<s<\frac32$.
\end{corollary}

\begin{proof}
The condition \eqref{admissible} reads $|\sqrt2-\sqrt{2/J}|<\sqrt2<\sqrt2+\sqrt{2/J}$, which holds for every $J\geq1$. By \eqref{deftheta}, $\cos\vartheta=\frac{2+Q-2}{2\sqrt{2Q}}=\frac{1}{2\sqrt J}$, hence $\Lambda=2\sqrt{2Q}\sin\vartheta=\frac{4}{\sqrt J}\sqrt{1-\frac{1}{4J}}=\frac{2\sqrt{4J-1}}{J}$. Finally, \eqref{invariantsmain} gives $P=JQ-M=0$ and $E=\frac{JQ+M}{2}+\frac{M^2}4=3$.
\end{proof}


\section{Endpoint regularity}
\label{sectionendpoint}


In this section we prove Theorem \ref{thmendpoint}. The proof of Proposition \ref{propfinalvalue} integrates in time the $H^s$ norm of the forcing, and it stops at $s=\frac32$. At the endpoint, we integrate each Fourier coefficient separately, and we take the supremum over the modes afterwards.

Let us explain the mechanism on the leading forcing of the transverse coordinate, namely the term $-i\mathbf T_{\overline{\tilde A}^2}h_1$ of \eqref{mainforcing}. By \eqref{envelopeh1}, and since $\|\tilde A^2\|_\W\lesssim1$, its $n$-th coefficient is of size at most $\frac{\lambda}{n+1}e^{-\kappa_*\lambda n}$, up to a constant. This coefficient is exponentially small as long as $\lambda(t)\gg\frac1n$, and it decays like $\lambda(t)$ once $\lambda(t)\ll\frac1n$. Hence the mode $n$ is excited essentially during a time interval of length $O(1)$ around the time where $\lambda(t)=\frac1n$, with the amplitude $\frac1{n^2}$. Since the propagator of the transverse coordinate is bounded on each mode $n\geq1$, uniformly in $n$ and in time (Lemma \ref{lemmahardygroup}), each mode receives a contribution of size $n^{-2}$. A Fourier tail of size $n^{-2}$ belongs to $\FL$, and to $\Hlog$ for $\sigma>\frac12$, but not to $H^{\frac32}$, since $\sum_nn^3n^{-4}$ diverges logarithmically.

Two difficulties appear when this heuristic is made rigorous. First, the residual forcing $r_1$ of the negative part contains $\mathbf T_{\Re F}h_1$, whose coefficients are of size $\lambda^2\ell$ up to the frequency $\lambda^{-1}$: they come from the convolution of the peaks with the $\frac1n$ tail of $h_1$. Their time integral is of size $\frac{\log n}{n^2}$, which would produce a logarithmic loss. We remove this term with a second correction $h_2:=\frac12\D^{-1}r_1$, which gains one more power of $n$ (Subsection \ref{subsectionsecond}). Second, the multiplication by a function $\varpi$ concentrated at the scale $\lambda$, in the sense of \eqref{envelope}, does not preserve $\FL$ uniformly in time: $\|\varpi g\|_{\FL}$ is only controlled by $\|g\|_{\FL}+\lambda^{-1}\|g\|_\W$. The loss $\lambda^{-1}$ is compensated by the smallness $\|g\|_\W\lesssim\lambda\ell$ of the Wiener norms of the coordinates, which leaves a factor $\ell$; the resulting remainders are of size $\lambda\ell^2$, which is integrable in time (Subsection \ref{subsectionupper}). The leading Fourier tail of the transverse coordinate is then an explicit integral, which we compute in Subsection \ref{subsectiontail}.


\subsection{Estimates in the Fourier--Lebesgue space}


As in Subsection \ref{subsectionnormal}, a pair $\Phi=(\Phi_1,\Phi_2)$ is measured in $\FL$ by $\|\Phi\|_{\FL}:=\|\Phi_1\|_{\FL}+\|\Phi_2\|_{\FL}$. The space $\FL$ measures the decay $O(\la n\ra^{-2})$ of the Fourier coefficients obtained at the endpoint in Theorem \ref{thmendpoint}, but the multiplication by a function concentrated at the scale $\lambda$ costs a factor $\lambda^{-1}$ in it; the following lemma quantifies this loss and the gain provided by $\D^{-1}$.

\begin{lemma}[Estimates in $\FL$]
\label{lemmaFL}
After increasing $T_0$, the following holds for $t\geq T_0$ and for every $\varpi\in L^2_+$ satisfying \eqref{envelope}.
\begin{enumerate}
\item {\rm (}Products{\rm ).} For $f,g\in\W\cap\FL$, $\|fg\|_{\FL}\leq2\big(\|f\|_\W\|g\|_{\FL}+\|f\|_{\FL}\|g\|_\W\big)$. Moreover $\|\varpi\|_{\FL}\lesssim\lambda^{-1}$, so that $\|\varpi g\|_{\FL}+\|\overline\varpi g\|_{\FL}\lesssim\|g\|_{\FL}+\lambda^{-1}\|g\|_\W$.
\item {\rm (}Smoothing{\rm ).} For $f\in L^2_+\cap\W\cap\FL$,
\begin{gather*}
\|\D^{-1}(\varpi f)\|_{\FL}\lesssim\lambda\|f\|_{\FL}+\|f\|_\W,\qquad \|\D^{-1}\varpi\|_{\FL}\lesssim1,\\
\|\mathbf T_{\overline\varpi}\D^{-1}f\|_{\FL}\lesssim\lambda\ell\|f\|_{\FL},
\end{gather*}
and, if moreover $\hat\varpi(0)=0$, $\|\mathbf T_{D^{-1}(\varpi+\overline\varpi)}f\|_{\FL}\lesssim\lambda\ell\|f\|_{\FL}+\|f\|_\W$.
\item {\rm (}Normal form{\rm ).} For $\Phi\in(L^2_+\cap\W\cap\FL)^2$, the operator $\mathcal T=\mathcal T(t)$ of \eqref{defTt} satisfies
\begin{multline*}
\|\mathcal T\Phi\|_{\FL}+\|(\pa_t\mathcal T)\Phi\|_{\FL}+\|((I+\mathcal T)^{-1}-I)\Phi\|_{\FL}\\
\lesssim\lambda\ell\|\Phi\|_{\FL}+\|\Phi\|_\W.
\end{multline*}
\end{enumerate}
In items {\rm (1)} and {\rm (2)}, the implicit constants are of the form $C_0\,C(\kappa_0)$.
\end{lemma}

The loss $\lambda^{-1}$ in item (1) is the $\FL$ norm of a function concentrated at the scale $\lambda$; it will always be compensated by the smallness of the Wiener norms of the coordinates. In item (2), the operator $\D^{-1}$ gains the factor $\lambda$ on the high frequencies of $f$, as in Lemma \ref{lemmasmoothing}.

\begin{proof}
\noindent{\bf Step 1:} {\it Products.} Since $\la n\ra\leq\la n-m\ra+\la m\ra$, we have $\la n\ra^2\leq2(\la n-m\ra^2+\la m\ra^2)$, hence
\begin{align*}
\la n\ra^2|\widehat{fg}(n)|&\leq2\sum_m\big(\la n-m\ra^2+\la m\ra^2\big)|\hat f(n-m)||\hat g(m)|\\
&\leq2\big(\|f\|_{\FL}\|g\|_\W+\|f\|_\W\|g\|_{\FL}\big).
\end{align*}
By \eqref{envelope}, $\la n\ra^2|\hat\varpi(n)|\leq C_0(1+n^2)\lambda e^{-\kappa_0\lambda n}\lesssim\lambda^{-1}$, and $\|\varpi\|_\W\leq\|\varpi\|_{X^1_t}\lesssim1$ by Lemma \ref{lemmasmoothing}, item (2). Since $\|\overline\varpi\|_\W=\|\varpi\|_\W$ and $\|\overline\varpi\|_{\FL}=\|\varpi\|_{\FL}$, this proves item (1).

\noindent{\bf Step 2:} {\it Smoothing.} For $n\geq0$, since $N(n+1)+K\geq n+1$ and $\la n\ra\leq n+1$,
$$\la n\ra^2|\widehat{\D^{-1}(\varpi f)}(n)|\leq(n+1)\sum_{m=0}^n|\hat\varpi(n-m)||\hat f(m)|.$$
For $m\leq\frac n2$, we have $n+1\leq2(n-m+1)$ and $(j+1)|\hat\varpi(j)|\lesssim\sup_{j\geq0}(j+1)\lambda e^{-\kappa_0\lambda j}\lesssim1$, so these terms contribute $O(\|f\|_\W)$. For $m>\frac n2$, we have $|\hat f(m)|\leq\la m\ra^{-2}\|f\|_{\FL}\leq4\la n\ra^{-2}\|f\|_{\FL}$, and these terms contribute at most $\frac{4(n+1)}{\la n\ra^2}\sum_{j\leq n}|\hat\varpi(j)|\,\|f\|_{\FL}\lesssim\frac{\min(n+1,\lambda^{-1})\lambda}{n+1}\|f\|_{\FL}\leq\lambda\|f\|_{\FL}$. Next, $\widehat{\mathbf T_{\overline\varpi}\D^{-1}f}(n)=\sum_{j\geq0}\overline{\hat\varpi(j)}\,\frac{\hat f(n+j)}{N(n+j+1)+K}$, and $N(n+j+1)+K\geq n+j+1$, hence
\begin{align*}
\la n\ra^2|\widehat{\mathbf T_{\overline\varpi}\D^{-1}f}(n)|&\leq\sum_{j\geq0}|\hat\varpi(j)|\,\frac{\la n\ra^2}{\la n+j\ra^2(n+j+1)}\,\|f\|_{\FL}\\
&\leq\sum_{j\geq0}\frac{|\hat\varpi(j)|}{j+1}\|f\|_{\FL}\leq(N+K)\|\D^{-1}\varpi\|_\W\|f\|_{\FL},
\end{align*}
which is $\lesssim\lambda\ell\|f\|_{\FL}$ by Lemma \ref{lemmasmoothing}, item (2). Also $\la n\ra^2|\widehat{\D^{-1}\varpi}(n)|\leq(n+1)C_0\lambda e^{-\kappa_0\lambda n}\lesssim1$. Finally, if $\hat\varpi(0)=0$, the symbol $D^{-1}(\varpi+\overline\varpi)$ has the Wiener norm $\lesssim\lambda\ell$, by Lemma \ref{lemmasmoothing}, item (2), and the $\FL$ norm $\leq\sup_{n\geq1}\frac{\la n\ra^2}{n}|\hat\varpi(n)|\lesssim1$, so that item (1) gives the last bound of item (2).

\noindent{\bf Step 3:} {\it Normal form.} By \eqref{concentratedpart} and Lemma \ref{lemmapeaks}, item (5), item (2) of the present lemma applies to the entries of $\mathcal T$ and $\pa_t\mathcal T$, with $\varpi$ equal to $\tilde A^2-\tilde c^{\,2}$, $\frac12F$ and their time derivatives. Hence there is a constant $C_4$ such that $\|\mathcal T\Phi\|_{\FL}+\|(\pa_t\mathcal T)\Phi\|_{\FL}\leq C_4(\lambda\ell\|\Phi\|_{\FL}+\|\Phi\|_\W)$, and by Lemma \ref{lemmatransformation}, item (1), we may also assume $\|\mathcal T\Phi\|_\W\leq C_4\lambda\ell\|\Phi\|_\W$. After increasing $T_0$, $C_4\lambda\ell\leq\frac12$. By induction on $m\geq1$,
\begin{gather*}
\|\mathcal T^m\Phi\|_\W\leq(C_4\lambda\ell)^m\|\Phi\|_\W,\\
\|\mathcal T^m\Phi\|_{\FL}\leq(C_4\lambda\ell)^m\|\Phi\|_{\FL}+C_4\,m\,(C_4\lambda\ell)^{m-1}\|\Phi\|_\W,
\end{gather*}
since $\|\mathcal T^{m+1}\Phi\|_{\FL}\leq C_4\lambda\ell\|\mathcal T^m\Phi\|_{\FL}+C_4\|\mathcal T^m\Phi\|_\W$. Summing the Neumann series $(I+\mathcal T)^{-1}-I=\sum_{m\geq1}(-\mathcal T)^m$ gives the bound for $(I+\mathcal T)^{-1}-I$.
\end{proof}

The norm of $\FL$, and the norm of $\Hlog$ for $\sigma\leq\frac12$, satisfy a tame product estimate, and they are propagated by the flow of \eqref{halfwave} on bounded time intervals.

\begin{lemma}[Tame estimate and persistence of the Fourier decay]
\label{lemmapersistence}
Let $\sigma\leq\frac12$.
\begin{enumerate}
\item {\rm (}Tame estimate{\rm ).} For $f,g\in\W\cap\Hlog$,
\begin{equation}
\label{tamelog}
\|fg\|_{\Hlog}\lesssim_\sigma\|f\|_\W\|g\|_{\Hlog}+\|f\|_{\Hlog}\|g\|_\W.
\end{equation}
\item {\rm (}Persistence{\rm ).} Let $u\in\mathcal C(\R,H^1)$ be a solution of \eqref{halfwave}, and let $t_0\in\R$. If $u(t_0)\in\FL$, then $u(t)\in\FL$ for every $t\in\R$, and $\|u(t)\|_{\FL}$ is bounded on bounded time intervals. The same holds with $\FL$ replaced by $\Hlog$.
\end{enumerate}
\end{lemma}

\begin{proof}
For $0\leq\delta\leq1$ and $m\in\Z$, let $w_\delta(m):=\frac{\la m\ra^2}{1+\delta\la m\ra}$ in the case of $\FL$, and $w_\delta(m):=\frac{\la m\ra^{\frac32}\log^{-\sigma}(e+\la m\ra)}{1+\delta\la m\ra}$ in the case of $\Hlog$. Let $\mathcal M_\delta(f):=\sup_mw_\delta(m)|\hat f(m)|$ in the case of $\FL$, and $\mathcal M_\delta(f):=\big(\sum_mw_\delta(m)^2|\hat f(m)|^2\big)^{\frac12}$ in the case of $\Hlog$. Then $\mathcal M_0$ is the norm of $\FL$, and it is equivalent to the norm of $\Hlog$, up to constants depending on $\sigma$, since $\log(e+|m|)\leq\log(e+\la m\ra)\leq2\log(e+|m|)$. The parameter $\delta$ truncates the weight, so that $\mathcal M_\delta(f)$ is finite for every $f\in H^1$.

\noindent{\bf Step 1:} {\it Products.} The weight $w_\delta(m)$ is the value at $x=\la m\ra$ of the function $x\mapsto\frac{x^2}{1+\delta x}$, respectively $x\mapsto\frac{x^{\frac32}\log^{-\sigma}(e+x)}{1+\delta x}$, which is nondecreasing on $[1,+\infty)$: in the case of $\Hlog$, its logarithmic derivative is $\frac{3}{2x}-\frac{\sigma}{(e+x)\log(e+x)}-\frac{\delta}{1+\delta x}\geq\frac3{2x}-\frac1{2x}-\frac1x=0$, since $\sigma\leq\frac12$. Moreover, the value of this function at $2x$ is at most $2^{2+|\sigma|}$ times its value at $x$, with $\sigma=0$ in the case of $\FL$, since $\log(e+2x)\leq2\log(e+x)$. Since $\la n\ra\leq2\max(\la n-m\ra,\la m\ra)$, we obtain $w_\delta(n)\leq2^{2+|\sigma|}\big(w_\delta(n-m)+w_\delta(m)\big)$, uniformly in $\delta$. Hence, by Young's inequality for convolutions,
$$\mathcal M_\delta(fg)\leq2^{2+|\sigma|}\big(\|f\|_\W\mathcal M_\delta(g)+\mathcal M_\delta(f)\|g\|_\W\big).$$
With $\delta=0$, this gives item (1); in the case of $\FL$ and $\delta=0$, it is the product estimate of Lemma \ref{lemmaFL}, item (1), up to the constant.

\noindent{\bf Step 2:} {\it Persistence.} Since $w_\delta(-m)=w_\delta(m)$, we have $\mathcal M_\delta(\bar u)=\mathcal M_\delta(u)$, and Step 1 gives $\mathcal M_\delta(|u|^2u)\leq C\|u\|_\W^2\mathcal M_\delta(u)$, with $C$ independent of $\delta$. The Duhamel formula in Fourier variables,
$$\hat u(t,m)=e^{-i|m|(t-t_0)}\hat u(t_0,m)-i\int_{t_0}^te^{-i|m|(t-t')}\widehat{|u|^2u}(t',m)\,dt',$$
gives $\mathcal M_\delta(u(t))\leq\mathcal M_\delta(u(t_0))+C\big|\int_{t_0}^t\|u(t')\|_\W^2\mathcal M_\delta(u(t'))\,dt'\big|$. For $\delta>0$, since $w_\delta(m)\lesssim_{\sigma}\delta^{-1}\la m\ra$, the quantity $\mathcal M_\delta(u(t'))$ is bounded by $\delta^{-1}\|u(t')\|_{H^1}$ up to a constant, hence finite and locally bounded in $t'$, and $\|u(t')\|_\W\lesssim\|u(t')\|_{H^1}$. By the Gronwall lemma,
$$\mathcal M_\delta(u(t))\leq\mathcal M_\delta(u(t_0))\exp\Big(C\Big|\int_{t_0}^t\|u(t')\|_\W^2\,dt'\Big|\Big),$$
and item (2) follows by letting $\delta\to0$, by monotone convergence, since $w_\delta$ increases to $w_0$.
\end{proof}


\subsection{A second negative-frequency correction}
\label{subsectionsecond}


We let
\begin{equation}
\label{defh2}
h_2:=\frac12\,\D^{-1}r_1.
\end{equation}
Since $2\D h_2=r_1$, the definition \eqref{defrh} gives $r_{h_1+h_2}=r_2$, where
\begin{equation}
\label{defr2}
r_2:=-i\pa_th_2-(M-2K)h_2-4\sqrt{2KQ}\,\mathbf T_{\Re F}h_2:
\end{equation}
with the correction $h:=h_1+h_2$, the forcing $r_1$ of the negative part is removed, and replaced by $r_2$, which contains the derivatives of $h_2$. In this section, $\kappa_1\in(0,\kappa_*]$ denotes a small constant, depending only on the parameters, which may decrease from line to line; in particular, \eqref{envelopeh1} holds with $\kappa_*$ replaced by $\kappa_1$. We call envelope of a function $g$ a majorant of the moduli of its Fourier coefficients, that is, a function $\Psi(n)$ such that $|\hat g(n)|\leq\Psi(n)$ for all $n$.

\begin{lemma}[The second correction]
\label{lemmasecond}
For $t\geq T_0$ and $n\geq0$, the following holds.
\begin{enumerate}
\item {\rm (}The correction $h_2${\rm ).} $|\hat h_2(n)|\lesssim\min(\lambda\ell,\la n\ra^{-2})$, $\|h_2\|_{X^1_t}\lesssim\lambda$, and
$$|\hat h_2(n)|+|\pa_t\hat h_2(n)|\lesssim\Big(\frac{\lambda}{(n+1)^2}+\frac{\lambda^2\ell}{n+1}\Big)e^{-\kappa_1\lambda n}.$$
\item {\rm (}The residual forcing $r_2${\rm ).} $\|r_2\|_{X^1_t}\lesssim\lambda$, and
$$|\hat r_2(n)|\lesssim\Big(\frac{\lambda}{(n+1)^2}+\frac{\lambda^2\ell}{n+1}+\lambda^2\Big)e^{-\kappa_1\lambda n}.$$
\item {\rm (}Coupling terms{\rm ).} The coefficients of the coupling terms of the transverse coordinate satisfy
$$|\widehat{\mathbf T_{\overline{\tilde A}^2}h_1}(n)|\lesssim\frac{\lambda}{n+1}e^{-\kappa_1\lambda n},\qquad |\widehat{\mathbf T_{\overline{\tilde A}^2}h_2}(n)|\lesssim\Big(\frac{\lambda}{(n+1)^2}+\frac{\lambda^2\ell}{n+1}\Big)e^{-\kappa_1\lambda n}.$$
\item {\rm (}Mode integrals{\rm ).} For $n\geq1$,
\begin{equation}
\label{modeintegral}
\int_t^{+\infty}\Big(\big|\widehat{\mathbf T_{\overline{\tilde A}^2}(h_1+h_2)}(\tau,n)\big|+|\hat r_2(\tau,n)|\Big)\,d\tau\lesssim\frac1{n^2}.
\end{equation}
\end{enumerate}
\end{lemma}

Item (4) is the mode-by-mode version of the heuristic described at the beginning of this section: the $n$-th mode of the forcing is of size $\frac\lambda n$ during a time of order one, around the time where $\lambda=\frac1n$.

\begin{proof}
\noindent{\bf Step 1:} {\it The forcing $r_1$.} We have $4\mathbf T_{\Re F}h_1=2\Pi_+(Fh_1)+2\mathbf T_{\overline F}h_1$. By \eqref{envelopeGj} and \eqref{envelopeh1},
$$|\widehat{Fh_1}(n)|\lesssim\sum_{m=0}^n\lambda e^{-\kappa_1\lambda(n-m)}\frac{\lambda\,e^{-\kappa_1\lambda m}}{m+1}\lesssim\lambda^2\log(n+2)\,e^{-\kappa_1\lambda n}\lesssim\lambda^2\ell\,e^{-\frac{\kappa_1}2\lambda n},$$
where we used $\log(n+2)\leq\ell+\lambda n$ and $(1+\lambda n)e^{-\frac{\kappa_1}2\lambda n}\lesssim1$. Similarly,
$$|\widehat{\mathbf T_{\overline F}h_1}(n)|\leq\sum_{j\geq0}|\hat F(j)||\hat h_1(n+j)|\lesssim\lambda^2e^{-\kappa_1\lambda n}\sum_{j\geq0}\frac{e^{-\kappa_1\lambda j}}{j+1}\lesssim\lambda^2\ell\,e^{-\kappa_1\lambda n}.$$
Together with \eqref{envelopeh1} for $h_1$ and $\pa_th_1$, and with \eqref{defr1}, this gives $|\hat r_1(n)|\lesssim\big(\frac{\lambda}{n+1}+\lambda^2\ell\big)e^{-\kappa_1\lambda n}$; the same computation, with $\pa_tF$, $\pa_th_1$ and $\pa_t^2h_1$, gives the same bound for $\pa_t\hat r_1(n)$. Since $\hat h_2(n)=\frac{\hat r_1(n)}{2(N(n+1)+K)}$, this proves the bound on the coefficients of $h_2$ and $\pa_th_2$ in item (1).

\noindent{\bf Step 2:} {\it Toeplitz operators with conjugate symbols.} The envelopes of item (1) and of \eqref{envelopeh1} are nonincreasing functions of $n$. For such an envelope $\Psi$ of a function $g\in L^2_+$, and for $b\in\W\cap L^2_+$, we have $|\widehat{\mathbf T_{\bar b}g}(n)|\leq\sum_{j\geq0}|\hat b(j)|\Psi(n+j)\leq\|b\|_\W\,\Psi(n)$. Applied with $b=\tilde A^2$, which satisfies $\|\tilde A^2\|_\W\lesssim1$ by Lemma \ref{lemmanorms}, item (1), this gives item (3), and applied with $b=F$, it gives the same envelope for $\mathbf T_{\overline F}h_2$ as for $h_2$.

\noindent{\bf Step 3:} {\it The forcing $r_2$.} By \eqref{envelopeGj} and item (1),
\begin{align*}
|\widehat{Fh_2}(n)|&\lesssim\sum_{m=0}^n\lambda e^{-\kappa_1\lambda(n-m)}\Big(\frac{\lambda}{(m+1)^2}+\frac{\lambda^2\ell}{m+1}\Big)e^{-\kappa_1\lambda m}\\
&\lesssim\lambda^2\big(1+\lambda\ell\log(n+2)\big)e^{-\kappa_1\lambda n}\lesssim\lambda^2e^{-\frac{\kappa_1}2\lambda n},
\end{align*}
since $\log(n+2)\leq\ell+\lambda n$ and $\lambda\ell^2\lesssim1$. Together with item (1), Step 2 and \eqref{defr2}, this gives the bound on the coefficients of $r_2$.

\noindent{\bf Step 4:} {\it Norms.} By the bounds on the coefficients, the Wiener norms of $h_2$ and $r_2$ are $\lesssim\lambda+\lambda^2\ell^2+\lambda\lesssim\lambda$, and their $H^1$ norms are $\lesssim\lambda+\lambda^{\frac32}\ell+\lambda^{\frac12}$, since $\sum_n\la n\ra^{-2}<\infty$; multiplying by $\lambda^{\frac12}$ gives the bounds in $X^1_t$. Finally, $\frac{\lambda}{(n+1)^2}e^{-\kappa_1\lambda n}\leq\min(\lambda,\la n\ra^{-2})$ and $\frac{\lambda^2\ell}{n+1}e^{-\kappa_1\lambda n}=\frac{\lambda\ell}{(n+1)^2}\,\lambda(n+1)e^{-\kappa_1\lambda n}\lesssim\min(\lambda\ell,\la n\ra^{-2})$, since $\lambda\ell\leq1$.

\noindent{\bf Step 5:} {\it Mode integrals.} By item (3) and item (2), the integrand of \eqref{modeintegral} is bounded by $\big(\frac{\lambda}{n}+\frac{\lambda^2\ell}{n}+\lambda^2\big)e^{-\kappa_1\lambda n}$, up to a constant, and \eqref{timeintegrals} with $\kappa=\kappa_1$ bounds its integral by $\frac{1}{n^2}+\frac{1+\log n}{n^3}+\frac1{n^2}\lesssim\frac1{n^2}$.
\end{proof}

By Lemma \ref{lemmah1}, item (1), and Lemma \ref{lemmasecond}, item (1), the coefficients of $h_1$, $h_2$ and of their first time derivatives are continuous in time and bounded by $Ce^{-\kappa_1\lambda(t)n}$, so that $h_1+h_2\in\mathcal C^1([T_0,+\infty),H^3_+)$. By Lemma \ref{lemmasecond}, items (1) and (2), by \eqref{boundh1}, and since $r_{h_1+h_2}=r_2$, the function $h=h_1+h_2$ satisfies \eqref{assumptionh} with $s=1$, so that Proposition \ref{propnormalform} and Proposition \ref{propfinalvalue}, item (3), apply. The operators $\mathcal A(t)$ and $\mathcal R(t)$ are the same as for $h=h_1$, and the main part of the forcing $\mathcal E_{h_1+h_2}$ is $-i\big(\Upsilon(t)(h_1+h_2),\,\mathbf T_{\overline{\tilde A}^2}(h_1+h_2),\,r_2\big)$, by \eqref{mainforcing}. In the rest of this section, the constants $R_0$ and $T_{1,R}$ of Proposition \ref{propfinalvalue}, item (3), are those associated with $s=1$ and $h=h_1+h_2$.


\subsection{Upper bounds}
\label{subsectionupper}


From now on, $u$ is the solution of Theorem \ref{thmmain}, $v$ and $\eps:=v-A$ are defined by \eqref{comoving}, and, for $t\geq T_0$, $Z(t):=(a(t),\Phi(t))$ denotes the coordinates of $\eps(t)$ given by Lemma \ref{lemmatransformation}, item (2), with $h=h_1+h_2$. We denote by $\Pi_{\geq1}$ the projection which removes the mode $0$ of both components of a pair, and we use the two envelopes, defined for $n\in\Z$ by
\begin{equation}
\label{defenvelopes}
\Psi_1(t,n):=\min\big(\lambda(t)\ell(t),\la n\ra^{-2}\big),\qquad \Psi_2(t,n):=\lambda(t)^2\ell(t)\,e^{-\kappa_1\lambda(t)|n|}.
\end{equation}
The first one describes the Fourier tail created by the forcing, and the second one describes the concentrated functions multiplied by the small coordinates.

\begin{proposition}[Fourier coefficients of the error]
\label{propFL}
There exists $T_*\geq T_0$ such that the following holds for $t\geq T_*$.
\begin{enumerate}
\item {\rm (}Coordinates{\rm ).} $|a(t)|+\|\Phi(t)\|_\W\lesssim\lambda\ell$ and $\|\Pi_{\geq1}\Phi(t)\|_{\FL}\lesssim1$.
\item {\rm (}Error{\rm ).} $|\hat\eps(t,n)|\lesssim\Psi_1(t,n)+\Psi_2(t,n)$ for $n\in\Z$.
\item {\rm (}Equation of the pair{\rm ).} $\pa_t\Phi=-i\mathcal H_\infty\Phi-i\big(\mathbf T_{\overline{\tilde A}^2}(h_1+h_2),\,r_2\big)+\mathcal S(t)$, where $\|\mathcal S(t)\|_{\FL}\lesssim\lambda(t)\ell(t)^2$.
\end{enumerate}
\end{proposition}

\begin{proof}
Let $h:=h_1+h_2$, let $\mathcal G_h$ be given by \eqref{eqZG}, and denote by $[\cdot]_\Phi$ the $\Phi$-component. In Steps 1 to 3, $Z=(a,\Phi)$ and $Z'=(a',\Phi')$ denote arbitrary elements, and $\eps$, $\eps'$ the corresponding functions given by Lemma \ref{lemmatransformation}, item (2); we come back to the solution in Step 4. For $T\geq T_0$, let $\mathcal X_T$ be the space \eqref{defET} with $s=1$, and let $\mathcal X^\sharp_T$ be the space of $Z\in\mathcal X_T$ such that
$$\|Z\|_{\mathcal X^\sharp_T}:=\|Z\|_{\mathcal X_T}+\sup_{t\geq T}\|\Pi_{\geq1}\Phi(t)\|_{\FL}<+\infty,$$
which is a Banach space.

\noindent{\bf Step 1:} {\it Reconstruction.} Let $R\geq1$, $t\geq T_0$, and let $Z=(a,\Phi)$ satisfy $|a|+\|\Phi\|_{X^1_t}\leq R\lambda\ell$ and $\|\Pi_{\geq1}\Phi\|_{\FL}\leq R$, so that $|\hat\Phi(n)|\lesssim R\,\Psi_1(t,n)$ for every $n$. We claim that the corresponding function $\eps$ satisfies
\begin{equation}
\label{reconstruction}
|\widehat{(\eps-\bar z\bar h)}(n)|\lesssim R\big(\Psi_1(t,n)+\Psi_2(t,n)\big),\qquad |\widehat{\bar z\bar h}(n)|\lesssim\Psi_1(t,n),\qquad n\in\Z.
\end{equation}
The second bound follows from Lemma \ref{lemmah1}, item (1), and Lemma \ref{lemmasecond}, item (1). For the first one, we go through the definitions of Lemma \ref{lemmatransformation}, item (2). We use three facts: $\Psi_1(t,n)$ and $\Psi_2(t,n)$ are nonincreasing functions of $|n|$; $\Psi_1(t,\frac n2)\leq4\Psi_1(t,n)$, while $\Psi_2(t,\frac n2)$ is an envelope of the same form with $\kappa_1$ replaced by $\frac{\kappa_1}2$; and if $|\hat b(m)|\lesssim\lambda e^{-\kappa_1\lambda|m|}$ for $m\in\Z$, $\|g\|_\W\lesssim R\lambda\ell$ and $|\hat g(m)|\lesssim R(\Psi_1(t,m)+\Psi_2(t,m))$, then splitting the convolution $\sum_m\hat b(n-m)\hat g(m)$ according to $|m|\leq\frac{|n|}2$ and $|m|>\frac{|n|}{2}$ shows that $|\widehat{bg}(n)|\lesssim R(\Psi_1(t,n)+\Psi_2(t,n))$. First, $|\widehat{\iota_t(a)}(n)|\lesssim|a|\big({\bf 1}_{n=0}+\lambda e^{-\kappa_1\lambda n}\big)\lesssim R(\Psi_1(t,n)+\Psi_2(t,n))$, since $\Psi_1(t,0)=\lambda\ell$, and $|\widehat{\D^{-1}\Sigma(t)a}(n)|\lesssim R\Psi_2(t,n)$ by \eqref{structureSigma}. Next, $[\mathcal T\Phi]_1=-\frac12\mathbf T_{\overline{\tilde A^2-\tilde c^2}}\D^{-1}\Phi_2$ has coefficients $\lesssim\frac{1}{n+1}\sup_{m\geq n}|\hat\Phi_2(m)|$, and the two terms of $[\mathcal T\Phi]_2=-\frac12\D^{-1}\big((\tilde A^2-\tilde c^{\,2})\Phi_1\big)-\frac{2\sqrt{2KQ}}{N}\mathbf T_{D^{-1}\Re F}\Phi_2$ are convolutions of the above type, since the symbols $\tilde A^2-\tilde c^{\,2}$ and $D^{-1}\Re F$ have coefficients $\lesssim\lambda e^{-\kappa_1\lambda|m|}$. Hence $f$ and $k-h=\check k+\frac12\D^{-1}\Sigma(t)a$ have coefficients $\lesssim R(\Psi_1+\Psi_2)$. Finally, $z\Theta^2=\bar p^{\,2}z-2\bar pzF+zF^2$, so that $\widehat{z\Theta^2f}(n)$ is $\bar p^{\,2}\hat f(n-1)$ plus a convolution of the above type, which proves \eqref{reconstruction}. Together with Lemma \ref{lemmatransformation}, item (2), \eqref{boundh1}, and Lemma \ref{lemmasecond}, item (1), we obtain
\begin{equation}
\label{FLeps}
\begin{gathered}
\|\eps-\bar z\bar h\|_\W\lesssim R\lambda\ell,\qquad \|\eps-\bar z\bar h\|_{\FL}\lesssim R\,\ell,\\
\|\bar z\bar h\|_\W\lesssim\lambda\ell,\qquad \|\bar z\bar h\|_{\FL}\lesssim1,
\end{gathered}
\end{equation}
where we used $\sup_n\la n\ra^2\Psi_2(t,n)\lesssim\ell$.

\noindent{\bf Step 2:} {\it Remainders in $\FL$.} For $Z$ as in Step 1, let $\mathcal S_Z(t):=[\mathcal G_h(t,Z)]_\Phi+i\big(\mathbf T_{\overline{\tilde A}^2}h,r_2\big)=[\mathcal R(t)Z+\mathcal E_h^{\rm rem}(t)+\mathcal N_h(t,Z)]_\Phi$, by \eqref{mainforcing}. We claim that, if $R\lambda\ell\leq1$,
\begin{equation}
\label{Rhigh}
\big\|[\mathcal R(t)Z]_\Phi\big\|_{\FL}\lesssim\lambda\ell\|\Phi\|_{\FL}+\ell\big(|a|+\|\Phi\|_\W\big),\quad \|\mathcal S_Z(t)\|_{\FL}\lesssim R^2\lambda\ell^2.
\end{equation}
We use the description of the right hand side of \eqref{eqZ} given by Proposition \ref{propnormalform}, item (1), and Lemma \ref{lemmaFL}. By \eqref{structurephi}, $[\mathcal R(t)Z]_\Phi$ consists of the following terms. The term $-i(\mathcal V(t)-\mathcal V_\infty)\Phi$ is $O(\lambda)\|\Phi\|_{\FL}$. In the term $-i(I+\mathcal T)^{-1}\big([\mathcal V,\mathcal T]+\mathcal B\mathcal T-i\pa_t\mathcal T\big)\Phi$, Lemma \ref{lemmaFL}, items (1) and (3), give $\|\mathcal B\mathcal T\Phi\|_{\FL}\lesssim\|\mathcal T\Phi\|_{\FL}+\lambda^{-1}\|\mathcal T\Phi\|_\W\lesssim\lambda\ell\|\Phi\|_{\FL}+\ell\|\Phi\|_\W$, and the other terms are smaller. The remaining terms are $-i(I+\mathcal T)^{-1}$ applied to the terms of $\mathcal Y$ which are linear in $Z$ and contain neither $h$ nor $\mathcal N$: they are obtained by applying the operators $\mathbf T_{\overline{\tilde A}^2}$, $\mathbf T_{\Re F}$, or a constant, to the functions $\D^{-1}\Sigma(t)a$, $\D^{-1}(\pa_t[\Sigma(t)]a)$ and $\D^{-1}\Sigma(t)\big(-i\Omega(t)a-i\Upsilon(t)(k-h)\big)$. By Proposition \ref{proplinearized}, items (1), (2) and (3), and Lemma \ref{lemmaFL}, item (2), these functions have $\FL$ norms $\lesssim|a|+\|k-h\|_\W$, and by Lemma \ref{lemmasmoothing}, item (2), they have Wiener norms $\lesssim\lambda\ell(|a|+\|k-h\|_\W)$; by Lemma \ref{lemmaFL}, items (1) and (3), the resulting terms are $O(\ell)(|a|+\|\Phi\|_\W)$ in $\FL$, since $\|k-h\|_\W\lesssim\|\Phi\|_\W+\lambda\ell|a|$. This proves the first bound of \eqref{Rhigh}, and with $\|\Phi\|_{\FL}\leq R+R\lambda\ell$, it gives $\|[\mathcal R(t)Z]_\Phi\|_{\FL}\lesssim R\lambda\ell^2$.

By Proposition \ref{propnormalform}, item (1), a function $\mathcal N$ contributes to the equation of $\Phi$ the term \eqref{contributionN}, whose $\FL$ norm is $\lesssim\|\mathcal N\|_{\FL}+\lambda^{-1}\|\mathcal N\|_\W$, since $\|\overline{z\Theta^2}\|_\W\lesssim1$, $\|\overline{z\Theta^2}\|_{\FL}\lesssim\lambda^{-1}$, and $|\iota_t^{-1}P_t\Pi_+\mathcal N|\lesssim\|\mathcal N\|_\W$ by Lemma \ref{lemmacoordinates}, item (3). By \eqref{defNA}, the function $\mathcal N_A(\eps)-\mathcal N_A(\bar z\bar h)$ is a sum of products of three factors, each product containing a factor $\eps-\bar z\bar h$ or its conjugate, at most one factor $A$ or $\overline A$, and otherwise factors among $\eps-\bar z\bar h$, $\bar z\bar h$ and their conjugates. By \eqref{FLeps}, by the bounds $\|A\|_\W\lesssim1$ of Lemma \ref{lemmanorms}, item (1), and $\|A\|_{\FL}\lesssim\lambda^{-1}$ of Lemma \ref{lemmaFL}, item (1), which apply since $A=c+\sqrt QF$, and by Lemma \ref{lemmaFL}, item (1), the $\FL$ norm of such a product is $\lesssim\lambda^{-1}(R\lambda\ell)^2+R\ell\cdot R\lambda\ell\lesssim R^2\lambda\ell^2$, while its Wiener norm is $\lesssim R^2(\lambda\ell)^2$. The same bounds hold for $\mathcal N_A(\bar z\bar h)$, with $R=1$. Finally, by Proposition \ref{propnormalform}, item (3), the other terms of $[\mathcal E_h^{\rm rem}]_\Phi$ are $-i\big((I+\mathcal T)^{-1}-I\big)(\mathbf T_{\overline{\tilde A}^2}h,r_2)$ and $\frac12(I+\mathcal T)^{-1}\big(0,\D^{-1}\Sigma(t)(i\Upsilon(t)h)\big)$, which are $O(\lambda\ell)$ in $\FL$, by Lemma \ref{lemmaFL}, \eqref{boundh1}, Lemma \ref{lemmasecond}, items (1)--(3), and $|\Upsilon(t)h|\lesssim\|h\|_\W\lesssim\lambda\ell$. This proves \eqref{Rhigh}. Since $\mathcal R(t)$ is linear and the nonlinear terms are sums of products, the same argument gives, for $Z,Z'$ as in Step 1,
\begin{equation}
\label{Rhighdiff}
\|\mathcal S_Z(t)-\mathcal S_{Z'}(t)\|_{\FL}\lesssim R\lambda\ell^2\|\Pi_{\geq1}(\Phi-\Phi')\|_{\FL}+R\ell\big(|a-a'|+\|\Phi-\Phi'\|_{X^1_t}\big).
\end{equation}

\noindent{\bf Step 3:} {\it Fixed point with Fourier control.} Let $\mathcal J_h$ be the map of Proposition \ref{propfinalvalue}, item (3), with $s=1$ and $h=h_1+h_2$. By this item, for $R\geq R_0$ and $T\geq T_{1,R}$, $\mathcal J_h$ maps the closed ball of radius $R$ of $\mathcal X_T$ into the closed ball of radius $R_0$, and it is Lipschitz on this ball, with constant $\frac12$. For the Fourier part of the norm, let $n\geq1$ and $t\geq T$. By \eqref{defpropagator} and \eqref{defJ},
$$\widehat{[\mathcal J_h(Z)]_\Phi}(t,n)=-\int_t^{+\infty}e^{i(\tau-t)\mathcal H_\infty(n)}\,\widehat{[\mathcal G_h(\tau,Z(\tau))]_\Phi}(n)\,d\tau,$$
and $[\mathcal G_h(\tau,Z(\tau))]_\Phi=-i\big(\mathbf T_{\overline{\tilde A}^2}h,r_2\big)(\tau)+\mathcal S_Z(\tau)$. By \eqref{boundgroup}, \eqref{modeintegral} and \eqref{Rhigh},
$$\la n\ra^2\big|\widehat{[\mathcal J_h(Z)]_\Phi}(t,n)\big|\lesssim1+\int_t^{+\infty}\|\mathcal S_Z(\tau)\|_{\FL}\,d\tau\lesssim1+R^2\lambda(t)\ell(t)^2,$$
since $\int_t^{+\infty}\lambda\ell^2\,d\tau\lesssim\lambda(t)\ell(t)^2$ by \eqref{ratio}. Similarly, \eqref{Rhighdiff} gives
$$\sup_{t\geq T}\big\|\Pi_{\geq1}[\mathcal J_h(Z)-\mathcal J_h(Z')]_\Phi(t)\big\|_{\FL}\lesssim R\lambda(T)\ell(T)^2\|Z-Z'\|_{\mathcal X^\sharp_T}.$$
By the first bound and \eqref{ratio}, there is a constant $C_5$ such that
$$\sup_{t\geq T}\|\Pi_{\geq1}[\mathcal J_h(Z)]_\Phi(t)\|_{\FL}\leq C_5\big(1+R^2\lambda(T)\ell(T)^2\big),$$
and we let $R_\sharp:=R_0+2C_5$. For $T\geq T_{1,R_\sharp}$ large enough, $\mathcal J_h$ maps the closed ball of radius $R_\sharp$ of $\mathcal X^\sharp_T$ into itself, since $\|\mathcal J_h(Z)\|_{\mathcal X_T}\leq R_0$ and $C_5(1+R_\sharp^2\lambda(T)\ell(T)^2)\leq2C_5$, and it is a contraction there. We fix such a time $T_\sharp$, and we denote by $Z^\sharp$ the fixed point of $\mathcal J_h$ in this ball.

\noindent{\bf Step 4:} {\it Identification.} By Proposition \ref{propfinalvalue}, item (3), the fixed point $Z^\sharp$ is the coordinate vector of a solution $\eps^\sharp\in\mathcal C^1([T_\sharp,+\infty),L^2)\cap\mathcal C([T_\sharp,+\infty),H^1)$ of \eqref{eqeps}, with $\|\eps^\sharp(t)\|_{X^1_t}\lesssim R_\sharp\lambda(t)\ell(t)$, by Lemma \ref{lemmatransformation}, item (2), \eqref{boundh1} and Lemma \ref{lemmasecond}, item (1). By \eqref{approxmain} with $s=1$, Lemma \ref{lemmacomoving} and \eqref{comparable}, $\|\eps(t)\|_{X^1_t}\lesssim\lambda(t)\ell(t)$ for $t\geq0$, and Proposition \ref{propfinalvalue}, item (2), gives $\eps^\sharp=\eps$ on $[T_*,+\infty)$ for some $T_*\geq T_\sharp$. Therefore the coordinates $(a(t),\Phi(t))$ of $\eps(t)$ equal $Z^\sharp(t)$ for $t\geq T_*$. Item (1) follows, item (2) follows from \eqref{reconstruction}, and item (3) follows from \eqref{eqZ} and \eqref{Rhigh}.
\end{proof}

We can now prove the upper bounds of Theorem \ref{thmendpoint}: the envelopes of Proposition \ref{propFL}, item (2), are summed with the weights of $\FL$ and of $\Hlog$.

\begin{proof}[Proof of Theorem \ref{thmendpoint}, item (1)]
By Lemma \ref{lemmacomoving}, the coefficients of $u(t)-U(t)$ are the coefficients $\hat\eps(t,n)$, placed at the frequencies $Nn-K$, and the weights $\la Nn-K\ra$ and $\la n\ra$ are comparable by \eqref{comparable}. Let $t\geq T_*$. By Proposition \ref{propFL}, item (2), and \eqref{defenvelopes}, $\|\eps(t)\|_{\FL}\lesssim\sup_n\la n\ra^2(\Psi_1(t,n)+\Psi_2(t,n))\lesssim\ell$. For $\sigma>\frac12$, we split the sum defining $\|\Psi_1(t,\cdot)\|_{\Hlog}^2$ at $n_0:=(\lambda\ell)^{-\frac12}$, where $\log(e+n_0)\gtrsim\ell$:
\begin{gather*}
\sum_{|n|\leq n_0}\frac{\la n\ra^3(\lambda\ell)^2}{\log^{2\sigma}(e+|n|)}\lesssim\frac{(\lambda\ell)^2n_0^4}{\log^{2\sigma}(e+n_0)}\lesssim\ell^{-2\sigma},\\
\sum_{|n|>n_0}\frac{\la n\ra^{-1}}{\log^{2\sigma}(e+|n|)}\lesssim_\sigma\log^{1-2\sigma}(e+n_0)\lesssim\ell^{1-2\sigma},
\end{gather*}
while $\sum_n\frac{\la n\ra^3}{\log^{2\sigma}(e+|n|)}\Psi_2(t,n)^2\lesssim\lambda^4\ell^2\cdot\lambda^{-4}\ell^{-2\sigma}=\ell^{2-2\sigma}$, since the sum is dominated by the frequencies $|n|\sim\lambda^{-1}$, where $\log(e+|n|)\sim\ell$. Hence $\|\eps(t)\|_{\Hlog}\lesssim_\sigma\ell^{1-\sigma}$. This proves \eqref{endpointupper} for $t\geq T_*$. In particular $u(T_*)\in\FL$, and by Lemma \ref{lemmapersistence}, item (2), $u(t)\in\FL$ for every $t\in\R$, with $\|u(t)\|_{\FL}$ bounded on $[0,T_*]$; since $\FL\subset\Hlog$ for $\sigma>\frac12$, since $U(t)$ is bounded in these spaces for $0\leq t\leq T_*$, and since $1\leq\ell(t)\leq\ell(T_*)$ for $0\leq t\leq T_*$, \eqref{endpointupper} also holds for $0\leq t\leq T_*$. The inclusion $\FL\subset\mathrm B^{\frac32}_{2,\infty}$ follows from $\sum_{2^m\leq\la n\ra<2^{m+1}}\la n\ra^3|\hat f(n)|^2\leq\sum_{2^m\leq\la n\ra<2^{m+1}}\la n\ra^{-1}\|f\|_{\FL}^2\leq4\|f\|_{\FL}^2$.

Finally, by Lemma \ref{lemmacomoving} and Lemma \ref{lemmanorms}, item (2), with $s=\frac32$, and since the mode $-K$ of $U(t)$ is bounded,
$$\|U(t)\|_{\Hlog}^2=N^3Q\,\Gamma(4)\sum_{j=1}^J|B'(\zeta_j)|^2\,\lambda(t)^{-2}\ell(t)^{-2\sigma}(1+o(1)),\qquad \Gamma(4)=6.$$
Since the error $\|u(t)-U(t)\|_{\Hlog}\lesssim\ell^{1-\sigma}$ is negligible with respect to $\lambda^{-1}\ell^{-\sigma}$, this gives the asymptotics of $\|u(t)\|_{\Hlog}$.
\end{proof}


\subsection{The leading Fourier tail}
\label{subsectiontail}


We now compute the leading part of the coefficients of the first component $\Phi_1$ of the pair $\Phi$ at a fixed time $T$, as $n\to+\infty$. By \eqref{defTt} and Proposition \ref{propFL}, item (1), $\Phi_1$ differs from the transverse coordinate $f$ by $[\mathcal T\Phi]_1$, whose coefficients are $O(n^{-3})$. By Proposition \ref{propFL}, item (3), and Lemma \ref{lemmahardygroup}, item (3), the leading part is
$$i\int_T^{+\infty}e^{i(M-2K)(\tau-T)}\,\widehat{\mathbf T_{\overline{\tilde A}^2}h_1}(\tau,n)\,d\tau.$$
The phase $e^{i(M-2K)\tau}$ is the rotation of the transverse coordinate. The integral is computed in two rescaled variables: the variable $\rho:=n\lambda_j(\tau)$, which measures the time $\tau$ through the ratio of the frequency $n$ to the frequency scale $\lambda_j(\tau)^{-1}$ of the $j$-th peak, and the variable $\varsigma:=n(1-x)$, where $x$ is the variable of the integral representation $\frac{1}{N(m+1)+K}=\frac1N\int_0^1x^{m+\frac KN}\,dx$. In these variables, the phase $e^{i(M-2K)\tau}$ becomes the factor $\rho^{-i\omega}$, so that the time integral becomes a Mellin transform, that is, an integral of the form $\int_0^{+\infty}\rho^{-i\omega}g(\rho)\,d\rho$, with the exponent
\begin{equation}
\label{defomega}
\omega:=\frac{M-2K}{\Lambda},
\end{equation}
the ratio of the rotation frequency of the transverse coordinate to the rate of concentration. With $\chi$ given by \eqref{lambdaj}, so that $1-\chi=\frac{ie^{-i\vartheta}}{2\sin\vartheta}$ and $\Re(1-\chi)=\frac12$, we define
\begin{equation}
\label{defxi}
\Xi:=-\frac{iK\sqrt Q}{N\Lambda}\,\Gamma(2-i\omega)\int_0^1\frac{y^{-i\omega}\big(\overline{\tilde c_\infty}+\sqrt{2K}\,e^{-i\vartheta}y\big)^2}{\big(1-(1-\chi)y\big)^{2-i\omega}}\,dy,
\end{equation}
where the complex powers are defined with the principal branch of the logarithm; note that $\Re\big(1-(1-\chi)y\big)=1-\frac y2>0$ for $0\leq y\leq1$, and $|y^{-i\omega}|=1$.

\begin{proposition}[The leading Fourier tail]
\label{proptail}
For $n\geq1$ and $T\in\R$, let
\begin{equation}
\label{defSn}
S_n:=e^{i\vartheta}\sum_{j=1}^J|B'(\zeta_j)|^{-i\omega}\,\overline{\zeta_j}^{\,n+1},\quad \Xi_n(T):=e^{-i(M-2K)T}\big(4\sin^2\vartheta\big)^{i\omega}\,\Xi\,n^{i\omega}S_n.
\end{equation}
There is a constant $C_6$ such that, for every $T\geq T_*$,
\begin{equation}
\label{tail}
\limsup_{n\to+\infty}\,\big|n^2\hat\Phi_1(T,n)-\Xi_n(T)\big|\leq C_6\,\lambda(T)\ell(T)^2.
\end{equation}
\end{proposition}

\begin{proof}
\noindent{\bf Step 1:} {\it Reduction to the leading forcing.} By Proposition \ref{propFL}, item (3), and since $\hat\Phi(\tau,n)\to0$ as $\tau\to+\infty$ while $e^{i(\tau-T)\mathcal H_\infty(n)}$ is bounded for $n\geq1$, the Duhamel formula from $\tau=+\infty$ gives, for $n\geq1$,
\begin{align*}
\hat\Phi(T,n)={}&i\int_T^{+\infty}e^{i(\tau-T)\mathcal H_\infty(n)}\Big(\widehat{\mathbf T_{\overline{\tilde A}^2}(h_1+h_2)}(\tau,n),\,\hat r_2(\tau,n)\Big)\,d\tau\\
&-\int_T^{+\infty}e^{i(\tau-T)\mathcal H_\infty(n)}\,\hat{\mathcal S}(\tau,n)\,d\tau.
\end{align*}
By \eqref{boundgroup} and Proposition \ref{propFL}, item (3), the last integral is bounded by $C\,n^{-2}\int_T^{+\infty}\lambda\ell^2\,d\tau\leq C_6\,n^{-2}\lambda(T)\ell(T)^2$. In the first component of the first integral, we keep the leading term $g_1:=\mathbf T_{\overline{\tilde A}^2}h_1$ propagated by the entry $(1,1)$, and we estimate the other terms with Lemma \ref{lemmahardygroup}, item (3), Lemma \ref{lemmasecond}, items (2) and (3), and \eqref{timeintegrals}: the contribution of $\mathbf T_{\overline{\tilde A}^2}h_2$ is $O(\frac{1+\log n}{n^3})$, the contribution of $r_2$ through the entry $(1,2)$ is $O(\frac1n\cdot\frac1{n^2})$, and replacing the entry $(1,1)$ by $e^{i(M-2K)(\tau-T)}$ in the leading term costs
\begin{align*}
\int_T^{+\infty}\Big(\frac{\tau-T}{n}+\frac1{n^2}\Big)|\hat g_1(\tau,n)|\,d\tau&\lesssim\frac1{n^2}\int_T^{+\infty}(\tau-T)\lambda e^{-\kappa_1\lambda n}\,d\tau+\frac1{n^4}\\
&\lesssim\frac{1+\log n}{n^3}.
\end{align*}
Therefore
\begin{equation}
\label{tailreduction}
\limsup_{n\to+\infty}\,n^2\Big|\hat\Phi_1(T,n)-i\int_T^{+\infty}e^{i(M-2K)(\tau-T)}\hat g_1(\tau,n)\,d\tau\Big|\leq C_6\,\lambda(T)\ell(T)^2.
\end{equation}

\noindent{\bf Step 2:} {\it An integral formula.} For $m\geq0$, Lemma \ref{lemmah1}, item (1), gives $\hat h_1(m)=-K\sqrt Q\,\frac{\hat F(m+1)}{N(m+1)+K}$, and \eqref{partialfractions} gives $\hat F(m+1)=\sum_jb_j\lambda_jp_j^m$. Writing $\frac{1}{N(m+1)+K}=\frac1N\int_0^1x^{m+\frac KN}dx$, we obtain, for $m\geq0$,
$$\hat h_1(m)=-\frac{K\sqrt Q}{N}\sum_{j=1}^Jb_j\lambda_jp_j^m\int_0^1x^{m+\frac KN}\,dx.$$
We have $\hat g_1(n)=\sum_{l\geq0}\overline{\widehat{\tilde A^2}(l)}\,\hat h_1(n+l)$, and $\sum_{l\geq0}\overline{\widehat{\tilde A^2}(l)}\,w^l=\overline{\tilde A(\bar w)}^{\,2}$ for $|w|\leq1$. With $w=p_jx$, we obtain
$$\hat g_1(\tau,n)=-\frac{K\sqrt Q}{N}\sum_{j=1}^Jb_j\lambda_jp_j^n\int_0^1x^{n+\frac KN}\,\overline{\tilde A(\overline{p_j}x)}^{\,2}\,dx,$$
where all the functions of time are evaluated at $\tau$. By \eqref{partialfractions},
$$\overline{\tilde A(\overline{p_j}x)}=\overline{\tilde c}+\sqrt{2K}\sum_{l=1}^J\overline{b_l}\,\frac{\lambda_lp_jx}{1-\overline{p_l}p_jx}=\overline{\tilde c}+\sqrt{2K}\,\overline{b_j}p_j\,\frac{\lambda_jx}{(1-x)+\lambda_jx}+O(\lambda),$$
uniformly in $x\in[0,1]$, since $|1-\overline{p_l}p_jx|\gtrsim1$ for $l\neq j$.

\noindent{\bf Step 3:} {\it The scaling limit.} We fix $j$, and we denote by $\hat g_{1,j}$ the $j$-th term of the sum which gives $\hat g_1$ in Step 2. For $\tau\geq T_0$, $\lambda_j$ is decreasing, with $-\frac{\lambda_j'}{\lambda_j}=\Lambda+O(\lambda)$ by Lemma \ref{lemmapeaks}, item (1), and by \eqref{lambdaj} and \eqref{defplambda},
$$\lambda_j(\tau)=\frac{4\sin^2\vartheta}{|B'(\zeta_j)|}\,e^{-\Lambda\tau}\big(1+o(1)\big),\qquad\tau\to+\infty.$$
We change variables $\tau\mapsto\rho:=n\lambda_j(\tau)$ and $x\mapsto\varsigma:=n(1-x)$, so that $\lambda_j\,d\tau\,dx=\frac{d\rho\,d\varsigma}{n^2(\Lambda+o(1))}$; the domain $\{\tau\geq T,\ 0\leq x\leq1\}$ becomes $\{0<\rho\leq n\lambda_j(T),\ 0\leq\varsigma\leq n\}$, which increases to $(0,+\infty)\times[0,+\infty)$. For fixed $(\rho,\varsigma)\in(0,+\infty)^2$, as $n\to+\infty$, we have $\tau\to+\infty$, and, by Lemma \ref{lemmapeaks},
\begin{gather*}
\overline{\zeta_j}^{\,-n}p_j^n=\Big(1-\chi\frac\rho n+O\big(\tfrac{\rho\lambda}{n}\big)\Big)^n\to e^{-\chi\rho},\qquad x^{n+\frac KN}\to e^{-\varsigma},\\
\frac{\lambda_jx}{(1-x)+\lambda_jx}\to\frac{\rho}{\rho+\varsigma},\qquad b_j\to e^{i\vartheta}\overline{\zeta_j},\qquad \overline{b_j}p_j\to e^{-i\vartheta},\qquad \tilde c\to\tilde c_\infty,\\
n^{-i\omega}e^{i(M-2K)\tau}=\big(ne^{-\Lambda\tau}\big)^{-i\omega}\to\Big(\frac{4\sin^2\vartheta}{|B'(\zeta_j)|}\Big)^{i\omega}\rho^{-i\omega},
\end{gather*}
where we used $M-2K=\omega\Lambda$.
Since $\overline{\tilde c_\infty}+\sqrt{2K}e^{-i\vartheta}=\sqrt Q$, the limit of $\overline{\tilde A(\overline{p_j}x)}$ is $\frac{\overline{\tilde c_\infty}\,\varsigma+\sqrt Q\,\rho}{\rho+\varsigma}$. The integrand is dominated by $Ce^{-\frac\rho2-\varsigma}$, since $|p_j|^n=(1-\lambda_j)^{\frac n2}\leq e^{-\frac\rho2}$, $x^{n+\frac KN}\leq e^{-\varsigma}$, and $|\tilde A|\leq\sqrt Q+\sqrt{2K}$ on $\overline\Disc$. By dominated convergence, as $n\to+\infty$,
\begin{align*}
&n^{2-i\omega}\,\overline{\zeta_j}^{\,-n}\,i\int_T^{+\infty}e^{i(M-2K)\tau}\hat g_{1,j}(\tau,n)\,d\tau\\
&\qquad\longrightarrow-\frac{iK\sqrt Q}{N\Lambda}\,e^{i\vartheta}\,\overline{\zeta_j}\Big(\frac{4\sin^2\vartheta}{|B'(\zeta_j)|}\Big)^{i\omega}\\
&\qquad\qquad\times\int_0^{+\infty}\!\!\int_0^{+\infty}e^{-\chi\rho-\varsigma}\rho^{-i\omega}\Big(\frac{\overline{\tilde c_\infty}\,\varsigma+\sqrt Q\,\rho}{\rho+\varsigma}\Big)^2d\rho\,d\varsigma.
\end{align*}
In the double integral, we let $\rho=\xi y$ and $\varsigma=\xi(1-y)$, with $\xi>0$ and $0<y<1$, so that $d\rho\,d\varsigma=\xi\,d\xi\,dy$, and we use the formula $\int_0^{+\infty}\xi^{1-i\omega}e^{-\xi(1-(1-\chi)y)}\,d\xi=\Gamma(2-i\omega)\big(1-(1-\chi)y\big)^{i\omega-2}$, which holds since $\Re\big(1-(1-\chi)y\big)=1-\frac y2>0$, by analytic continuation in $q$ of the identity $\int_0^{+\infty}\xi^{1-i\omega}e^{-\xi q}d\xi=\Gamma(2-i\omega)q^{i\omega-2}$, $q>0$. Since $\overline{\tilde c_\infty}(1-y)+\sqrt Qy=\overline{\tilde c_\infty}+\sqrt{2K}e^{-i\vartheta}y$, the double integral equals $-\frac{N\Lambda}{iK\sqrt Q}\Xi$, by \eqref{defxi}. Multiplying by $e^{-i(M-2K)T}$ and summing over $j$, we obtain, by \eqref{defSn},
$$i\int_T^{+\infty}e^{i(M-2K)(\tau-T)}\hat g_1(\tau,n)\,d\tau=n^{-2}\,\Xi_n(T)+o(n^{-2}),\qquad n\to+\infty,$$
which, together with \eqref{tailreduction}, proves \eqref{tail}.
\end{proof}

The coefficient $\Xi$ is not zero for an open dense set of parameters, and in particular on the slice $M=2K$.

\begin{lemma}[The coefficient $\Xi$]
\label{lemmaxi}
Fix $1\leq K\leq N$, and let $\mathcal P_K$ be the set of the pairs $(M,Q)\in(0,+\infty)^2$ satisfying \eqref{admissible}.
\begin{enumerate}
\item {\rm (}The slice $M=2K${\rm ).} If $M=2K$, then $\Re\,\Xi>0$; in particular $\Xi\neq0$.
\item {\rm (}Genericity{\rm ).} The coefficient $\Xi$ is a real-analytic function on the connected open set $\mathcal P_K$. Consequently, the set of $(M,Q)\in\mathcal P_K$ with $\Xi\neq0$ is open and dense in $\mathcal P_K$, and its complement in $\mathcal P_K$ has Lebesgue measure zero.
\end{enumerate}
\end{lemma}

\begin{proof}
\noindent{\bf Step 1:} {\it The slice $M=2K$.} Let $M=2K$, so that \eqref{admissible} reads $0<Q<8K$. Then $\omega=0$ and $\Gamma(2)=1$, and, by \eqref{deftheta}, $\sqrt{2K/Q}\,\cos\vartheta=\frac12$, so that $0<\vartheta<\frac\pi2$ and $\sqrt{2K/Q}\,e^{-i\vartheta}=\frac12(1-i\tan\vartheta)$. Hence $\overline{\tilde c_\infty}/\sqrt Q=1-\sqrt{2K/Q}\,e^{-i\vartheta}=\frac12(1+i\tan\vartheta)$, and $1-\chi=\frac{ie^{-i\vartheta}}{2\sin\vartheta}=\frac12(1+i\cot\vartheta)$. Therefore
$$\frac{iN\Lambda}{KQ\sqrt Q}\,\Xi=\int_0^1g(y)^2\,dy,\qquad g(y):=\frac{(1+y)+i(1-y)\tan\vartheta}{(2-y)-iy\cot\vartheta}.$$
Multiplying the numerator and the denominator of $g(y)$ by $(2-y)+iy\cot\vartheta$, we find, for $0\leq y\leq1$,
\begin{gather*}
\Re g(y)=\frac{(1+y)(2-y)-y(1-y)}{(2-y)^2+y^2\cot^2\vartheta}=\frac{2}{(2-y)^2+y^2\cot^2\vartheta}>0,\\
\Im g(y)=\frac{y(1+y)\cot\vartheta+(1-y)(2-y)\tan\vartheta}{(2-y)^2+y^2\cot^2\vartheta}>0.
\end{gather*}
Hence $\Im\big(g(y)^2\big)=2\Re g(y)\Im g(y)>0$ for $0\leq y\leq1$, so that
$$\frac{N\Lambda}{KQ\sqrt Q}\,\Re\,\Xi=\Im\Big(\frac{iN\Lambda}{KQ\sqrt Q}\,\Xi\Big)=\int_0^1\Im\big(g(y)^2\big)\,dy>0,$$
which proves item (1).

\noindent{\bf Step 2:} {\it Genericity.} For each $Q>0$, the set of $M$ with $(M,Q)\in\mathcal P_K$ is the interval $\big((\sqrt{2K}-\sqrt Q)^2,(\sqrt{2K}+\sqrt Q)^2\big)$, so that $\mathcal P_K$ is connected. On $\mathcal P_K$, $\cos\vartheta=\frac{2K+Q-M}{2\sqrt{2KQ}}\in(-1,1)$, hence $\vartheta$, $\Lambda=2\sqrt{2KQ}\sin\vartheta>0$, $\omega$, $\tilde c_\infty$ and $\chi$ are real-analytic functions of $(M,Q)$. When $(M,Q)$ is replaced by complex values in a small neighborhood of a point of $\mathcal P_K$, the integrand of \eqref{defxi} remains holomorphic in these values for each $y\in(0,1]$, and it is dominated by $Cy^{-\frac12}$, since $|y^{-i\omega}|=y^{\Im\omega}$ and $\Re(1-(1-\chi)y)\geq\frac14$. By dominated convergence and Morera's theorem, $\Xi$ extends holomorphically to this neighborhood, so that $\Xi$ is real-analytic on $\mathcal P_K$. The zero set of $\Xi$ is closed, by continuity. By item (1), the real-analytic function $\Re\,\Xi$ does not vanish identically on the connected set $\mathcal P_K$, so that its zero set, which contains the zero set of $\Xi$, has empty interior by the identity theorem, and Lebesgue measure zero \cite{Mi2020}.
\end{proof}

\begin{remark}
\label{remarktail}
By Proposition \ref{proptail}, for $T\geq T_*$, $n^2\hat\Phi_1(T,n)=\Xi_n(T)+O(\lambda(T)\ell(T)^2)+o(1)$ as $n\to+\infty$, and the same holds for the coefficients of the transverse coordinate $f$. When $\Xi\neq0$, each peak contributes to $\Xi_n(T)$ an oscillating term proportional to $\overline{\zeta_j}^{\,n}$, with a coefficient of modulus $|\Xi|$ independent of its width, and these $J$ terms cannot cancel on average over dyadic blocks, because the points $\zeta_j$ are distinct. Each dyadic block $2^m\leq n<2^{m+1}$ then contributes at least a constant times $m^{-2\sigma_+}$, $\sigma_+:=\max(\sigma,0)$, to $\|\Phi_1(T)\|_{\Hlog}^2$, so that this norm is infinite when $\sigma\leq\frac12$. This is the argument of the proof below.
\end{remark}

\begin{proof}[Proof of Theorem \ref{thmendpoint}, item (2)]
Assume $\Xi\neq0$, and choose $T\geq T_*$ such that $C_6\lambda(T)\ell(T)^2\leq\frac{|\Xi|}{2}\sqrt J$. Let $\Xi_n:=\Xi_n(T)$ be given by \eqref{defSn}. Since $\omega$ is real, $|\Xi_n|=|\Xi||S_n|$, and $\limsup_n|n^2\hat\Phi_1(T,n)-\Xi_n|\leq\frac{|\Xi|}2\sqrt J$ by \eqref{tail}. The numbers $e^{i\vartheta}|B'(\zeta_j)|^{-i\omega}\overline{\zeta_j}$ have modulus one, and for $j\neq l$ and $n_1\geq1$, $\big|\sum_{n=n_1}^{2n_1-1}(\overline{\zeta_j}\zeta_l)^n\big|\leq\frac{2}{|1-\overline{\zeta_j}\zeta_l|}$. Expanding $|S_n|^2$, we obtain $\frac1{n_1}\sum_{n=n_1}^{2n_1-1}|S_n|^2\to J$ as $n_1\to+\infty$. Since $|\Xi_n|^2\leq2n^4|\hat\Phi_1(T,n)|^2+2|n^2\hat\Phi_1(T,n)-\Xi_n|^2$, we get
$$\liminf_{n_1\to+\infty}\frac1{n_1}\sum_{n=n_1}^{2n_1-1}n^4|\hat\Phi_1(T,n)|^2\geq\frac{J|\Xi|^2}{2}-\frac{J|\Xi|^2}{4}>0.$$
Let $\sigma\leq\frac12$ and $\sigma_+:=\max(\sigma,0)$. For $n_1=2^m$ and $m$ large, since $n^3\geq\frac{n^4}{2^{m+1}}$ for $n<2^{m+1}$,
\begin{align*}
\sum_{n=2^m}^{2^{m+1}-1}\frac{\la n\ra^3|\hat\Phi_1(T,n)|^2}{\log^{2\sigma}(e+n)}&\geq\frac{1}{2\log^{2\sigma_+}(e+2^{m+1})}\,\frac1{2^m}\sum_{n=2^m}^{2^{m+1}-1}n^4|\hat\Phi_1(T,n)|^2\\
&\gtrsim m^{-2\sigma_+},
\end{align*}
and the sum over $m$ diverges. Hence $\Phi_1(T)\notin\Hlog$.

We now go back to $u$. By \eqref{defTt}, $f=\Phi_1-\frac12\mathbf T_{\overline{\tilde A^2-\tilde c^2}}\D^{-1}\Phi_2$, and the second term belongs to $H^2\subset\Hlog$, since $\Phi_2(T)\in H^1$ and $\D^{-1}$ gains one derivative. Hence $f(T)\notin\Hlog$. Assume by contradiction that $u(T)\in\Hlog$. Then $v(T)\in\Hlog$ by Lemma \ref{lemmacomoving} and \eqref{comparable}, hence $\Pi_+\eps(T)=\Pi_+v(T)-A(T)\in\Hlog$ and $z\Theta^2f=\Pi_+\eps(T)-\iota_T(a)\in\Hlog$, since $A(T)$ and $\iota_T(a)$ are smooth. By Lemma \ref{lemmapersistence}, item (1), $f=\Pi_+\big(\overline{z\Theta^2}\,z\Theta^2f\big)\in\Hlog$, which is a contradiction. Therefore $u(T)\notin\Hlog$, and by Lemma \ref{lemmapersistence}, item (2), $u(t)\notin\Hlog$ for every $t\in\R$. Since $H^{\frac32,0}_{\log}=H^{\frac32}$, in particular $u(t)\notin H^{\frac32}$. Finally, Lemma \ref{lemmaxi} gives the statements on the set of parameters with $\Xi\neq0$.
\end{proof}


\section{Concentration of the kinetic energy}
\label{sectiondefect}


In this section we prove Theorem \ref{thmdefect}, and we describe each peak at its own scale. We work in the frame moving with unit speed, in which the solution converges to a plane wave. By \eqref{comoving}, the function $V$ of Theorem \ref{thmdefect} and its limit are
\begin{equation}
\label{defV}
V(t,x):=e^{i(M+2K)t}u(t,x+t)=e^{-iKx}v(t,Nx),\qquad V_\infty(x):=c_\infty e^{-iKx}.
\end{equation}
We decompose $V$ into the plane wave $V_\infty$, the peaks $S$, and a remainder $R$. The peaks are bounded in $H^{\frac12}$ but they tend to zero in $L^2$, so that their kinetic energy concentrates, while the remainder is negligible at the energy regularity $H^{\frac12}$.

\begin{lemma}[Decomposition of $V$]
\label{lemmadecompV}
For $t\geq0$, let
\begin{equation}
\label{decompV}
S(t,x):=\sqrt Q\,e^{-iKx}F\big(t,e^{iNx}\big),\qquad R(t,x):=e^{-iKx}\big(c(t)-c_\infty+\eps(t,Nx)\big).
\end{equation}
Then $V=V_\infty+S+R$, and the following holds.
\begin{enumerate}
\item {\rm (}Peaks{\rm ).} The Fourier coefficients of $S(t)$ vanish outside the set of the nonnegative integers $Nn-K$, $n\geq1$, and
\begin{gather*}
\|S(t)\|_{L^2}^2=Q\lambda(t),\qquad \|S(t)\|_{\dot H^{\frac12}}^2=Q\big(NJ-K\lambda(t)\big),\\
\|S(t)\|_{L^\infty}=\sqrt Q\,\big(1+|p(t)|\big).
\end{gather*}
Moreover, in the variable $y$, with $D:=-i\pa_y$, we have $\|\overline FDF\|_{L^1}+\|DF\|_{L^1}\lesssim1$ for $t\geq T_0$, and, as $t\to+\infty$, weakly in the sense of measures,
\begin{equation}
\label{peakmomentum}
\Re\big(\overline FDF\big)\frac{dy}{2\pi}\rightharpoonup\sum_{j=1}^J\delta_{\arg\zeta_j}.
\end{equation}
\item {\rm (}Remainder{\rm ).} For $t\geq0$,
\begin{equation}
\label{boundR}
\|R(t)\|_\W\lesssim\lambda\ell,\qquad \|R(t)\|_{H^r}\lesssim_r\lambda^{\frac32-r}\ell,\qquad \frac12<r<\frac32.
\end{equation}
\end{enumerate}
\end{lemma}

\begin{proof}
The decomposition follows from \eqref{defV}, $v=A+\eps$ and $A=c+\sqrt QF$. Since $F(t,0)=0$, the coefficients of $S$ are $\sqrt Q\hat F(n)$, placed at the frequencies $Nn-K\geq0$, $n\geq1$. By \eqref{quadraticF}, $\|S\|_{L^2}^2=Q\|F\|_{L^2}^2=Q\lambda$. Since $\hat A(n)=\sqrt Q\hat F(n)$ for $n\geq1$ and $\hat A(0)=c$, \eqref{comovingfourier} gives $P(U)=-K|c|^2+Q\sum_{n\geq1}(Nn-K)|\hat F(n)|^2$, and \eqref{invariantsU} and \eqref{modc} give
$$\|S\|_{\dot H^{\frac12}}^2=Q\sum_{n\geq1}(Nn-K)|\hat F(n)|^2=P(U)+K|c|^2=Q\big(NJ-K\lambda\big).$$
Finally, $\|S\|_{L^\infty}=\sqrt Q\|F\|_{L^\infty}=\sqrt Q(1+|p|)$ by Proposition \ref{propalgebra}, item (1).

For the momentum density, we work in the variable $y:=Nx$, with $D:=-i\pa_y$, so that $DG_j=z\pa_zG_j=b_j\lambda_j\frac{z}{(1-p_jz)^2}$ on $\T$. Since $\overline{G_j}=\overline{b_j}\lambda_j\frac{\bar z}{1-\overline{p_jz}}$ on $\T$, we have, using the Parseval identity in the second line,
\begin{gather*}
\overline{G_j}DG_j=\frac{|b_j|^2\lambda_j^2}{|1-p_jz|^2(1-p_jz)},\quad \Re\big(\overline{G_j}DG_j\big)=\frac{|b_j|^2\lambda_j^2\big(1-\Re(p_jz)\big)}{|1-p_jz|^4}\geq0,\\
\int_\T\overline{G_j}DG_j\frac{dy}{2\pi}=\sum_{n\geq1}n|\hat G_j(n)|^2=|b_j|^2\lambda_j^2\sum_{n\geq1}n|p_j|^{2(n-1)}=|b_j|^2\to1.
\end{gather*}
If $|z-\zeta_j|\geq r_0>0$, then $|1-p_jz|\gtrsim r_0$ for $t$ large, so that $\Re(\overline{G_j}DG_j)=O(\lambda^2r_0^{-4})$ there. Hence $\Re(\overline{G_j}DG_j)\frac{dy}{2\pi}\rightharpoonup\delta_{\arg\zeta_j}$. For every $l$, we have $\|DG_l\|_{L^1}=|b_l|\lambda_l\int_\T|1-p_lz|^{-2}\frac{dy}{2\pi}=|b_l|$ and $\|G_l\|_{L^\infty}=\frac{|b_l|\lambda_l}{1-|p_l|}\leq2|b_l|$, so that, by \eqref{partialfractions} and $|b_l|\leq2$, $\|DF\|_{L^1}+\|\overline FDF\|_{L^1}\leq(1+\|F\|_{L^\infty})\|DF\|_{L^1}\lesssim1$ for $t\geq T_0$. For $l\neq j$, we have $|G_j|\lesssim\lambda$ near $\zeta_l$, and $|DG_l|\lesssim\lambda$ away from $\zeta_l$, while $\|G_j\|_{L^1}\leq\|G_j\|_{L^2}\lesssim\lambda^{\frac12}$. Therefore $\|\overline{G_j}DG_l\|_{L^1}\to0$, which gives \eqref{peakmomentum} and completes the proof of item (1). For item (2), $|c-c_\infty|=\sqrt Q|p-e^{-i\vartheta}|\lesssim\lambda$ by \eqref{defc} and Lemma \ref{lemmaplambda}, item (3), and by Lemma \ref{lemmacomoving}, the coefficients of $e^{-iKx}\eps(t,Nx)$ are the coefficients of $u(t)-U(t)$, multiplied by unimodular factors. Hence \eqref{boundR} follows from \eqref{approxmain}.
\end{proof}


\subsection{\texorpdfstring{Proof of Theorem \ref{thmdefect}}{Proof of Theorem 1.4}}


The kinetic energy of the peaks is computed through their momentum density $\Re(\overline SDS)$, which is local in $x$, and the commutator of $|D|^{\frac12}$ with a smooth function, which is a bounded operator, shows that the kinetic density $\big||D|^{\frac12}S\big|^2$ has the same limit. The points $x_{j,k}$ are defined by \eqref{defxjk}.

\begin{proof}[Proof of Theorem \ref{thmdefect}]
\noindent{\bf Step 1:} {\it Norms.} By Lemma \ref{lemmadecompV}, $\|V-V_\infty\|_{L^2}\leq\|S\|_{L^2}+\|R\|_{L^2}\to0$, $\|S\|_{\dot H^{\frac12}}^2\to NJQ$, and $\|R\|_{H^{\frac12}}\leq\|R\|_{H^1}\to0$. Hence $\|V-V_\infty\|_{\dot H^{\frac12}}^2\to NJQ$. Finally, $V(t)$ is bounded in $H^{\frac12}$ and converges to $V_\infty$ in $L^2$, hence weakly in $H^{\frac12}$. This proves item (1).

\noindent{\bf Step 2:} {\it Momentum density of the peaks.} We come back to the variable $x$, with $D:=-i\pa_x$. By \eqref{decompV}, $DS(t,x)=\sqrt Q\,e^{-iKx}\big(-KF+N\,DF\big)(e^{iNx})$, hence $\overline SDS=Q\big(-K|F|^2+N\,\overline FDF\big)(e^{iNx})$. For a continuous function $\psi$ on $\T$, the change of variables $y:=Nx$ on each of the $N$ arcs of length $\frac{2\pi}N$ gives
$$\int_\T \psi(x)g(Nx)\frac{dx}{2\pi}=\int_\T \psi_N(y)g(y)\frac{dy}{2\pi},\qquad \psi_N(y):=\frac1N\sum_{k=0}^{N-1}\psi\Big(\frac{y+2\pi k}{N}\Big).$$
The function $\psi_N$ is continuous on $\T$. Since $\||F|^2\|_{L^1}=\lambda\to0$, Lemma \ref{lemmadecompV}, item (1), gives
$$\int_\T \psi\,\Re\big(\overline SDS\big)\frac{dx}{2\pi}\to QN\sum_{j=1}^J\psi_N(\arg\zeta_j)=Q\sum_{j=1}^J\sum_{k=0}^{N-1}\psi(x_{j,k}).$$

\noindent{\bf Step 3:} {\it Kinetic density.} Since $S$ has nonnegative frequencies, $|D|S=DS$. Let $\psi$ be real and smooth. Then
\begin{align*}
\int_\T \psi\,\overline SDS\,\frac{dx}{2\pi}&=\big(|D|^{\frac12}S\,\big|\,|D|^{\frac12}(\psi S)\big)\\
&=\int_\T \psi\big||D|^{\frac12}S\big|^2\frac{dx}{2\pi}+\big(|D|^{\frac12}S\,\big|\,[|D|^{\frac12},\psi]S\big).
\end{align*}
In Fourier variables, the commutator $[|D|^{\frac12},\psi]$ has the kernel $(|n|^{\frac12}-|m|^{\frac12})\hat\psi(n-m)$, and $\big||n|^{\frac12}-|m|^{\frac12}\big|\leq|n-m|^{\frac12}$, so that it is bounded on $L^2$ with norm $\leq\sum_m|m|^{\frac12}|\hat\psi(m)|$. Hence the last term is $O(\|S\|_{\dot H^{\frac12}}\|S\|_{L^2})=O(\lambda^{\frac12})$. Taking real parts and using Step 2, we get $\int_\T \psi\,||D|^{\frac12}S|^2\frac{dx}{2\pi}\to Q\sum_{j,k}\psi(x_{j,k})$ for real smooth functions $\psi$, hence for all continuous functions $\psi$, by linearity and density, since $\||D|^{\frac12}S\|_{L^2}$ is bounded. Since $\||D|^{\frac12}R\|_{L^2}\to0$ and $|D|^{\frac12}S$ is bounded in $L^2$, the same limit holds for $\big||D|^{\frac12}(S+R)\big|^2=\big||D|^{\frac12}(V-V_\infty)\big|^2$, which is the second limit of item (2). Next, $|D|^{\frac12}V_\infty=\sqrt Kc_\infty e^{-iKx}$ has the constant density $KM$, and the cross term $2\Re\big(\overline{|D|^{\frac12}V_\infty}\,|D|^{\frac12}(S+R)\big)$ tends to zero weakly, since $|D|^{\frac12}S$ is bounded in $L^2$ and each of its Fourier coefficients, $\sqrt Q(Nn-K)^{\frac12}\hat F(n)=O(\lambda n^{\frac12})$, tends to zero, so that $|D|^{\frac12}S\rightharpoonup0$ in $L^2$. This proves the first limit of item (2).

\noindent{\bf Step 4:} {\it Density of the defect measures.} Let $u$ be a solution of Theorem \ref{thmmain} with $(M(u),P(u),E(u))=(2,0,3)$. By \eqref{invariantsmain}, $E(u)-\frac{P(u)}2-\frac{M(u)^2}4=KM$ and $P(u)+KM=NJQ$, so that $K=1$ and $NJQ=2$. Conversely, Corollary \ref{corsameshell} provides, for every $J\geq1$ and every family of distinct points $\zeta_1,\dots,\zeta_J\in\T$, a solution with $N=1$, $Q=\frac2J$ and the concentration points $x_{j,0}=\arg\zeta_j$. Its defect measure is $\nu_J:=\frac2J\sum_{j=1}^J\delta_{\arg\zeta_j}$. Let $\nu$ be a positive Borel measure on $\T$ with $\nu(\T)=2$, and let $d\geq1$. We partition $\T$ into $d$ half-open arcs $I_1,\dots,I_d$ of length $\frac{2\pi}{d}$. For $J\geq1$, we choose integers $m_l\geq0$ with $\sum_lm_l=J$ and $|\frac{2m_l}{J}-\nu(I_l)|\leq\frac2J$, which is possible by rounding the numbers $\frac J2\nu(I_l)$, whose sum is $J$, to integers with the same sum, and we choose $m_l$ distinct points in $I_l$. Then, for every continuous function $\psi$ on $\T$,
\begin{align*}
\Big|\int \psi\,d\nu_J-\int \psi\,d\nu\Big|&\leq\sum_{l=1}^d\Big|\int_{I_l}\psi\,d\nu_J-\int_{I_l}\psi\,d\nu\Big|\\
&\leq4\sup_{|x-x'|\leq\frac{2\pi}d}|\psi(x)-\psi(x')|+\frac{2d}{J}\|\psi\|_{L^\infty},
\end{align*}
where we compared both integrals over $I_l$ with $\psi(x_l)\nu_J(I_l)$ and $\psi(x_l)\nu(I_l)$ for a point $x_l\in I_l$. Choosing $d$ large and then $J$ large proves item (3).
\end{proof}

The diffuse measures are thus obtained as weak limits of defect measures of different solutions, on the same level set of the conservation laws.

\begin{remark}
\label{remarkdensities}
The same decomposition gives the limits of the densities of momentum and energy. Weakly in the sense of measures,
\begin{align*}
\Re\big(\overline VDV\big)\frac{dx}{2\pi}&\rightharpoonup-KM\frac{dx}{2\pi}+Q\sum_{j=1}^J\sum_{k=0}^{N-1}\delta_{x_{j,k}},\\
\Big(\frac12\big||D|^{\frac12}V\big|^2+\frac14|V|^4\Big)\frac{dx}{2\pi}&\rightharpoonup\Big(\frac{KM}{2}+\frac{M^2}{4}\Big)\frac{dx}{2\pi}+\frac Q2\sum_{j=1}^J\sum_{k=0}^{N-1}\delta_{x_{j,k}}.
\end{align*}
Indeed, $\overline{V_\infty}DV_\infty=-KM$; the term $\overline SDV_\infty=-K\overline SV_\infty$ tends to zero in $L^1$, and the term $\overline{V_\infty}DS$ tends to zero weakly, after an integration by parts against smooth functions, since $\|S\|_{L^1}\leq\|S\|_{L^2}\to0$ and $\|DS\|_{L^1}\lesssim1$ by Lemma \ref{lemmadecompV}, item (1); the terms containing $R$, namely $\overline VDR$ and $\overline RD(V_\infty+S)$, tend to zero in $L^1$, since $\|\overline VDR\|_{L^1}\leq\|V\|_{L^2}\|R\|_{H^1}$ and $\|\overline RD(V_\infty+S)\|_{L^1}\leq\|R\|_{L^\infty}\|D(V_\infty+S)\|_{L^1}$, with $\|D(V_\infty+S)\|_{L^1}\lesssim1$ by Lemma \ref{lemmadecompV}, item (1); and $|V|^4\to M^2$ in $L^1$, since $|V|^2-M=2\Re\big(\overline{V_\infty}(S+R)\big)+|S+R|^2\to0$ in $L^1$ and $V$ is bounded in $L^\infty$. The total masses of the three limits are the limit $NJQ+KM$ of the kinetic energy $\||D|^{\frac12}u(t)\|_{L^2}^2=\||D|^{\frac12}V(t)\|_{L^2}^2$, the momentum $P(u)=NJQ-KM$ and the energy $E(u)=\frac{NJQ+KM}{2}+\frac{M^2}{4}$ of \eqref{invariantsmain}. The plane wave carries the uniformly distributed part, and each of the $NJ$ concentration points carries the kinetic energy $Q$, independently of its width. Below the energy regularity, by Lemma \ref{lemmanorms}, item (2), and \eqref{boundR}, $\|V(t)-V_\infty\|_{H^s}^2=N^{2s}Q\,\Gamma(2s+1)\sum_j|B'(\zeta_j)|^{2s-1}\lambda(t)^{1-2s}(1+o(1))$ for $0\leq s<\frac12$.
\end{remark}


\subsection{The profile of the peaks}


We now look at the peaks at their own scale. Let $y_j(t)$ be the continuous determination of the argument of $\overline{p_j(t)}$ such that $y_j(t)\to\arg\zeta_j$, for a fixed determination of $\arg\zeta_j$, and let
\begin{equation}
\label{defcenters}
\tilde x_{j,k}(t):=\frac{y_j(t)+2\pi k}{N},\quad \eta_j(t):=1-|p_j(t)|,\qquad 1\leq j\leq J,\ \ 0\leq k\leq N-1.
\end{equation}
Then $\tilde x_{j,k}(t)\to x_{j,k}$ and, by \eqref{lambdaj}, $\eta_j(t)=\frac{\lambda(t)}{2|B'(\zeta_j)|}(1+o(1))$. The point $\tilde x_{j,k}(t)$ is the center of a peak in the variable $x$, and $\frac{\eta_j(t)}N$ is its half-width: by Proposition \ref{propspikes} below, $|V-V_\infty|^2$ is close to $\frac{4Q}{1+\xi^2}$ at the point $\tilde x_{j,k}(t)+\frac{\eta_j}{N}\xi$, so that $\frac{\eta_j}{N}$ is the half-width at half-maximum of $|V-V_\infty|^2$ near $\tilde x_{j,k}(t)$.

\begin{proposition}[Profile of the peaks]
\label{propspikes}
As $t\to+\infty$, the following holds.
\begin{enumerate}
\item {\rm (}Profiles{\rm ).} Uniformly for $\xi$ in compact subsets of $\R$,
\begin{gather*}
e^{iK\tilde x_{j,k}(t)}V\Big(t,\tilde x_{j,k}(t)+\frac{\eta_j}N\xi\Big)\to c_\infty+\frac{2\sqrt Q\,e^{i\vartheta}}{1-i\xi},\\
\sqrt{\frac{\eta_j}N}\,e^{iK\tilde x_{j,k}(t)}\big(|D|^{\frac12}V\big)\Big(t,\tilde x_{j,k}(t)+\frac{\eta_j}N\xi\Big)\to\frac{\sqrt{\pi Q}\,e^{i\vartheta}}{(1-i\xi)^{\frac32}},
\end{gather*}
where the power $(1-i\xi)^{\frac32}$ is defined with the principal branch of the logarithm.
\item {\rm (}Extremal values{\rm ).} $\|V(t)-V_\infty\|_{L^\infty}\to2\sqrt Q$, $\|V(t)\|_{L^\infty}\to\sqrt{2K}+\sqrt Q$ and $\min_{x\in\T}|V(t,x)|\to|\sqrt{2K}-\sqrt Q|$.
\end{enumerate}
\end{proposition}

The rescaled profile $\frac{2\sqrt Q\,e^{i\vartheta}}{1-i\xi}$ of a peak depends neither on $j$ and $k$ nor on $B$. As $\xi$ runs over $\R$, the quantity $\frac{2}{1-i\xi}-1=\frac{1+i\xi}{1-i\xi}$ runs once over the unit circle, so that a peak is a full rotation of the phase of the unimodular part $\Theta$ of the profile, half of which occurs for $|\xi|\leq1$, that is, within one half-width $\frac{\eta_j}N$ of the center. Near $\tilde x_{j,k}(t)$, the limiting density of kinetic energy is $\frac{1}{2\pi}\frac{\pi QN}{\eta_j}(1+\xi^2)^{-\frac32}dx=\frac Q2(1+\xi^2)^{-\frac32}d\xi$, since $dx=\frac{\eta_j}{N}d\xi$; its integral over $\xi\in\R$ is $Q$, in agreement with Theorem \ref{thmdefect}.

\begin{proof}
\noindent{\bf Step 1:} {\it Profile of $V$.} Let $x=\tilde x_{j,k}(t)+\frac{\eta_j}N\xi$, so that $Nx=y_j+2\pi k+\eta_j\xi$. Since $p_j=|p_j|e^{-iy_j}$, we have $p_je^{iNx}=(1-\eta_j)e^{i\eta_j\xi}$ and $1-p_je^{iNx}=\eta_j(1-i\xi)+O(\eta_j^2(1+\xi^2))$. Since $\lambda_j=\eta_j(2-\eta_j)$ and $b_je^{iy_j}\to e^{i\vartheta}\overline{\zeta_j}\zeta_j=e^{i\vartheta}$,
$$G_j\big(e^{iNx}\big)=b_j\lambda_j\frac{e^{iNx}}{1-p_je^{iNx}}\longrightarrow\frac{2e^{i\vartheta}}{1-i\xi},$$
uniformly for $\xi$ in compact sets, while $G_l(e^{iNx})=O(\lambda)$ for $l\neq j$. Since $e^{iK(\tilde x_{j,k}(t)-x)}\to1$, $c\to c_\infty$ and $\|R\|_{L^\infty}\leq\|R\|_\W\to0$, this proves the first convergence of item (1).

\noindent{\bf Step 2:} {\it Profile of $|D|^{\frac12}V$.} By \eqref{boundR} with $r=\frac54$, $\||D|^{\frac12}R\|_{L^\infty}\leq\sum_m|m|^{\frac12}|\hat R(m)|\lesssim\|R\|_{H^{\frac54}}\lesssim\lambda^{\frac14}\ell$, so that $\sqrt{\eta_j/N}\,|D|^{\frac12}R\to0$ uniformly; also $|D|^{\frac12}V_\infty$ is bounded. We have $|D|^{\frac12}S=\sqrt Q\sum_{n\geq1}(Nn-K)^{\frac12}\hat F(n)e^{i(Nn-K)x}$ with $\hat F(n)=\sum_lb_l\lambda_lp_l^{n-1}$. At the point $x=\tilde x_{j,k}(t)+\frac{\eta_j}N\xi$, $p_j^{n-1}e^{iNnx}=|p_j|^{n-1}e^{iy_j}e^{in\eta_j\xi}$, and the term $l=j$ equals
$$\sqrt Q\,b_j\lambda_je^{iy_j}e^{-iKx}\sqrt N\sum_{n\geq1}\Big(n-\frac KN\Big)^{\frac12}(1-\eta_j)^{n-1}e^{in\eta_j\xi}.$$
By comparison with a Riemann integral, and by dominated convergence with the bound $(1-\eta_j)^{n-1}\leq e^{-\eta_j(n-1)}$,
$$\eta_j^{\frac32}\sum_{n\geq1}\Big(n-\frac KN\Big)^{\frac12}(1-\eta_j)^{n-1}e^{in\eta_j\xi}\to\int_0^{+\infty}\rho^{\frac12}e^{-\rho(1-i\xi)}\,d\rho=\Gamma\Big(\frac32\Big)(1-i\xi)^{-\frac32},$$
uniformly for $\xi$ in compact sets. Since $\lambda_j\eta_j^{-\frac32}\sqrt N\sqrt{\eta_j/N}=2-\eta_j$ and $2\Gamma(\frac32)=\sqrt\pi$, the term $l=j$, multiplied by $\sqrt{\eta_j/N}\,e^{iK\tilde x_{j,k}(t)}$, converges to $\sqrt{\pi Q}\,e^{i\vartheta}(1-i\xi)^{-\frac32}$. For $l\neq j$, the phase of $p_le^{iNx}$ stays away from $0$, so that the partial sums of $(p_le^{iNx}/|p_l|)^n$ are bounded; since the sequence $(Nn-K)^{\frac12}|p_l|^{n-1}$ is unimodal, with maximum $\lesssim\lambda_l^{-\frac12}$, a summation by parts shows that the term $l$ is $O(\lambda_l\lambda_l^{-\frac12})=O(\lambda^{\frac12})$. This proves the second convergence of item (1).

\noindent{\bf Step 3:} {\it Extremal values.} By Lemma \ref{lemmadecompV}, $\|S\|_{L^\infty}=\sqrt Q(1+|p|)\to2\sqrt Q$. Since $|A|=|\sqrt{2K}+\sqrt Q\,\Theta|$ and $\Theta(t)$ maps $\T$ onto $\T$, the bound $|A|\leq\sqrt{2K}+\sqrt Q$ of Proposition \ref{propalgebra}, item (1), is attained, and $\min_\T|A|=|\sqrt{2K}-\sqrt Q|$. Since $V-V_\infty=S+R$ with $\|R\|_{L^\infty}\to0$, and since $|V(t,x)|=|A(t,e^{iNx})+\eps(t,Nx)|$ with $\|\eps\|_{L^\infty}\leq\|\eps\|_\W\to0$, item (2) follows.
\end{proof}

\appendix

\section{Finite Blaschke products}
\label{appendixblaschke}


In this appendix, $B$ denotes a finite Blaschke product of degree $d\geq1$,
\begin{equation}
\label{generalblaschke}
B(z):=e^{i\gamma_0}\prod_{m=1}^dB_m(z),\qquad B_m(z):=\frac{z-\beta_m}{1-\overline{\beta_m}z},\qquad \gamma_0\in\R,\ \ \beta_m\in\Disc.
\end{equation}
We recall the elementary properties of $B$ which are used in the paper; we refer to \cite[Chapters 2 and 5]{GMR2016} for a general presentation, and to \cite{Cl1972} and \cite[Chapter 11]{GMR2016} for the Clark measures, which underlie item (5) below. The Blaschke products $\Theta(t)$ of \eqref{defTheta} are of this form, with $d=J$.

\begin{lemma}[Finite Blaschke products]
\label{lemmablaschke}
\begin{enumerate}
\item {\rm (}Boundary values{\rm ).} For $\zeta\in\T$,
\begin{equation}
\label{boundaryB}
\frac{\zeta B'(\zeta)}{B(\zeta)}=|B'(\zeta)|=\sum_{m=1}^d\frac{1-|\beta_m|^2}{|\zeta-\beta_m|^2}>0.
\end{equation}
Consequently, the argument of $B(e^{iy})$ is an increasing function of $y$, with derivative $|B'(e^{iy})|$ and total increase $2\pi d$ over $[0,2\pi]$, and for every $w\in\T$ the equation $B(\zeta)=w$ has exactly $d$ solutions $\zeta\in\T$.
\item {\rm (}Reflection{\rm ).} $B(1/\bar z)=1/\overline{B(z)}$ for every $z\in\C\setminus\{0\}$ which is neither a zero nor a pole of $B$; in particular, $B$ is infinite at infinity if $B(0)=0$. Moreover $|B|<1$ in $\Disc$, $|B|>1$ outside $\overline\Disc$, and for $w\in\Disc$ the equation $B(z)=w$ has exactly $d$ solutions in $\Disc$, counted with multiplicity.
\item {\rm (}The function $\varrho${\rm ).} $1-|B(z)|^2=(1-|z|^2)\varrho(z)$, where
$$\varrho(z):=\sum_{m=1}^d\frac{1-|\beta_m|^2}{|1-\overline{\beta_m}z|^2}\prod_{l<m}|B_l(z)|^2$$
is smooth and positive in a neighborhood of $\overline\Disc$, and $\varrho(\zeta)=|B'(\zeta)|$ for $\zeta\in\T$.
\item {\rm (}Dirichlet integral{\rm ).} $\sum_{n\geq0}n|\hat B(n)|^2=d$.
\item {\rm (}Clark identity{\rm ).} If $B(0)=0$ and $w\in\T$, the $d$ solutions $\zeta_1,\dots,\zeta_d$ of $B(\zeta)=w$ satisfy
\begin{equation}
\label{clark}
\sum_{j=1}^d\frac{1}{|B'(\zeta_j)|}=1.
\end{equation}
\end{enumerate}
\end{lemma}

\begin{proof}
\noindent{\bf Step 1:} {\it Boundary values.} For $\zeta\in\T$ and $\beta\in\Disc$, we have $1-\bar\beta\zeta=\zeta\,\overline{(\zeta-\beta)}$, hence
$$\frac{\zeta B_m'(\zeta)}{B_m(\zeta)}=\frac{\zeta}{\zeta-\beta_m}+\frac{\overline{\beta_m}\zeta}{1-\overline{\beta_m}\zeta}=\frac{\zeta(1-|\beta_m|^2)}{(\zeta-\beta_m)(1-\overline{\beta_m}\zeta)}=\frac{1-|\beta_m|^2}{|\zeta-\beta_m|^2}.$$
Summing over $m$ gives the second expression of \eqref{boundaryB}, and since $|B(\zeta)|=|\zeta|=1$, its modulus is $|B'(\zeta)|$. Next, $\pa_y\log B(e^{iy})=i\frac{e^{iy}B'(e^{iy})}{B(e^{iy})}=i|B'(e^{iy})|$, and the integral of the Poisson kernel $\frac{1-|\beta|^2}{|e^{iy}-\beta|^2}$ over $[0,2\pi]$ is $2\pi$. Hence the argument of $B(e^{iy})$ increases by $2\pi d$, with a positive derivative, and it takes every value modulo $2\pi$ exactly $d$ times.

\noindent{\bf Step 2:} {\it Reflection.} For each factor, $B_m(1/\bar z)=\frac{1-\beta_m\bar z}{\bar z-\overline{\beta_m}}=1/\overline{B_m(z)}$. Since $|z-\beta|^2-|1-\bar\beta z|^2=-(1-|z|^2)(1-|\beta|^2)$, we have $|B_m|<1$ in $\Disc$, and $|B_m|>1$ outside $\overline\Disc$ by reflection. For $w\in\Disc$, the equation $B(z)=w$ is equivalent to $e^{i\gamma_0}\prod_m(z-\beta_m)=w\prod_m(1-\overline{\beta_m}z)$, a polynomial equation of degree $d$, since $|w\prod_m\overline{\beta_m}|<1$. Its $d$ roots lie in $\Disc$, since $|B|\geq1$ outside $\Disc$.

\noindent{\bf Step 3:} {\it The function $\varrho$.} Since $|B|=\prod_m|B_m|$ and $1-\prod_{m=1}^dx_m=\sum_{m=1}^d(1-x_m)\prod_{l<m}x_l$, we have $1-|B|^2=\sum_m(1-|B_m|^2)\prod_{l<m}|B_l|^2$, and $1-|B_m(z)|^2=\frac{(1-|z|^2)(1-|\beta_m|^2)}{|1-\overline{\beta_m}z|^2}$ by Step 2. The function $\varrho$ is smooth for $|z|<\min_m|\beta_m|^{-1}$, its first term is positive and the others are nonnegative. On $\T$, $|B_l|=1$ and $|1-\overline{\beta_m}\zeta|=|\zeta-\beta_m|$, and \eqref{boundaryB} gives $\varrho(\zeta)=|B'(\zeta)|$.

\noindent{\bf Step 4:} {\it Dirichlet integral.} By the Parseval identity, $\sum_nn|\hat B(n)|^2=\int_\T\overline{B}\,\zeta B'(\zeta)\frac{dy}{2\pi}$, with $\zeta:=e^{iy}$. Since $|B|=1$ on $\T$, the integrand is $\frac{\zeta B'(\zeta)}{B(\zeta)}=|B'(\zeta)|$, whose mean is $d$ by Step 1.

\noindent{\bf Step 5:} {\it Clark identity.} The rational function $H:=\frac{w+B}{w-B}$ satisfies $\Re H=\frac{1-|B|^2}{|w-B|^2}$, so that $\Re H>0$ in $\Disc$ and $\Re H=0$ on $\T$ outside the points $\zeta_j$. It is bounded at infinity, since $B(\infty)=\infty$ by Step 2, and $H(0)=1$ since $B(0)=0$. The poles of $H$ in $\C$ are the solutions of $B(z)=w$, since $H=-1$ at the poles of $B$; by Step 2, these solutions lie on $\T$, so that they are the $d$ points $\zeta_j$ of Step 1, and they are simple poles of $H$ since $B'(\zeta_j)\neq0$. Moreover $\zeta_jB'(\zeta_j)=|B'(\zeta_j)|w$ by \eqref{boundaryB}, so that the residue of $H$ at $\zeta_j$ is $-\frac{2w}{B'(\zeta_j)}=-\frac{2\zeta_j}{|B'(\zeta_j)|}$. The residue of $\frac{\zeta_j+z}{\zeta_j-z}$ at $\zeta_j$ is $-2\zeta_j$. Hence the function $H(z)-\sum_j|B'(\zeta_j)|^{-1}\frac{\zeta_j+z}{\zeta_j-z}$ is rational without poles in $\C$ and bounded at infinity, hence constant. Its real part vanishes on $\T$ outside the points $\zeta_j$, so that the constant is purely imaginary. Taking the real part at $z=0$ gives $1-\sum_j|B'(\zeta_j)|^{-1}=0$.
\end{proof}


\section{Conservation laws and growth rates of the explicit family}
\label{appendixinvariants}


In this appendix we describe the restrictions which the condition \eqref{admissible} imposes on the conservation laws \eqref{invariantsmain} of the solutions of Theorem \ref{thmmain}, and the largest possible growth rates. These quantities only depend on $K$, $M$, $Q$ and on the total number $NJ$ of concentration points, with $NJ\geq N\geq K$. Recall that
\begin{equation}
\label{invparameters}
\Lambda^2=8KQ-(2K+Q-M)^2,\quad P=NJQ-KM,\quad E=\frac{NJQ+KM}{2}+\frac{M^2}{4},
\end{equation}
and that the $H^s$ norm grows at the rate $(s-\frac12)\Lambda$, $\frac12<s<\frac32$. The second condition in \eqref{admissible}, that is $\Lambda^2>0$, is equivalent to
\begin{equation}
\label{qinterval}
q_-(M,K)<Q<q_+(M,K),\qquad q_\pm(M,K):=\big(\sqrt M\pm\sqrt{2K}\big)^2.
\end{equation}

\begin{proposition}[Energy floor]
\label{propfloor}
\begin{enumerate}
\item {\rm (}Fixed mass{\rm ).} For every $M>0$, the infimum of the energies of the solutions of Theorem \ref{thmmain} with mass $M$ is $e_\star(M):=\frac{M^2}{4}+M+1-\sqrt{2M}$, and it is not attained.
\item {\rm (}All masses{\rm ).} The set of energies of the solutions of Theorem \ref{thmmain} is the interval $(E_\star,+\infty)$, where $E_\star:=1-x_0+\frac{x_0^2}{2}+\frac{x_0^4}{16}=0.5438713443\dots$ and $x_0$ is the unique positive root of $x^3+4x-4=0$.
\end{enumerate}
In particular, if $u_m$, $m\geq1$, are solutions of Theorem \ref{thmmain} and $t_m\in\R$, then $u_m(t_m)$ does not converge to zero in $H^{\frac12}$ as $m\to+\infty$.
\end{proposition}

\begin{proof}
Since $NJ\geq K$ and $Q>q_-$, \eqref{invparameters} gives $E>\frac{Kq_-+KM}{2}+\frac{M^2}{4}=e_K(M)$, with $e_K(M):=\frac{M^2}{4}+KM+K^2-K\sqrt{2KM}$. As a function of the real variable $K>0$, $\pa_Ke_K(M)=M+2K-\frac32\sqrt{2KM}$, which is positive, since $1+2x^2-\frac{3\sqrt2}{2}x>0$ for every real $x$, and in particular for $x=\sqrt{K/M}$. Hence $E>e_1(M)=e_\star(M)$. The bound is approached with $K=N=J=1$ and $Q\downarrow q_-(M,1)$. The derivative of $e_\star$ vanishes if and only if $x:=\sqrt{2M}$ solves $x^3+4x-4=0$, which gives the minimum $E_\star$ of $e_\star$, at $M=\frac{x_0^2}{2}$. For $K=N=J=1$, the set of admissible pairs $(M,Q)$ is connected, by \eqref{qinterval}, and the energy $E=\frac{Q+M}{2}+\frac{M^2}4$ is continuous on it, with $\pa_QE=\frac12$. The set of energies is therefore an open interval, which is unbounded from above and has the infimum $E_\star$. Finally, the energy is continuous at $0$ in $H^{\frac12}$, by the Sobolev embedding $H^{\frac12}(\T)\subset L^4(\T)$, and it vanishes at $0$, while $E(u_m(t_m))=E(u_m)>E_\star>0$.
\end{proof}

\begin{proposition}[Optimal rate]
\label{proprate}
Let $C_\star:=4.3135109824\dots$ be the unique root in $(4,5)$ of $C^3-256C+1024=0$, and let $m_\star:=\frac{16}{C_\star}-2$ and $q_\star:=\frac{16}{C_\star}-\frac{C_\star}{4}$. Every solution of Theorem \ref{thmmain} satisfies
$$\Lambda^2\leq C_\star\,E,$$
with equality if and only if $N=K$, $J=1$, $M=Km_\star$ and $Q=Kq_\star$. Consequently, the largest ratio between the square of the exponential growth rate of the $H^s$ norm and the energy is $C_\star(s-\frac12)^2$.
\end{proposition}

\begin{proof}
For fixed $K$ and $NJ=K$, the quantity $C E-\Lambda^2$ is a quadratic polynomial in $(M,Q)$, by \eqref{invparameters}, and $C_\star$ is the smallest value of $C$ for which this polynomial is nonnegative; we write it as a sum of squares. Let $d_\star:=2-\frac{C_\star}{4}$, so that $q_\star=m_\star+d_\star$. Expanding both sides, one checks the identity
$$C_\star E-\Lambda^2=\big(Q-M-Kd_\star\big)^2+\frac{C_\star}{4}\big(M-Km_\star\big)^2+\frac{C_\star}{2}(NJ-K)Q.$$
Indeed, the left hand side equals $(Q-M)^2+\frac{C_\star}{4}M^2+4K^2+(\frac{C_\star}{2}-4)KM+(\frac{C_\star NJ}{2}-4K)Q$, and the coefficients agree because $-2d_\star=\frac{C_\star}{2}-4$, $2d_\star-\frac{C_\star}{2}m_\star=\frac{C_\star}{2}-4$, and $d_\star^2+\frac{C_\star}{4}m_\star^2=4$; the last identity is equivalent to $C_\star^3-256C_\star+1024=0$. All the terms on the right hand side are nonnegative, which proves the inequality. Equality holds if and only if $NJ=K$, $M=Km_\star$ and $Q=M+Kd_\star=Kq_\star$, and $NJ=K$ with $N\geq K$ means $N=K$ and $J=1$. These parameters are admissible, since $m_\star,q_\star>0$ and the identity gives $\Lambda^2=C_\star E>0$.
\end{proof}

\begin{proposition}[Number of concentration points on a level set]
\label{proppeaks}
Fix $(M,P,E)$ with $M>0$, and consider the solutions of Theorem \ref{thmmain} with these conservation laws.
\begin{enumerate}
\item {\rm (}Carrier and existence{\rm ).} These solutions have the carrier $e^{-iKx}$ with
$$K=\frac{1}{M}\Big(E-\frac P2-\frac{M^2}{4}\Big),$$
and they exist if and only if $K$ is a positive integer, $P+KM>0$, and there is an integer $n\geq K$ with $\frac{P+KM}{q_+(M,K)}<n<\frac{P+KM}{q_-(M,K)}$, where the upper bound is omitted if $q_-(M,K)=0$; every such $n$ is the number $NJ$ of concentration points of such a solution.
\item {\rm (}Unbounded number of concentration points{\rm ).} The number of concentration points is unbounded on the level set if and only if $M=2K$, $P>-2K^2$ and $E=\frac P2+3K^2$ for a positive integer $K$.
\item {\rm (}Rate{\rm ).} In the case of item (2), a solution with $n$ concentration points has the rate $\Lambda_n$ given by $\Lambda_n^2:=\frac{8K(P+KM)}{n}-\frac{(P+KM)^2}{n^2}$, so that $\Lambda_n\to0$ as $n\to+\infty$.
\end{enumerate}
\end{proposition}

\begin{proof}
By \eqref{invparameters}, $E-\frac P2=KM+\frac{M^2}{4}$ and $NJQ=P+KM$, which gives $K$ and $Q=\frac{P+KM}{n}$ with $n:=NJ$. By \eqref{qinterval}, the condition \eqref{admissible} becomes $\frac{P+KM}{q_+}<n<\frac{P+KM}{q_-}$. Conversely, an integer $n\geq K$ with this property is realized by $N=n$, $J=1$ and $Q=\frac{P+KM}n$; if $K=1$, it is also realized by $N=1$ and $J=n$, with the arbitrary geometry of Corollary \ref{corgeometry}. If $M\neq2K$, then $q_->0$ and $n$ is bounded. If $M=2K$, then $q_-=0$ and every large integer $n$ is allowed; the conditions $P+KM>0$ and $E-\frac P2=KM+\frac{M^2}4$ become $P>-2K^2$ and $E=\frac P2+3K^2$. Finally, \eqref{invparameters} with $Q=\frac{P+KM}n$ and $M=2K$ gives the formula for $\Lambda_n^2$.
\end{proof}

The level set $(M,P,E)=(2,0,3)$ of Corollary \ref{corsameshell} corresponds to $K=1$ and $P+KM=2$. On such level sets, arbitrarily many concentration points are compatible with the conservation laws, at the price of a vanishing growth rate within the explicit family.


\end{document}